\documentclass[11pt,reqno]{amsart}
\usepackage[T1]{fontenc}
\usepackage{graphicx}
\usepackage{cite}

\usepackage{color}
\definecolor{MyLinkColor}{rgb}{0,0,0.4}

\newcommand{\re}{\mathop{\rm Re}\nolimits}

\newcommand{\PV}{\mathop{\rm PV}\nolimits}
\newcommand{\id}{\mathop{\rm id}\nolimits}
\newcommand{\supp}{\mathop{\rm supp}\nolimits}
\newtheorem{thm}{Theorem}

\newtheorem{prop}{Proposition}
\newtheorem{lemma}{Lemma}

\theoremstyle{remark} 
\newtheorem{rem}{Remark}

\numberwithin{equation}{section} 

\title[Well-posedness of the  Muskat problem in subcritical $L_p$-Sobolev spaces]{Well-posedness of the  Muskat problem\\ in subcritical $L_p$-Sobolev spaces}

  \author{Helmut Abels}
\address{Fakult\"at f\"ur Mathematik, Universit\"at Regensburg,   93040 Regensburg, Deutschland.}
\email{helmut.abels@ur.de}
\email{bogdan.matioc@ur.de}
  
\author{Bogdan--Vasile Matioc}


\usepackage[colorlinks=true,linkcolor=MyLinkColor,citecolor=MyLinkColor]{hyperref} 

\begin{document}

\begin{abstract}
We study the Muskat problem describing the vertical motion of two immiscible fluids in a two-dimensional homogeneous porous medium in an $L_p$-setting with $p\in(1,\infty)$.
The Sobolev space $W^s_p(\mathbb{R})$ with $s=1+1/p$ is  a critical space  for this problem. 
We prove, for $s\in (1+1/p,2),$ that  the Rayleigh-Taylor condition identifies an open subset of $W^s_p(\mathbb{R})$ within which the Muskat problem is of parabolic type. 
This enables us to establish the local well-posedness of the problem in all these subcritical spaces together with a parabolic smoothing property.
\end{abstract}

\maketitle

\section{Introduction}\label {Sec:1}

In this paper we study the following system of nonlinear and nonlocal  equations  
\begin{subequations}\label{P}
\begin{equation}\label{P:1}
\begin{array}{rlll}
 \displaystyle\partial_ tf(t,x)\!\!&=&\!\! \displaystyle\frac{1}{\pi}\PV\int_\mathbb{R}\frac{y+\partial_ xf(t,x)(f(t,x)-f(t,x-y))}{ y^2+(f(t,x)-f(t,x-y))^2 }\overline{\omega}(t,x-y)\, dy,\\[3ex]
  \displaystyle -C_{\Theta} \partial_ xf(t,x)\!\!&=&\!\!\displaystyle \overline{\omega}(t,x) \\[1ex]
  &&\!\!+ \displaystyle\frac{a_\mu}{\pi  } \PV\int_\mathbb{R}\frac{ y\partial_ xf(t,x)-(f(t,x)-f(t,x-y)) }{ y^2+(f(t,x)-f(t,x-y))^2 }\overline{\omega}(t,x-y)\, dy,
\end{array}
\end{equation}
for $t\geq0$ and $x\in\mathbb{R}$, that describes the motion of two immiscible Newtonian fluids   with   viscosities $\mu_-$ and $\mu_+$  and    densities $\rho_-$ and $\rho_+$
 in a vertical two-dimensional  porous medium with 
constant permeability $k$ that we identify with $\mathbb{R}^2$. The fluid located below is denoted by~$-$.
The unknown function $f$ parameterizes the sharp interface between the fluids and~${2(1+(\partial_ xf)^2)^{-1/2}\overline{\omega}}$
 is the jump of the velocity field  in tangential direction at the interface, cf. \cite[Eq. (2.6)]{MBV18}.
For the Muskat problem (\ref{P:1}) we consider the general scenario when
\[\mu_--\mu_+\in\mathbb{R}\qquad\mbox{and}\qquad \Theta:=g(\rho_--\rho_+)  +\frac{\mu_--\mu_+}{k}V\in\mathbb{R}.\]
The constant  $g$ is the Earth's gravity, $|V|\in\mathbb{R}$ is the velocity at which the fluid system moves vertically upwards if $V>0$ or downwards if $V<0$,
 and 
\[
a_\mu:=\frac{\mu_--\mu_+}{\mu_-+\mu_+}\in(-1,1) \qquad\mbox{and}\qquad C_{\Theta}:=\frac{k\Theta}{\mu_-+\mu_+},
\]
where $a_\mu$ is called Atwood number.
 Moreover, $\PV$ denotes the principal value and is taken at zero and/or at infinity.
The system~(\ref{P:1}) is supplemented with the initial condition
\begin{equation}\label{P:1_IC}
\begin{array}{rlll}
f(0,\cdot)\!\!&=&\!\! f_0.
\end{array}
\end{equation}
\end{subequations}

\subsection*{Critical spaces for $(\ref{P})$} It can be verified, that if $f$ is a solution to (\ref{P:1}), then, given~${\lambda>0}$, the function $f_\lambda$ with
\[
f_\lambda(t,x):=\lambda^{-1}f(\lambda  t,\lambda x)
\]  
 also solves (\ref{P:1}). 
 Moreover, given $p\in(1,\infty) $ and $r\in(0,1)$, it holds that
 $$[\partial_ x^{(k)} f(\lambda t)]_{W^{r}_p}=[\partial_ x^{(k)} f_\lambda(t)]_{W^{r}_p}$$
exactly for $k=1$ and  $r=1/p$. 
This property identifies the  space $W^s_p(\mathbb{R})$ with $s=1+1/p$ as a critical space for (\ref{P:1}).
We recall  that, given $0<s\not\in\mathbb{N}$ with $s=[s]+\{s\}$, where~${[s]\in\mathbb{N}}$ and  $\{s\}\in(0,1)$, $W^s_p(\mathbb{R})$ is a Banach space with the norm
\[
\|f\|_{W^s_p}:=\big(\|f\|_{W^{[s]}_p}^p+[f]_{W^{s}_p}^p\big)^{1/p},
\]
where
\[
[f]_{W^{s}_p}^p:=\int_{\mathbb{R}^2}\frac{|f^{([s])}(x)-f^{([s])}(y)|^p}{|x-y|^{1+\{s\}p}}\, d(x,y)=\int_{\mathbb{R}}\frac{\|f^{([s])} -\tau_\xi f^{([s])}\|_p^p}{|\xi|^{1+\{s\}p}}\, d\xi.
\] 
Here $\{\tau_\xi\}_{\xi\in\mathbb{R}}$ denotes the group of right translations and  $\|\cdot\|_p:=\|\cdot\|_{L_p(\mathbb{R})}$.
We study the problem (\ref{P}) in all subcritical spaces $W^s_p(\mathbb{R})$ with $s\in (1+1/p,2)$.\medskip

\subsection*{Reformulation of  $(\ref{P})$} In a compact form, the problem (\ref{P}) can be formulated as
\begin{equation}\label{P'}
\left\{
\begin{array}{rlll}
  \cfrac{d f}{dt}\!\!&=&\!\! \mathbb{B}(f)[\overline{\omega}],\quad t\geq0,\\[2ex]
  \displaystyle -C_{\Theta}  f' \!\!&=&\!\! (1+ a_\mu \mathbb{A}(f))[\overline{\omega}],\quad t\geq0,\\[2ex]
  f(0)\!\!&=&\!\! f_0.
\end{array}
\right.
\end{equation}
The first two equations of (\ref{P'}) should hold in $W^{s-1}_p(\mathbb{R})$ and $f'(t):=d(f(t))/dx$.  
 Moreover,~$\mathbb{A}(f)$ and $\mathbb{B}(f)$ are the singular integral operators defined by
\begin{align}
\mathbb{A}(f)[\overline{\omega}]&:=\frac{1}{\pi}\PV\int_\mathbb{R}\frac{ y f'(x)-(f(x)-f(x-y)) }{ y^2+(f(x)-f(x-y))^2 }\overline{\omega}(x-y)\, dy,\label{OpA}\\[1ex]
\mathbb{B}(f)[\overline{\omega}]&:=\frac{1}{\pi}\PV\int_\mathbb{R}\frac{y+ f'(x)(f(x)-f(x-y))}{ y^2+(f(x)-f(x-y))^2 }\overline{\omega}(x-y)\, dy\label{OpB}.\\[-2ex]\nonumber
\end{align}

\subsection*{Summary of known results} The Muskat problem was introduced in \cite{Mu34}, but the reformulation (\ref{P}) and many of the results  on this classical problem are very recent.
It is important to stress out that most of the results pertaining to (\ref{P}) are established in $L_2$-based Sobolev spaces. 
The main reasons are:
\begin{itemize}  
\item The $L_2$-continuity of singular integral operators is an important problem in the harmonic analysis and many results are available in this context;  
\item Plancherel's theorem can be used; 
\item When $a_\mu\neq0$, the  equation $(\ref{P:1})_2$ (see also $(\ref{P'})_2$) is a linear equation for $\overline{\omega}$. 
In the $L_2$-setting this equation can be solved by using  an integral identity,  known as  the Rellich formula. 
An $L_p$-version, $p\neq 2$,  of the Rellich formula is not available.
\end{itemize}
  
In the particular case when the Atwood number satisfies $a_\mu=0$, the equation $(\ref{P:1})_2$ identifies $\overline{\omega}$ as a function of $f$ and (\ref{P}) can be recast as a quasilinear equation for $f$ which is parabolic when
the fluid located below is   denser, that is when  $\rho_->\rho_+,$ cf. e.g. \cite{MBV19}.
The well-posedness of the resulting equation in $L_2$-based Sobolev spaces was established in \cite{CG07} in $H^3(\mathbb{R})$ and in \cite{ MBV19} for $H^{3/2+\varepsilon}(\mathbb{R})$-data, $\varepsilon\in(0,1/2)$,
while \cite{CGSV17} addressed this issue in $W^2_p(\mathbb{R})\cap L_2(\mathbb{R})$ with $1<p\leq \infty.$ 
Solutions corresponding to medium size data   in $H^{3/2+\varepsilon}(\mathbb{R})$  exists globally, cf. \cite{MBV19, CCGS13, CCGS16, PS17, Cam19}, while  the solutions determined by certain  
 initial data with steeper slope
break  down in finite time
 \cite{CCFG13, CCFGL12, CGFL11}.
 Exponential stability results of the (flat) equilibria  for the periodic counterpart of (\ref{P}) were established in \cite{Matioc2019, MW20}. 
  For well-posedness  results in homogeneous $L_2$-Sobolev spaces we refer to \cite{DLL17, AL20}.
  Moreover, the papers \cite{GG14, BCG14} studied the inhomogeneous Muskat problem with nonconstant permeability, while \cite{CGO14, GB14} consider  (\ref{P}) in a confined geometry. 
  
  The general case when $a_\mu\neq0$ is more involved as  additionally the equation $(\ref{P:1})_2$ needs to be solved.
  In this context the quasilinear character is lost and the Muskat problem has to be treated as a fully nonlinear and nonlocal problem which is of parabolic type in the open subset of the phase space 
   identified by the Rayleigh-Taylor condition, cf. e.g. \cite{MBV18}. 
 The Rayleigh-Taylor condition is a restriction imposed in the classical formulation of the Muskat problem on the sign 
of the jump of the normal derivative of the pressure at the interface between the fluids. 
The normal is taken to point into the upper region occupied by the fluid~$+$.
To be more precise, the jump of the normal derivative of the pressure has to have positive sign at each point of the interface when passing from the region occupied by the fluid $-$ into the region of the fluid $+$.
Local existence for the periodic counterpart of~(\ref{P}) was first established in \cite{CCG11}  in the phase space $H^3(\mathbb{S})$. 
Later on in \cite{BCS16} the authors proved a well-posedness result for $H^2$-data  with small $H^{3/2+\varepsilon}$-norm, with $\varepsilon<<1$.
More recently, it was shown in \cite{MBV18, MBV20} that (\ref{P}) is well-posed in $H^2(\mathbb{R})$ and $H^2(\mathbb{S})$ without any smallness conditions.   
Well-posedness in the subcritical spaces $H^s(\mathbb{R}^d)$ with $s>1+d/2$ was only recently established in \cite{NP20} by using a paradifferential approach.
This is the first   local well-posedness result that covers all $L_2$-subcritical spaces in all dimensions.
The existence of global  weak solutions  for medium size initial data in critical spaces together with sharp algebraic decay estimates for the 3D counterpart of (\ref{P}) was addressed in \cite{GGPS19}.
 Finally, we point out that the exponential stability of the (flat) equilibria is established  in the periodic setting in \cite{MBV20}.

Other papers consider the Muskat problem in other geometries or settings, cf. \cite{A04, A14, EMM12a, EM11a, EMW18, SCH04, Y96, CG10, GS19, BV14, NF20x}. 
Besides, there are also many studies which address the Muskat problem with surface tension effects, cf.
 \cite{A14, EMM12a, EM11a, EMW18, PS16, PS16x, Ngu20, NF20x,  To17, GGS20}, see also the review articles \cite{G17, GL20}.
 A particular feature of the Muskat problem with surface tension is that in the case when the less viscous fluid  penetrates the region occupied by the fluid with a larger viscosity (or when the denser fluid is located above) there may exist finger-shaped equilibria. The finger-shaped equilibria with  small  amplitude are unstable, cf. \cite{EM11a, EMM12a, MBV20}.
Moreover, it is shown in \cite{GGS20} in the context of the one-phase Muskat problem that surface tension prevents, for fluid interfaces
 smaller than an explicit constant, the formation of fluid drops in finite time, and the corresponding  solutions are global in time. Furthermore, the solutions become instantly analytic.

\subsection*{Main results and strategy of proof}
The main goal of this paper is to establish  a well-posedness theory for  (\ref{P}) that covers all subcritical spaces $W^s_p(\mathbb{R})$ with 
\[\mbox{$s\in (1+1/p,2)$\quad and\quad $1<p<\infty$.}\]
This setting has been previously considered only in \cite{CGSV17} in the special case $a_\mu=0$. 
We point out that in \cite{CGSV17} not all subcritical spaces were covered and   additional $L_2$-integrability of the data was required.
Our strategy is to formulate (\ref{P}) as an abstract evolution problem, cf.~(\ref{NNEP}), and to prove that  this problem is parabolic in the set where the Rayleigh-Talyor condition holds.
In this setting the Rayleigh-Taylor condition   can be formulated as 
\begin{align}\label{RT'}
C_{\Theta}+a_\mu\mathbb{B}(f)[\overline{\omega}(f)]>0,
\end{align}
see Section~\ref{Sec:4}, where  $\mathbb{B}(f)$ is the operator introduced in (\ref{OpB}). 
 Moreover,  given $f\in W^s_p(\mathbb{R})$, the function $\overline{\omega}=\overline{\omega}(f)$ is identified as the unique solution to $(\ref{P'})_2$, cf.~(\ref{SO}). 
Our analysis shows that $\mathbb{B}(f)[\overline{\omega}(f)]\in W^{s-1}_p(\mathbb{R})$,  and therefore (\ref{RT'}) implies that $\Theta\geq0$.
For $\Theta>0$ (the case $\Theta=0$ is not interesting, see Section~\ref{Sec:4}) we prove that  
\[
\mathcal{O}:=\{f\in W^s_p(\mathbb{R})\,:\, C_{\Theta}+a_\mu\mathbb{B}(f)[\overline{\omega}(f)]>0\}
\] 
defines an open subset of $W^s_p(\mathbb{R})$ and the problem (\ref{P})
is parabolic  within $\mathcal{O}$\footnote{Given $f\in W^s_p(\mathbb{R})$, the operators  $\mathbb{A}(f),\,\mathbb{B}(f) $ are linear and $\overline{\omega}(f)=-C_{\Theta}(1+a_\mu\mathbb{A}(f))^{-1}[f']$. 
Hence, the Rayleigh-Taylor condition is equivalent to the relations
\[\Theta>0\quad\mbox{and}\quad 1-a_\mu\mathbb{B}(f)[(1+a_\mu\mathbb{A}(f))^{-1}[f']]>0.\]
 The first condition is imposed on the constants only, while the second one relates the Atwood number $a_\mu$~to~$f$.}.

An important tool in our analysis is the following result.

\begin{thm}\label{T:A}
Let $a:\mathbb{R}\to\mathbb{R}$ be continuously differentiable with bounded and H\"older-continuous first derivative. For $f\in{\rm C}^\infty_0(\mathbb{R})$ let
    \begin{align*}
      T_a[f](x)&:= \PV\int_\mathbb{R} \frac{f(x-y)}{y}\exp\Big(i\frac{a(x)-a(x-y)}{y}\Big)\, dy\\
      &:=\lim_{\varepsilon\to 0} \int_{|y|\geq \varepsilon} \frac{f(x-y)}{y}\exp\Big(i\frac{a(x)-a(x-y)}{y}\Big)\, dy.         
    \end{align*}
  Given $p\in (1,\infty)$,  the operator  $T_a$ has an extension  $T_a\in \mathcal{L}(L_p(\mathbb{R}))$  and it holds that
  \[
\|T_a\|_{\mathcal{L}(L_p(\mathbb{R}))}\leq C_p(1+\|a'\|_\infty).
\]
The constant $C_p$ depends only on  $p$.
\end{thm}
 
 The result of Theorem~\ref{T:A} also holds for $a$ merely Lipschitz continuous.
 Then, the operator~$T_a$ has to be defined by a suitable series as in \cite[Section 9.6]{CM97}.
In the canonical case~$p=2$, this result has   already been established in  \cite{TM86} (see also \cite[Chapter 9, Rel.~(6.7)]{CM97} and \cite{TM86b} for a weaker version  of this result).
Theorem~\ref{T:A}  extends the result of \cite{TM86}  to the  $L_p$-setting with $p\in(1,\infty)$.
Actually, having established Theorem~\ref{T:A} for $p\in(1,2)$, the case $p>2$ follows by duality since 
 the adjoint $T_a^*$ of $T_a\in\mathcal{L}(L_2(\mathbb{R}))$ is given by the formula
\begin{equation*}
T_a^*=-T_{-a}.
\end{equation*}
Theorem~\ref{T:A} follows in the case $p\in (1,2)$  from well-known results of the theory of singular integral operators, e.g.~\cite[Theorem~5.5]{AB12}, once the so-called H\"ormander condition is established, which is done in Lemma \ref{L:HC} below. We note that the estimate of the operator norm by a multiple of $1+\|a'\|_\infty$ follows by a simple scaling argument and an inspection of the proof in the same way as e.g.
 in \cite[Proposition 4.28]{AB12}. Here one uses that the constant in the H\"ormander condition and the operator norm on $L_2(\mathbb{R})$ can be bounded by a multiple of~${1+\|a'\|_\infty}$.

A further issue that we had to consider was to solve the  equation $(\ref{P:1})_2$ (or equivalently~$(\ref{P'})_2$) for~$\overline{\omega}$, as the Rellich formula is not available for $p\neq2$.
The arguments use quite technical localization procedures. 
Moreover, the proof in the case $p\in(1,2)$ is different from that for $p\in(2,\infty)$, see Theorem~\ref{T:I1} and Theorem~\ref{T:I2} below.

 The analysis becomes  quite involved also when showing that the evolution problem~(\ref{NNEP}) below (which is a compact reformulation of (\ref{P'}))  is parabolic in $\mathcal{O}$. 
With respect to this goal we establish in Lemma \ref{L:B2}   a commutator estimate  which is used several times in the paper (especially in the proof of the lemmas in the Appendix~\ref{Sec:C}, 
 Theorem~\ref{T:AP}, and  Proposition~\ref{P:IP0}). \smallskip

The main result of this paper is the following theorem.

\begin{thm}\label{MT} Let $p\in(1,\infty)$, $s\in(1+1/p,2), $ and assume that $\Theta>0$.
Then, the following hold true:
\begin{itemize}
\item[(i)]  {\em (Well-posedness)}  Given $f_0\in\mathcal{O}$, there exists a unique maximal solution 
\[f=f(\cdot;f_0)\in {\rm C}([0,T_+), \mathcal{O})\cap {\rm C}^1([0,T_+), W^{s-1}_p(\mathbb{R})),\]
where $T_+=T_+(f_0)\in (0,\infty]$, to (\ref{P}). 
Moreover, $[(t,f_0)\mapsto f(t;f_0)]$ defines a  semiflow on $\mathcal{O}$.\\[-2.2ex]
\item[(ii)]  {\em (Parabolic smoothing)} 
\begin{itemize}
\item[(iia)] $[(t,x)\mapsto f(t,x)]:(0,T_+)\times\mathbb{R}\to\mathbb{R}$ is a real-analytic function;
\item[(iib)] Given $k\in\mathbb{N}$, it holds that $f\in {\rm C}^\omega ((0,T_+), W^k_p(\mathbb{R})).$\footnote{Here ${\rm C}^\omega$   denotes real-analyticity.}
\end{itemize} 
\end{itemize} 
\end{thm}

The proof of Theorem \ref{MT} and of  Remark \ref{R:MR} below is postponed to the end of Section~\ref{Sec:4}.

 \begin{rem}\label{R:MR}
 Given $\alpha\in(0,1)$, $T>0$,  and a Banach space $X$,   let $B((0,T], X)$   denote the Banach space of all bounded functions from $(0,T]$  into $X$ and  set 
 \[
 {\rm C}^\alpha_\alpha((0,T], X):=\Big\{f\in B((0,T], X)\,:\, [f]_{C^\alpha_\alpha}:=\sup_{s\neq t}\frac{\|t^\alpha f(t)-s^\alpha f(s)\|_X}{|t-s|^\alpha}<\infty\Big\}.
 \]
 Then, for each  $f_0\in\mathcal{O}$, the solution $f=f(\cdot;f_0)$ found in  Theorem \ref{MT} also satisfies
 \[
 f\in \bigcap_{\alpha\in(0,1)}{\rm C}^{\alpha}_{\alpha}((0,T], W^s_p(\mathbb{R})) \qquad \forall \, T\in(0,T_+(f_0)).
 \]
 \end{rem}

\subsection*{Organization of the paper} In Section \ref{Sec:2} we establish the boundedness of certain multilinear singular operators which is then
 used to derive some useful  mapping properties for the operators $\mathbb{A}$ and $\mathbb{B}$ in~(\ref{P'}). 
Section~\ref{Sec:3} is devoted to the solvability issue for the equation~$(\ref{P'})_2$. 
Finally, in Section \ref{Sec:4}, we formulate (\ref{P}) as an evolution equation for $f$, and show that this equation is parabolic in $\mathcal{O}$. We conclude this section with the proof of Theorem~\ref{MT}.
In Appendix~\ref{Sec:C} we prove some technical results that are used in Section~\ref{Sec:4}.

\section{Preliminaries}\label{Sec:2}
We first clarify the  notation used in this paper. 
Then, we check the H\"ormander condition for the kernel of the operator $T_a$ in Theorem~\ref{T:A}. 
This condition builds the fundament of the proof of Theorem~\ref{T:A}. 
The bulk of this section  addresses the boundedness of certain multilinear singular operators and culminates with the proof of Lemma~\ref{L:REG}
 where mapping properties for the operators $\mathbb{A}$ and $\mathbb{B}$ from (\ref{OpA}) 
and (\ref{OpB}) are established.

\subsection*{Notation}
Given $k\in\mathbb{N}$, we let ${\rm C}^k(\mathbb{R}) $ denote the Banach space of $k$-times continuously differentiable functions having bounded derivatives.
Given $\alpha\in(0,1)$, the Hölder space ${\rm C}^{k+\alpha}(\mathbb{R}) $ is the subspace of ${\rm C}^k(\mathbb{R}) $ that consists of functions with $k$th derivative  having finite Hölder seminorm, that is 
\[
[f^{(k)}]_\alpha:=\sup_{x\neq y}\frac{|f^{(k)}(x)-f^{(k)}(y)|}{|x-y|^\alpha}<\infty.
\]
Sobolev's embedding states that $W^{r}_p(\mathbb{R})\hookrightarrow {\rm C}^{r-1/p}(\mathbb{R})$ provided that $r>1/p$.
Besides, given $k\in\mathbb{N}$ with $k<r-1/p,$ since the smooth  function with compact support are dense in $W^r_p(\mathbb{R})$, for $f\in W^r_p(\mathbb{R})$ it holds that 
$f^{(k)}(x)\to0$ for $|x|\to\infty$.
Furthermore,   the following estimate finds several times application in the analysis:
\begin{align}\label{MES}
\hspace{-0.25cm}\|gh\|_{W^{r}_p}\leq 2(\|g\|_\infty\|h\|_{W^{r}_p}+\|h\|_\infty\|g\|_{W^{r}_p}),\quad\mbox{ $g,$ $ h\in W^{r}_p(\mathbb{R})$, $r\in(1/p,1)$, $p\in[1,\infty)$.}
\end{align} 
We also write ${\rm C}^{1-}$ to denote  local Lipschitz continuity.

\subsection*{The H\"ormander condition} 
Defining the singular kernel 
\begin{align}\label{ker}
k(x,y):=\frac{1}{y}\exp\Big(i\frac{a(x)-a(x-y)}{y}\Big), \quad x\in\mathbb{R},\, y\in\mathbb{R}\setminus\{0\},
\end{align}
for $f\in{\rm C}^\infty_0(\mathbb{R})$    it holds that
\[
T_a[f](x)= \PV\int_\mathbb{R} f(y)k(x,x-y)\, dy.
\]
A simple computation reveals that 
\begin{align}\label{Est:K1}
|\partial_y k(x,y)|\leq 2(1+\|a'\|_\infty)y^{-2}, \quad x\in\mathbb{R},\, y\in\mathbb{R}\setminus\{0\},
\end{align}
and the H\"ormander condition can be now established.
\begin{lemma}[The H\"ormander condition]\label{L:HC} Let $a:\mathbb{R}\to\mathbb{R}$ be a Lipschitz continuous function and let $k$ be the kernel defined in (\ref{ker}). 
Given $x_0\in\mathbb{R}$ and $y\in\mathbb{R}\setminus\{0\}$, it then holds
\begin{align*}
\int_{[|x|>2|y|]} |k(x+x_0,x-y)- k(x+x_0,x)|\, dx\leq 8(1+\|a'\|_\infty).
\end{align*}
\end{lemma}
\begin{proof}
It follows from   (\ref{Est:K1}) and the mean value theorem  that
\begin{align*}
&\int_{[|x|>2|y|]} |k(x+x_0,x-y)- k(x+x_0,x)|\, dx=\int_{[|x|>2|y|]} |\partial_y k(x+x_0,\xi_y)y|\, dx\\[1ex]
&\qquad \leq 2(1+\|a'\|_\infty)\int_{[|x|>2|y|]} |y\xi_y^{-2}|\, dx\leq 8(1+\|a'\|_\infty),
\end{align*}
where we used that $\xi_y=x-ty$, with $t\in[0,1],$ satisfies  $|\xi_y|\geq |x|/2.$
\end{proof}

\subsection*{Boundedness of some multilinear singular integral operators}

The first goal of this  subsection is to show that, for any $s\in(1+1/p,2)$, with $p\in(1,\infty)$, it holds 
\begin{align}\label{PRO1} 
 \mathbb{A},\, \mathbb{B}\in {\rm C}^{\omega}(W^s_p(\mathbb{R}),\mathcal{L}( L_p(\mathbb{R})))\cap {\rm C}^{\omega}(W^s_p(\mathbb{R}),\mathcal{L}(  W^{s-1}_p(\mathbb{R}))).
\end{align} 
 Theorem \ref{T:A} is essential for this purpose.
In the following we set
\[
\delta_{[x,y]} f:=f(x)-f(x-y) =(f-\tau_y f)(x) \qquad\mbox{for $x,\, y\in\mathbb{R}$.}
\]
In order to establish (\ref{PRO1}), but also for later purposes, we provide the following lemma.
\begin{lemma}\label{L:MP1} Let $p\in(1,\infty)$ and $n,\, m\in\mathbb{N}$ be given.
  \begin{itemize}
  \item[(i)] Given  Lipschitz continuous  functions $a_1,\ldots, a_{m},\, b_1, \ldots, b_n:\mathbb{R}\to\mathbb{R}$, the singular integral operator 
  $B_{n,m}(a_1,\ldots, a_{m})[b_1,\ldots,b_n,\,\cdot\,]$ defined by
\[
B_{n,m}(a_1,\ldots, a_m)[b_1,\ldots,b_n,\overline{\omega}](x)
:=\PV\int_\mathbb{R}  \frac{\overline{\omega}(x-y)}{y}\cfrac{\prod_{i=1}^{n}\big(\delta_{[x,y]} b_i /y\big)}{\prod_{i=1}^{m}\big[1+\big(\delta_{[x,y]}  a_i /y\big)^2\big]}\, dy ,
\]
belongs to $\mathcal{L}(L_p(\mathbb{R}))$ and $\|B_{n,m}(a_1,\ldots, a_m)[b_1,\ldots,b_n,\,\cdot\,]\|_{\mathcal{L}(L_p(\mathbb{R}))}\leq C\prod_{i=1}^{n} \|b_i'\|_{\infty}$, where $C$ is a constant depending only 
on $n,\, m$, and $\max_{i=1,\ldots, m}\|a_i'\|_{\infty}.$

 Moreover,   $B_{n,m}\in {\rm C}^{1-}((W^1_\infty(\mathbb{R}))^{m},\mathcal{L}_{n+1}((W^1_\infty(\mathbb{R}))^{n}\times L_p(\mathbb{R}),L_p(\mathbb{R}))).$\\[-1ex]
 \item[(ii)] Given $r\in (1+1/p,2)$ and $\tau\in(1/p,1),$  it holds
   \begin{align*}   
 \|B_{n,m}(a_1,\ldots, a_{m})[b_1,\ldots, b_n,\overline{\omega}]\|_\infty\leq C\|\overline{\omega}\|_{W^\tau_p} \prod_{i=1}^{n} \|b_i\|_{W^r_p} 
\end{align*}
  for all $a_1,\ldots, a_{m}, b_1,\ldots, b_n\in W^r_p(\mathbb{R})$ and $\overline{\omega}\in W^\tau_p(\mathbb{R})$,  with $C$ depending only on $\tau,$ $ r,$ $n $, $m,$ and  $\max_{i=1,\ldots, m}\|a_i \|_{W^r_p}$.
  
 Moreover,   $B_{n,m}\in {\rm C}^{1-}((W^r_p(\mathbb{R}))^{m},\mathcal{L}_{n+1}((W^r_p(\mathbb{R}))^{n}\times W^\tau_p(\mathbb{R}),L_\infty(\mathbb{R}))).$ 
\end{itemize}
\end{lemma}
\begin{proof}
 The proof of (i) is similar to that in the case $p=2$, cf. \cite[Lemma 3.3]{MBV19}, and relies to a large extent on Theorem \ref{T:A}.
 The proof of (ii) uses  similar arguments as  that in the case~$p=2$, cf. \cite[Lemma 3.1]{MBV18}.
\end{proof}

The next lemma collects some properties of the operators $B_{n,m}$.

\begin{lemma}\label{L:SMP1} Let $p\in(1,\infty)$ and $n,\, m\in\mathbb{N}$. Let further $a_1,\ldots, a_{m},\, b_1, \ldots, b_n:\mathbb{R}\to\mathbb{R}$ be Lipschitz continuous and $\overline{\omega}\in L_p(\mathbb{R})$.
  \begin{itemize}
 \item[(i)] If $n\geq 1$ and additionally $b_1,\,\varphi\in W^1_\infty(\mathbb{R}) $,    then  
\begin{equation}\label{spr2}  
  \begin{aligned}
 & \varphi B_{n,m}(a_1,\ldots, a_{m})[b_1,\ldots, b_n,\overline{\omega}]-B_{n,m}(a_1,\ldots, a_{m})[b_1,\ldots, b_n, \varphi \overline{\omega}] \\[1ex]
 &\hspace{0.5cm}= b_1 B_{n,m}(a_1,\ldots, a_{m})[b_2,\ldots,b_{n}, \varphi ,\overline{\omega}]-B_{n,m}(a_1,\ldots, a_{m})[b_2,\ldots,  b_{n}, \varphi , b_1 \overline{\omega}] .
  \end{aligned}
  \end{equation}
   \item[(ii)] If  $\widetilde a_1,\ldots, \widetilde a_m$ are  Lipschitz continuous, then
\begin{equation}\label{spr3}   
   \begin{aligned}
 &B_{n,m}(\widetilde a_{1}, \ldots, \widetilde a_{m})[b_1,\ldots, b_{n},\overline{\omega}]-  B_{n,m}(a_1, \ldots, a_{m})[b_1,\ldots, b_{n},\overline{\omega}]\\[1ex]
 &\hspace{0.5cm}=\sum_{i=1}^{m} B_{n+2,m+1}(\widetilde a_{1},\ldots, \widetilde a_{i},a_i,\ldots \ldots, a_{m})[b_1,\ldots, b_{n},a_i+\widetilde a_{i}, a_i-\widetilde a_{i},\overline{\omega}].
\end{aligned}
\end{equation}
\end{itemize}
\end{lemma}
\begin{proof}
 The proof  is elementary.
\end{proof}

 The importance of the operators $B_{n,m}$ becomes clear when considering the relations
 \begin{align}
 \pi\mathbb{A}(f)[\overline{\omega}]=f'B_{0,1}(f)[\overline{\omega}]-B_{1,1}(f)[f,\overline{\omega}],\label{FormulaA}\\[1ex]
 \pi\mathbb{B}(f)[\overline{\omega}]= B_{0,1}(f)[\overline{\omega}]+f'B_{1,1}(f)[f,\overline{\omega}].\label{FormulaB}
 \end{align}
These relations  together with Lemma \ref{L:MP1} (i) show that, given $f\in W^s_p(\mathbb{R}), $ $s\in(1+1/p,2)$,
it holds that $\mathbb{A}(f),\, \mathbb{B}(f)\in\mathcal{L}(L_p(\mathbb{R})).$ 
Arguing as in \cite[Section 5]{MBV19}, it   actually  holds
 \begin{align*} 
 \mathbb{A},\, \mathbb{B}\in {\rm C}^{\omega}(W^s_p(\mathbb{R}),\mathcal{L}(  L_p(\mathbb{R}))).
\end{align*}

In order to establish the second mapping property in (\ref{PRO1}) some further analysis  of the operators $B_{n,m}$ is needed.
To this end we establish in Lemma \ref{L:MP2} new estimates.
 The estimate (\ref{REF1}) is used in Lemma \ref{L:MP3} below (which is the main ingredient in the proof of~(\ref{PRO1})), 
 while (\ref{REF2})  provides a commutator type $L_p$-estimate which is essential when estimating the $W^{r-1}_p$-norm  
 of this commutator, cf.   Lemma \ref{L:MP3b}.
Lemma \ref{L:MP3b} is used in the proof of Theorem~\ref{T:AP}.  
 \begin{lemma}\label{L:MP2}
 Let  $n,m\in\mathbb{N}$ with $n\geq1,$   $r\in(1+1/p ,2)$, and  $\tau\in(2-r+1/p,1)$   be given. 
 Given $a_1,\ldots, a_m\in W^r_p(\mathbb{R})$, there exists a constant  $C$, depending only on $n,\, m$, $r$,  and $\max_{1\leq i\leq m}\|a_i\|_{W^r_p}$ (and on $\tau$ in (\ref{REF2})), such that
\begin{align} 
&\| B_{n,m}(a_1,\ldots, a_{m})[b_1,\ldots, b_n,\overline{\omega}]\|_p\leq C\|b_1'\|_{p}\|\overline{\omega}\|_{W^{r-1}_p}\prod_{i=2}^{n}\|b_i'\|_{W^{r-1}_p} \label{REF1}
\end{align}
and 
\begin{equation}\label{REF2}
\begin{aligned} 
&\| B_{n,m}(a_1,\ldots, a_{m})[b_1,\ldots, b_n,\overline{\omega}]-\overline{\omega} B_{n-1,m}(a_1,\ldots, a_{m})[b_2,\ldots, b_n,b_1']\|_p \\[1ex]
&\hspace{4.6cm}\leq C\|b_1\|_{W^{\tau}_p}\|\overline{\omega}\|_{W^{r-1}_p}\prod_{i=2}^{n}\|b_i'\|_{W^{r-1}_p}
\end{aligned}
\end{equation}
for all $b_1,\ldots, b_n\in W^r_p(\mathbb{R})$ and $\overline{\omega}\in W^{r-1}_p(\mathbb{R}).$

Moreover,   $  B_{n,m}\in {\rm C}^{1-}((W^r_p(\mathbb{R}))^m,\mathcal{L}_{n+1}( W^1_p(\mathbb{R})\times(W_p^{r}(\mathbb{R}))^{n-1}\times  W^{r-1}_p(\mathbb{R}), L_p(\mathbb{R}))).$ 
 \end{lemma}
\begin{proof}
Without loss of generality  we may assume $\overline{\omega}\in W^{r}_p(\mathbb{R}).$ 
Using the identities
\[
\frac{\partial}{\partial y}\Big(\frac{\delta_{[x,y]}b_1}{y}\Big)=\frac{b_1'(x-y)}{y}-\frac{\delta_{[x,y]}b_1}{y^2}\quad\mbox{and}\quad \overline{\omega}'(x-y)=\frac{\partial}{\partial y}(\overline{\omega}(x)-\overline{\omega}(x-y))
\]
and integration by parts (as in the proof of \cite[Lemma 3.2]{MBV18}),  we arrive at 
\begin{align*}
&\hspace{-1cm} B_{n,m}(a_1,\ldots, a_m)[b_1,\ldots, b_n,\overline{\omega}](x)\\[1ex]
 &=\overline{\omega}(x)B_{n-1,m}(a_1,\ldots, a_{m})[b_2,\ldots, b_n,b_1'](x)\\[1ex]
  &\hspace{0.424cm}-\sum_{j=2}^n\int_\mathbb{R}  K_{1,j}(x,y)  \overline{\omega}(x-y)\, dy+\sum_{j=1}^m\int_\mathbb{R} K_{2,j}(x,y) \   \overline{\omega}(x-y)\, dy\\[1ex]
  &\hspace{0.424cm}-\int_\mathbb{R} K(x,y)\, dy-\sum_{j=2}^n \int_\mathbb{R} K_{3,j}(x,y)\, dy +2\sum_{j=1}^m\int_\mathbb{R} K_{4,j}(x,y)\, dy,
\end{align*}
where, given $x\in\mathbb{R}$ and $y\neq 0$, we have set
\begin{align*}
K(x,y)&:= \frac{\prod_{i=2}^n \delta_{[x,y]}b_i/y}{\prod_{i=1}^m \big[1+ \big(\delta_{[x,y]}a_i/y\big)^2\big]}\frac{\delta_{[x,y]}\overline{\omega}}{y}\frac{\delta_{[x,y]}b_1}{y},\\[1ex]
 K_{1,j}(x,y)&:= \cfrac{\prod_{i=2, i\neq j }^{n}\big(\delta_{[x,y]} b_i /y\big)}{\prod_{i=1}^{m}\big[1+\big(\delta_{[x,y]}  a_i /y\big)^2\big]}\frac{\delta_{[x,y]}b_j-yb_j'(x-y)}{y^2}\frac{\delta_{[x,y]}b_1}{y},\\[1ex]
 K_{2,j}(x,y)&:= 2\cfrac{\prod_{i=2 }^{n}\big(\delta_{[x,y]} b_i /y\big)}{\big[1+\big(\delta_{[x,y]}  a_j /y\big)^2\big]\prod_{i=1}^{m}\big[1+\big(\delta_{[x,y]}  a_i /y\big)^2\big]}
  \frac{\delta_{[x,y]}a_j-ya_j'(x-y)}{y^2}\frac{\delta_{[x,y]}a_j}{y}\frac{\delta_{[x,y]}b_1}{y},\\[1ex]
K_{3,j}(x,y)&:= \frac{\prod_{i=1, i\neq j}^n \delta_{[x,y]}b_i/y}{\prod_{i=1}^m \big[1+ \big(\delta_{[x,y]}a_i/y\big)^2\big]}\frac{\delta_{[x,y]}\overline{\omega}}{y}\Big(\frac{\delta_{[x,y]}b_j}{y}- b_j'(x-y)\Big),\\[1ex]
K_{4,j}(x,y)&:= \frac{\prod_{i=1}^n \delta_{[x,y]}b_i/y}{\big[1+ \big(\delta_{[x,y]}a_j/y\big)^2\big]\prod_{i=1}^m\big[ 1+ \big(\delta_{[x,y]}a_i/y\big)^2\big]}\frac{\delta_{[x,y]}\overline{\omega}}{y}\frac{\delta_{[x,y]}a_j}{y}\Big(\frac{\delta_{[x,y]}a_j}{y}- a_j'(x-y)\Big).
\end{align*}
Recalling Lemma \ref{L:MP1}~(i), we get, with respect to  (\ref{REF1}), that
\begin{align}\label{DE:1}
\|\overline{\omega} B_{n-1,m}(a_1,\ldots, a_{m})[b_2,\ldots, b_n,b_1']\|_p\leq C\|\overline{\omega}\|_\infty\|b_1'\|_p\prod_{i=2}^n\|b_1'\|_\infty.
\end{align}

Let now $\alpha\in\{\tau,1\}$. 
Since $\alpha>1/p$ we have $W^{\alpha}_p(\mathbb{R})\hookrightarrow {\rm C}^{\alpha-1/p}(\mathbb{R})$ and together with Minkowski's  integral inequality  we obtain that  
\begin{align*}
&\hspace{-0.5cm}\Big(\int_\mathbb{R}\Big|\int_\mathbb{R} K_{1,j}(x,y)\overline{\omega}(x-y)\, dy\Big|^p\, dx\Big)^{1/p}\\[1ex]
&\leq  \|\overline{\omega}\|_{\infty}\int_\mathbb{R}\Big(\int_\mathbb{R} |K_{1,j}(x,y)|^p\, dx\Big)^{1/p}\, dy\\[1ex]
&\leq \|\overline{\omega}\|_{\infty}[b_1]_{\alpha-1/p}\Big(\prod_{i=2, i\neq j}^{n}\|b_i'\|_{\infty}\Big)
\int_\mathbb{R}\frac{1 }{|y|^{3-\alpha+1/p}} \Big(\int_\mathbb{R}|b_j-\tau_y b_j-y\tau_y b_j'|^p\, dx\Big)^{1/p}\, dy 
\end{align*}
for  $2\leq j\leq n.$ 
 Fubini's theorem, Minkowski's integral inequality, H\"older's inequality, and a change of variables now  yield 
\begin{align*}
&\hspace{-0.5cm}\int_\mathbb{R}\frac{1 }{|y|^{3-\alpha+1/p}} 
\Big(\int_\mathbb{R}|b_j-\tau_y b_j-y\tau_y b_j'|^p\, dx\Big)^{1/p}\, dy\\[1ex]
&=\int_\mathbb{R}\frac{1 }{|y|^{2-\alpha+1/p}} 
\Big(\int_\mathbb{R} \Big|\int_0^1 [b_j'(x-(1-s)y)- b_j'(x-y)]\, ds\Big|^p\, dx\Big)^{1/p}\, dy
\end{align*}
\begin{align*}
&\leq\int_0^1 \Big[\int_\mathbb{R}\frac{1 }{|y|^{2-\alpha+1/p}}\Big(\int_\mathbb{R}|b_j'(x-(1-s)y)- b_j'(x-y)|^p\, dx\Big)^{1/p}\, dy\Big]\, ds\\[1ex]
  &\leq 2\|b_j'\|_{p}\int_{[|y|\geq1]}\frac{1 }{|y|^{2-\alpha+1/p}}\, dy+\int_0^1 \int_{[|y|<1]}\frac{\|b_j'-\tau_{-sy}b_j'\|_p}{|y|^{2-\alpha+1/p}}  dy\, ds\\[1ex]
  &\leq C\|b_j'\|_{p}+ C\Big(\int_{[|y|<1]}\frac{1 }{|y|^{(3-\alpha-r)p/(p-1)}}\, dy\Big)^{(p-1)/p}  \|b_j'\|_{W^{r-1}_p}\\[1ex]
  &\leq  C\|b_j'\|_{W^{r-1}_p}.
\end{align*}
Consequently, given  $2\leq j\leq n $,  we get
\begin{align}\label{DE:2}
\Big(\int_\mathbb{R}\Big|\int_\mathbb{R} K_{1,j}(x,y)\overline{\omega}(x-y)\, dy\Big|^p\, dx\Big)^{1/p}\leq C\|\overline{\omega}\|_\infty [b_1]_{\alpha-1/p} \Big(\prod_{i=2}^{n}\|b_i'\|_{W^{r-1}_p}\Big),
\end{align} 
and by similar arguments 
\begin{align}\label{DE:3}
\Big(\int_\mathbb{R}\Big|\int_\mathbb{R} K_{2,j}(x,y)\overline{\omega}(x-y)\, dy\Big|^p\, dx\Big)^{1/p}\leq   C\|\overline{\omega}\|_\infty [b_1]_{\alpha-1/p} \Big(\prod_{i=2}^{n}\|b_i'\|_{\infty}\Big)
\end{align}

Furthermore, given $2\leq j\leq n$, H\"older's inequality, Minkowski's integral inequality, and the Sobolev embedding $W^{\alpha}_p(\mathbb{R})\hookrightarrow {\rm C}^{\alpha-1/p}(\mathbb{R})$  yield
\begin{align*}
 &\hspace{-0.5cm}\Big(\int_\mathbb{R}\Big|\int_\mathbb{R} K_{3,j}(x,y)\, dy  \Big|^p\, dx\Big)^{1/p} \leq\int_\mathbb{R}\Big(\int_\mathbb{R} |K_{3,j}(x,y)|^p\, dx  \Big)^{1/p}\, dy\\[1ex]
 &\leq 2 [b_1]_{\alpha-1/p}\Big(\prod_{i=2}^n \|b_i'\|_\infty\Big)\int_\mathbb{R}\frac{1}{|y|^{2-\alpha+1/p}}\Big(\int_\mathbb{R} |\overline{\omega}-\tau_y\overline{\omega}|^p\, dx  \Big)^{1/p}\, dy
\end{align*}
and
\begin{align*}
&\hspace{-0.5cm}\int_\mathbb{R}\frac{1}{|y|^{2-\alpha+1/p}}\Big(\int_\mathbb{R} |\overline{\omega}-\tau_y\overline{\omega}|^p\, dx  \Big)^{1/p}\, dy\\[1ex]
&\leq C\|\overline{\omega}\|_p  +\int_{[|y|<1]}\frac{1}{|y|^{2-\alpha+1/p}}\Big(\int_\mathbb{R} |\overline{\omega}-\tau_y\overline{\omega}|^p\, dx  \Big)^{1/p}\, dy
\\[1ex]
&\leq C\|\overline{\omega}\|_p +\|\overline{\omega}\|_{W^{r-1}_p}\Big(\int_{[|y|<1]}\frac{1}{|y|^{(3-\alpha-r)p/(p-1)}}\, dy\Big)^{(p-1)/p}\\[1ex]
&\leq C\|\overline{\omega}\|_{W^{r-1}_p}.
\end{align*}
We arrive at
\begin{align}\label{DE:4}
\Big(\int_\mathbb{R}\Big|\int_\mathbb{R} K_{3,j}(x,y)\, dy  \Big|^p\, dx\Big)^{1/p} \leq C\|\overline{\omega}\|_{W^{r-1}_p}[b_1]_{\alpha-1/p}\Big(\prod_{i=2}^n \|b_i'\|_\infty\Big),\qquad 2\leq j\leq n.
\end{align}
The same arguments show that
  \begin{equation}\label{DE:5}
\begin{aligned}
&\Big(\int_\mathbb{R}\Big|\int_\mathbb{R} K(x,y)\, dy  \Big|^p\, dx\Big)^{1/p}+\Big(\int_\mathbb{R}\Big|\int_\mathbb{R} K_{4,j}(x,y)\, dy  \Big|^p\, dx\Big)^{1/p} \\[1ex]
&\hspace{2cm}\leq C\|\overline{\omega}\|_{W^{r-1}_p}[b_1]_{\alpha-1/p}\Big(\prod_{i=2}^n \|b_i'\|_\infty\Big),\qquad 1\leq j\leq m.
\end{aligned}
\end{equation}
Choosing $\alpha=1$,   (\ref{REF1}) follows from (\ref{DE:1})-(\ref{DE:5})  and the relation~${[b_1]_{\alpha-1/p}\leq \|b_1'\|_p}$.
Moreover, (\ref{REF2}) follows from (\ref{DE:2})-(\ref{DE:5}) for  $\alpha=\tau.$
Finally, the local Lipschitz continuity property is a consequence of (\ref{spr3}).
\end{proof}

Lemma \ref{L:MP2} enables us to establish estimates in suitable fractional Sobolev spaces for the multilinear operators $B_{n,m}$ considered above.

\begin{lemma}\label{L:MP3}
 Let  $ n,m\in\mathbb{N}$  and $r\in(1+1/p ,2)$   be given. 
 Given $a_1,\ldots, a_m \in W^r_p(\mathbb{R})$, there exists a constant  $C$, depending only on $n,\, m$, $r$,  and $\max_{1\leq i\leq m}\|a_i\|_{W^r_p}$, such that
\begin{align} 
\| B_{n,m}(a_1,\ldots, a_{m})[b_1,\ldots, b_n,\overline{\omega}]\|_{W^{r-1}_p}\leq C \|\overline{\omega}\|_{W^{r-1}_p}\prod_{i=1}^{n}\|b_i'\|_{W^{r-1}_p} \label{REF1'}
\end{align}
for all $ b_1,\ldots, b_n\in W^r_p(\mathbb{R})$ and $\overline{\omega}\in W^{r-1}_p(\mathbb{R}).$

Moreover,   $  B_{n,m}\in {\rm C}^{1-}((W^r_p(\mathbb{R}))^m,\mathcal{L}_{n+1}((W_p^{r}(\mathbb{R}))^{n}\times  W^{r-1}_p(\mathbb{R}), W^{r-1}_p(\mathbb{R}))).$ 
 \end{lemma}
\begin{proof}
Set $B_{n,m}:=B_{n,m}(a_1,\ldots, a_{m})[b_1,\ldots, b_n,\overline{\omega}].$ 
Recalling Lemma \ref{L:MP1}~(i), it holds that 
\begin{align*} 
\| B_{n,m}\|_{p}\leq C \|\overline{\omega}\|_{p}\prod_{i=1}^{n}\|b_i\|_{W^r_p}.
\end{align*}
It thus remains to consider  the $W^{r-1}_p$-seminorm of $B_{n,m}$.
To this end we observe that
\begin{align*}
[B_{n,m}]_{W^{r-1}_p}^p&= \int_{\mathbb{R}}\frac{\|B_{n,m}-\tau_\xi B_{n,m}\|_p^p}{|\xi|^{1+(r-1)p}}\, d\xi,
\end{align*}
where, using (\ref{spr3}), we write  
\begin{align*}
(B_{n,m}-\tau_\xi B_{n,m})(x)= T_1(x,\xi)+T_2(x,\xi)-T_3(x,\xi),\quad x\in\mathbb{R},\, \xi\ne0,
\end{align*}
with
\begin{align*}
T_1(x,\xi)&:=B_{n,m}(a_1,\ldots, a_{m})[b_1,\ldots, b_n,\overline{\omega}-\tau_\xi\overline{\omega}](x),\\[1ex]
T_2(x,\xi)&:=\sum_{i=1}^nB_{n,m}(a_1,\ldots, a_{m})[\tau_\xi b_1,\ldots,\tau_\xi b_{i-1}, b_i-\tau_\xi b_i, b_{i+1},\ldots b_n,\tau_\xi\overline{\omega}](x),\\[1ex]
T_3(x,\xi)&:=\sum_{i=1}^mB_{n+2,m+1}(a_1,\ldots,a_{i},\tau_\xi a_i,\ldots, \tau_\xi a_{m})[\tau_\xi b_1,\ldots\tau_\xi b_n,a_i+\tau_\xi a_i,a_i-\tau_\xi a_i, \tau_\xi\overline{\omega}](x).
\end{align*}
Hence,
\begin{align}\label{E:1}
[B_{n,m}]_{W^{r-1}_p}^p\leq 3^p\sum_{\ell=1}^3\int_{\mathbb{R}}\frac{\|T_\ell(\cdot,\xi)\|_p^p}{|\xi|^{1+(r-1)p}}\, d\xi,
\end{align} 
and, recalling  Lemma \ref{L:MP1}~(i), it holds
\begin{align}\label{AAE:1}
 \int_{\mathbb{R}}\frac{\|T_1(\cdot,\xi)\|_p^p}{|\xi|^{1+(r-1)p}}\, d\xi&\leq C^p\Big(\prod_{i=1}^{n}\|b_i'\|_{\infty}^p \Big) \int_{\mathbb{R}}\frac{\|  \overline{\omega}-\tau_\xi\overline{\omega}\|_p^p}{|\xi|^{1+(r-1)p}}\, d\xi
 \leq \Big(C\|\overline{\omega}\|_{W^{r-1}_p}\prod_{i=1}^{n}\|b_i'\|_{\infty}\Big)^p.
\end{align} 
Furthermore, in virtue of  (\ref{REF1}), we get that
\begin{equation}\label{AAE:2}
\begin{aligned}
\int_{\mathbb{R}}\frac{\|T_2(\cdot,\xi)\|_p^p}{|\xi|^{1+(r-1)p}}\, d\xi 
&\leq C^p\|\overline{\omega}\|_{W^{r-1}_p}^p\sum_{i=1}^n  \left(\int_{\mathbb{R}}\frac{\|  b_i'-\tau_\xi b_i'\|_{p}^p}{|\xi|^{1+(r-1)p}}\, d\xi\prod_{j\neq i}\|b_j'\|_{W^{r-1}_p}^p\right)  \\[1ex]
&\leq \Big(C\|\overline{\omega}\|_{W^{r-1}_p}\prod_{i=1}^{n}\|b_i'\|_{W^{r-1}_p} \Big)^p
\end{aligned}
\end{equation}
and, by similar arguments,
\begin{equation}\label{AAE:3}
\begin{aligned}
 \int_{\mathbb{R}}\frac{\|T_3(\cdot,\xi)\|_p^p}{|\xi|^{1+(r-1)p}}\, d\xi\leq \Big(C\|\overline{\omega}\|_{W^{r-1}_p}\prod_{i=1}^{n}\|b_i'\|_{W^{r-1}_p} \Big)^p.
\end{aligned}
\end{equation}
The estimates (\ref{E:1})-(\ref{AAE:3}) lead to the desired estimate.
Finally, the local Lipschitz continuity follows from (\ref{spr3}) and (\ref{REF1'}).
\end{proof}

We now estimate the commutator type operator from (\ref{REF2}) in the  $\|\cdot\|_{W^{r-1}_p}$-norm.
\begin{lemma}\label{L:MP3b}
 Let  $ n,m\in\mathbb{N}$, $n\geq 1$, $r\in(1+1/p ,2)$, and $r'\in (1+1/p ,r)$  be given. 
Given $a_1,\ldots, a_m\in W^r_p(\mathbb{R})$, there exists a constant  $C$, depending only on $n,\, m$, $r$, $r'$,  and $\max_{1\leq i\leq m}\|a_i\|_{W^r_p}$, such that
\begin{equation}\label{REF2'}
\begin{aligned} 
&\| B_{n,m}(a_1,\ldots, a_{m})[b_1,\ldots, b_n,\overline{\omega}] -\overline{\omega} B_{n-1,m}(a_1,\ldots, a_{m})[b_2,\ldots, b_n,b_1']\|_{W^{r-1}_p}\\[1ex]
&\hspace{2cm}\leq C \|b_1\|_{W^{r'}_p}\|\overline{\omega}\|_{W^{r-1}_p}\prod_{i=2}^{n}\|b_i\|_{W^r_p}
\end{aligned}
\end{equation}
for all  $ b_1,\ldots, b_n\in W^r_p(\mathbb{R})$ and $\overline{\omega}\in W^{r-1}_p(\mathbb{R}).$
 \end{lemma}
\begin{proof}
Letting $T:=B_{n,m}(a_1,\ldots, a_{m})[b_1,\ldots, b_n,\overline{\omega}] -\overline{\omega} B_{n-1,m}(a_1,\ldots, a_{m})[b_2,\ldots, b_n,b_1'],$ it follows from (\ref{REF2}) that $\|T\|_p$ can be estimated as in (\ref{REF2'}).
It remains to consider the term
\begin{align*}
[T]_{W^{r-1}_p}^p&=  \int_{\mathbb{R}}\frac{\|T-\tau_\xi T\|_p^p}{|\xi|^{1+(r-1)p}}\, d\xi,
\end{align*}
for which it is convenient to write 
\begin{align*}
(T-\tau_\xi T)(x)=T_1(x,\xi)+T_2(x,\xi)+T_3(x,\xi)+T_4(x,\xi),\qquad x\in\mathbb{R},\, \xi\ne0,
\end{align*}
where
\begin{align*}
T_1(\cdot,\xi)&:=B_{n,m}(a_1,\ldots, a_{m})[b_1,\ldots, b_n,\overline{\omega}-\tau_\xi\overline{\omega}]\\[1ex]
&\hspace{0.45cm}-(\overline{\omega}-\tau_\xi\overline{\omega})B_{n-1,m}(a_1,\ldots, a_{m})[b_2,\ldots, b_n,b_1'],\\[1ex]
T_2(\cdot,\xi)&:= B_{n,m}(a_1,\ldots, a_{m})[b_1-\tau_\xi b_1,b_2,\ldots,  b_n,\tau_\xi\overline{\omega}]\\[1ex]
&\hspace{0,45cm}-\tau_\xi\overline{\omega} B_{n-1,m}(a_1,\ldots, a_{m})[b_2,\ldots,  b_n,b_1'-\tau_\xi b_1'] ,
\end{align*}
\begin{align*}
T_3(\cdot,\xi)&:= B_{n,m}(a_1,\ldots, a_{m})[\tau_\xi b_1,b_2,\ldots,  b_n,\tau_\xi\overline{\omega}]\\[1ex]
&\hspace{0.45cm}-B_{n,m}(\tau_\xi a_1,\ldots, \tau_\xi a_{m})[\tau_\xi b_1,\ldots, \tau_\xi b_n,\tau_\xi\overline{\omega}],\\[1ex]
T_4(\cdot,\xi)&:=\tau_\xi\overline{\omega} B_{n-1,m}(\tau_\xi a_1,\ldots, \tau_\xi a_{m})[\tau_\xi b_2,\ldots, \tau_\xi b_n,\tau_\xi b_1']\\[1ex]
&\hspace{0,45cm}-\tau_\xi\overline{\omega} B_{n-1,m}(a_1,\ldots, a_{m})[b_2,\ldots,  b_n,\tau_\xi b_1'].
\end{align*}
Lemma \ref{L:MP1} (i) and (ii) (with $\tau=r'-1$) implies that
\begin{align*}
\|T_1(\cdot,\xi)\|_p\leq C\|\overline{\omega}-\tau_\xi\overline{\omega}\|_p\|b_1\|_{W^{r'}_p}\prod_{i=2}^n\|b_i\|_{W^r_p},
\end{align*} 
while (\ref{REF2}) (with $\tau=r'-r+1\in(2-r+1/p,1)$) yields
\begin{align*}
\|T_2(\cdot,\xi)\|_p\leq C\|b_1-\tau_\xi b_1\|_{W^{r'-r+1}_p}\|\overline{\omega}\|_{W^{r-1}_p}\prod_{i=2}^n\|b_i\|_{W^r_p}.
\end{align*} 
Finally, recalling (\ref{spr3}), it holds that
\begin{align*}
T_3(\cdot,\xi)&=\sum_{i=2}^nB_{n,m}(a_1,\ldots, a_{m})[\tau_\xi b_1,\ldots,\tau_\xi b_{i-1},b_i-\tau_\xi b_i, b_{i+1},\ldots,   b_n,\tau_\xi\overline{\omega}]\\[1ex]
&\hspace{0,45cm}-\sum_{i=1}^m B_{n+2,m+1}(  a_1,\ldots, a_i,\tau_\xi a_i,\ldots \tau_\xi a_m)[\tau_\xi b_1,  \ldots,  \tau_\xi b_n,a_i+\tau_\xi a_i,a_i-\tau_\xi a_i,\tau_\xi \overline{\omega}],\\[1ex]
T_4(\cdot,\xi)&=\tau_\xi\overline{\omega}\sum_{i=2}^nB_{n,m}(a_1,\ldots, a_{m})[ b_2,\ldots, b_{i-1},\tau_\xi b_i-b_i, \tau_\xi b_{i+1},\ldots,   \tau_\xi b_n,\tau_\xi b_1']\\[1ex]
&\hspace{0,45cm}+\tau_\xi\overline{\omega}\sum_{i=1}^m B_{n+1,m+1}(\tau_\xi a_1,\ldots,\tau_\xi a_i,a_i,\ldots a_m)[ \tau_\xi b_2,\ldots,  \tau_\xi b_n, a_i+\tau_\xi a_i,  a_i-\tau_\xi a_i,\tau_\xi b_1'],
\end{align*}
and repeated use of Lemma \ref{L:MP2} (with $r=r'$) yields
\begin{align*}
\|T_3(\cdot,\xi)\|_p &\leq C\|\overline{\omega}\|_{W^{r-1}_p}\|b_1\|_{W^{r'}_p}\Big[\sum_{i=2}^n\Big(\prod_{j=2, j\neq i}^n\|b_j\|_{W^r_p}\Big)\|b_i-\tau_\xi b_i\|_{W^1_p} \\[1ex]
 &\hspace{3,75cm} +\Big(\prod_{j=2}^n\|b_j\|_{W^r_p}\Big)\sum_{i=1}^m\|a_i-\tau_\xi a_i\|_{W^1_p}\Big]
 \end{align*}
 and
 \begin{align*}
 \|T_4(\cdot,\xi)\|_p&\leq C\|\overline{\omega}\|_{W^{r-1}_p}\|b_1\|_{W^{r'}_p}\Big[\sum_{i=2}^n\Big(\prod_{j=2, j\neq i}^n\|b_j\|_{W^r_p}\Big)\|b_i-\tau_\xi b_i\|_{W^1_p}  \\[1ex]
 &\hspace{3,75cm} +\Big(\prod_{j=2}^n\|b_j\|_{W^r_p}\Big)\sum_{i=1}^m\|a_i-\tau_\xi a_i\|_{W^1_p}\Big].
\end{align*}
Gathering these estimates, we conclude that
\begin{align*}
[T]_{W^{r-1}_p}\leq C \|\overline{\omega}\|_{W^{r-1}_p}\prod_{i=2}^{n}\|b_i\|_{W^r_p}\left[\|b_1\|_{W^{r'}_p}+\Big(\int_{\mathbb{R}}\frac{\|b_1-\tau_\xi b_1\|_{W^{r'-r+1}_p}^p}{|\xi|^{1+(r-1)p}}\, d\xi\Big)^{1/p}\right],
\end{align*}
which together with Lemma \ref{L:B1} below proves the claim.
\end{proof} 
 
In the proof of Lemma \ref{L:MP3b} we have used the following result.

\begin{lemma}\label{L:B1} Let $p\in(1,\infty) $ and $1<r'<r<2$.
 Then, there exists a constant $C>0$ such that
 \begin{align*}
\int_\mathbb{R}\frac{\|b-\tau_\xi b\|_{W^{r'-r+1}_p}^p}{|\xi|^{1+(r-1)p}}\, d\xi\leq C\|b\|_{W^{r'}_p}^p\qquad\mbox{for all $b\in W^{r'}_p(\mathbb{R})$.}
 \end{align*}
 \end{lemma}
\begin{proof}
The claim follows by using the mean value theorem and the definition of the Sobolev norm. We omit the   details.
\end{proof}

We are now in a position to prove the second claim in (\ref{PRO1}).

\begin{lemma}\label{L:REG}
Given $s\in (1+1/p,2)$, it holds that
\begin{align}\label{PRO1b} 
 \mathbb{A},\, \mathbb{B}\in   {\rm C}^{\omega}(W^s_p(\mathbb{R}),\mathcal{L}(  W^{s-1}_p(\mathbb{R}))).
\end{align}
\end{lemma}
\begin{proof}
Combining (\ref{FormulaA})-(\ref{FormulaB}), Lemma \ref{L:MP3}, and the algebra property of $W^{s-1}_p(\mathbb{R})$ it follows that 
  $$[f\mapsto\mathbb{A}(f)],\, [f\mapsto\mathbb{B}(f)]\in {\rm C}^{1-}(W^s_p(\mathbb{R}),\mathcal{L}(  W^{s-1}_p(\mathbb{R}))).$$ 
Moreover,   arguing as in \cite[Section 5]{MBV19}, it can be shown that $\mathbb{A}$ and $\mathbb{B}$ depend analytically on $f\in W^{s}_p(\mathbb{R}).$  
\end{proof}

\section{On the resolvent set of $\mathbb{A}(f)$}\label{Sec:3}

We now fix $s\in(1+1/p,2)$ and $f\in W^s_p(\mathbb{R})$.
The main goal  of this section is to show that  the equation $(\ref{P'})_2$ has  a unique solution $\overline{\omega}\in W^{s-1}_p(\mathbb{R})$.
Compared to the canonical case $p=2$, where the  Rellich formula, see \cite[Eq. (3.24)]{MBV18}, can be used to solve $(\ref{P'})_2$, for $p\neq 2$ 
we need to find a new approach as the Rellich formula does not apply directly.

To  start,  we infer from the arguments in \cite[Theorem~3.5]{MBV19} that, given $\lambda\in\mathbb{R}$   with $|\lambda|\geq 1$, 
the operator $\lambda-\mathbb{A}(f)$ is an $L_2(\mathbb{R})$-isomorphism, i.e. it  belongs to ${\rm Isom}(L_2(\mathbb{R}))$.
Moreover, the $L_2$-adjoint $(\mathbb{A}(f))^*$ of $  \mathbb{A}(f) $ is given by
\[
(\mathbb{A}(f))^*[\overline{\omega}]=\pi^{-1}(B_{1,1}(f)[f,\overline{\omega}]-B_{0,1}(f)[f'\overline{\omega}]),
\]
and, letting   $p'=p/(p-1)$ denote  the dual exponent to $p$,  it follows from   Lemma~\ref{L:MP1}~(i)   that
\begin{align}\label{PRO1''} 
  (\mathbb{A}(f))^* \in  \mathcal{L}( L_{p'}(\mathbb{R})).
\end{align}

The main step towards our goal is to prove the invertibility of $\lambda-\mathbb{A}(f)$  in $\mathcal{L}(L_p(\mathbb{R}))$ for all $\lambda\in\mathbb{R}$   with $|\lambda|\geq 1$, see  Theorem~\ref{T:I1} and Theorem~\ref{T:I2} below.
These results are then used to establish the invertibility of $\lambda-\mathbb{A}(f)$  in $\mathcal{L}(W^{s-1}_p(\mathbb{R}))$ for all $\lambda\in\mathbb{R}$   with $|\lambda|\geq 1$, see  Theorem~\ref{T:INV}.
This necessitates  the introduction of suitable partitions of unity.
To be more precise, we  choose  for each $\varepsilon\in(0,1)$, a finite $\varepsilon$-localization family, that is  a family  
\[\{\pi_j^\varepsilon\,:\, -N+1\leq j\leq N\}\subset  {\rm C}^\infty(\mathbb{R},[0,1]),\]
with $N=N(\varepsilon)\in\mathbb{N} $ sufficiently large, such that
\begin{align*}
\bullet\,\,\,\, \,\,  & \mbox{$\supp \pi_j^\varepsilon $ is an interval of length $\varepsilon$ for all $|j|\leq N-1$, $\supp \pi_{N}^\varepsilon\subset(-\infty,-1/\varepsilon]\cup [1/\varepsilon,\infty)$;} \\[1ex]
\bullet\,\,\,\, \,\, &\mbox{ $ \pi_j^\varepsilon\cdot  \pi_l^\varepsilon=0$ if $[|j-l|\geq2, \max\{|j|, |l|\}\leq N-1]$ or $[|l|\leq N-2, j=N];$} \\[1ex]
\bullet\,\,\,\, \,\, &\mbox{ $\sum_{j=-N+1}^N(\pi_j^\varepsilon)^2=1;$} \\[1ex]
 \bullet\,\,\,\, \,\, &\mbox{$\|(\pi_j^\varepsilon)^{(k)}\|_\infty\leq C\varepsilon^{-k}$ for all $ k\in\mathbb{N}, -N+1\leq j\leq N$.} 
\end{align*} 
To each finite $\varepsilon$-localization family we associate  a second family   
$$\{\chi_j^\varepsilon\,:\, -N+1\leq j\leq N\}\subset {\rm C}^\infty(\mathbb{R},[0,1])$$ with the following properties
\begin{align*}
\bullet\,\,\,\, \,\,  &\mbox{$\chi_j^\varepsilon=1$ on $\supp \pi_j^\varepsilon$;} \\[1ex]
\bullet\,\,\,\, \,\,  &\mbox{$\supp \chi_j^\varepsilon$ is an interval  of length $3\varepsilon$ and with the same midpoint as $ \supp \pi_j^\varepsilon$, $|j|\leq N-1$;} \\[1ex]
\bullet\,\,\,\, \,\,  &\mbox{$\supp\chi_N^\varepsilon\subset [|x|\geq 1/\varepsilon-\varepsilon]$ and $\xi+ \supp \pi_{N}^\varepsilon\subset\supp \chi_{N}^\varepsilon$ for $|\xi|<\varepsilon.$}  
\end{align*} 
Each finite $\varepsilon$-localization family induces on $W^r_p(\mathbb{R})$ a norm    equivalent to the standard  norm.

\begin{lemma}\label{L:EN}
Let $\varepsilon>0$ and let  $\{\pi_j^\varepsilon\,:\, -N+1\leq j\leq N\}$ be a finite $\varepsilon$-localization family.  Given $p\in(1,\infty)$ and $r\in[0,\infty)$, there exists $c=c(\varepsilon,r,p)\in(0,1)$ such that
\[
c\|f\|_{W^r_p}\leq \sum_{j=-N+1}^N\|\pi_j^\varepsilon f\|_{W^r_p}\leq c^{-1}\|f\|_{W^r_p},\qquad f\in W^r_p(\mathbb{R}). 
\]
\end{lemma} 
\begin{proof}
We omit the elementary proof.
\end{proof}

The result established in the next lemma is used in an essential way  in the proof of Theorem~\ref{T:I1} and Theorem~\ref{T:I2} below.

\begin{lemma}\label{L:comp}
Let  $\varepsilon\in(0,1)$ be arbitrary (but fixed) and let $\{\pi_j^\varepsilon\,:\, -N+1\leq j\leq N\} $ and $\{\chi_j^\varepsilon\,:\, -N+1\leq j\leq N\}$ be as described above.
Furthermore let   $f\in W^s_p(\mathbb{R})$, $s\in(1+1/p,2)$. 
Given ${-N+1\leq j\leq N}$, the operator $K_j:L_p(\mathbb{R})\to L_p(\mathbb{R})$ defined by
\begin{align*}
K_j[\overline{\omega}]:=\chi_{j}^\varepsilon(\pi_{j}^\varepsilon\mathbb{A}(f)[ \overline{\omega}]-\mathbb{A}(f)[\pi_{j}^\varepsilon \overline{\omega}]),\qquad \overline{\omega}\in L_p(\mathbb{R}),
\end{align*}
is compact.
\end{lemma}
\begin{proof}  
According to the Riesz-Fr\'echet-Kolmogorov theorem, it suffices to show that
\begin{align}
&\sup_{\|\overline{\omega}\|_p\leq 1}\int_{[|x|>R]}|K_j[\overline{\omega}](x)|^p\, dx\underset{R\to\infty}\to0,\label{Prop1}\\[1ex]
&\sup_{\|\overline{\omega}\|_p\leq 1}\|\tau_\xi (K_j[\overline{\omega}])-K_j[\overline{\omega}]\|_p\underset{\xi\to0}\to0.\label{Prop2}
\end{align}

\noindent{\em Step 1.} For $|j|\leq N-1$ the  assertion  (\ref{Prop1})  is obvious.  Let now $j=N$.
Then it holds
\begin{align*}
&\hspace{-0.5cm}\Big(\int_{[|x|>R]}|K_j[\overline{\omega}](x)|^p\, dx\Big)^{1/p}\\[1ex]
&\leq\Big(\int_{[|x|>R]}\Big|\int_\mathbb{R} \frac{yf'(x)-\delta_{[x,y]} f}{1+\big(\delta_{[x,y]} f/y\big)^2}\frac{\delta_{[x,y]} \pi_j^\varepsilon}{y^2}\overline{\omega}(x-y)\, dy\Big|^{p}\, dx\Big)^{1/p}\\[1ex]
&\leq \Big(\int_{[|x|>R]}\Big|\int_{[|y|\leq 1]} \frac{yf'(x)-\delta_{[x,y]} f}{1+\big(\delta_{[x,y]} f/y\big)^2}\frac{\delta_{[x,y]} \pi_j^\varepsilon}{y^2}\overline{\omega}(x-y)\, dy\Big|^{p}\, dx\Big)^{1/p}\\[1ex]
&\hspace{0.45cm}+ \Big(\int_{[|x|>R]}\Big|\int_{[|y|>1]} \frac{yf'(x)-\delta_{[x,y]} f}{1+\big(\delta_{[x,y]} f/y\big)^2}
\frac{\delta_{[x,y]} \pi_j^\varepsilon}{y^2}\overline{\omega}(x-y)\, dy\Big|^{p}\, dx\Big)^{1/p}\\[1ex]
&=:T_1+T_2.
\end{align*}
 If $R$ is sufficiently large, then $\pi_j^\varepsilon(x)=1=\pi_j^\varepsilon(x-y)$ for all $|x|>R$ and $|y|\leq 1$, hence~ $T_1=0$.
 Concerning $T_2,$ it holds $T_2\leq T_{2a}+T_{2b}+T_{2c}, $
 where 
 \begin{align*}
 T_{2a}&:=\Big(\int_{[|x|>R]}|f'(x)|^p \Big|\int_{[|y|>1]}  
\frac{\delta_{[x,y]} \pi_j^\varepsilon}{y}\overline{\omega}(x-y)\, dy\Big|^{p}\, dx\Big)^{1/p}\leq C\|f'\|_{L_p([|x|>R])},\\[1ex]
T_{2b}&:=\Big(\int_{[|x|>R]}|f(x)|^p\Big|\int_{[|y|>1]}  
\frac{\delta_{[x,y]} \pi_j^\varepsilon}{y^2}\overline{\omega}(x-y)\, dy\Big|^{p}\, dx\Big)^{1/p}\leq C\|f\|_{L_p([|x|>R])},\\[1ex]
T_{2c}&:=\Big(\int_{[|x|>R]}\Big|\int_{[|y|>1]}  
\frac{\delta_{[x,y]} \pi_j^\varepsilon}{y^2}(f\overline{\omega})(x-y)\, dy\Big|^{p}\, dx\Big)^{1/p}.
 \end{align*}
 If $R$ is sufficiently large, then $\pi_j^\varepsilon(x)=1$ for all $|x|>R$.
  Taking into account that $\pi_j^\varepsilon(x-y)=1$ for all $|x-y|>1/\varepsilon+\varepsilon$, it follows that 
 $\delta_{[x,y]}\pi_j^\varepsilon=0$ for $|x|>R$ and $|x-y|> 1/\varepsilon+\varepsilon$, hence
  \begin{align*}
T_{2c}&=\Big(\int_{[|x|>R]}\Big|\int_{[|y|>R-1/\varepsilon-\varepsilon]}  
\frac{\delta_{[x,y]} \pi_j^\varepsilon}{y^2}(f\overline{\omega})(x-y)\, dy\Big|^{p}\, dx\Big)^{1/p}\\[1ex]
&\leq\int_{[|y|>R-1/\varepsilon-\varepsilon]}\frac{2}{y^2}\Big(\int_{[|x|>R]}|(f\overline{\omega})(x-y)|^p\, dx\Big)^{1/p}\, dy\\[1ex]
&\leq \frac{C}{R-1/\varepsilon-\varepsilon}\|f\|_\infty.
 \end{align*}
 These arguments show that (\ref{Prop1}) holds for all $-N+1\leq j\leq N$.\medskip
 
 \noindent{\em Step 2.} With respect to (\ref{Prop2})  note that 
$\|\tau_\xi (K_j[\overline{\omega}])-K_j[\overline{\omega}]\|_p\leq T_1+T_2+T_3$,  $-N+1\leq j\leq N,$ 
 where
 \begin{align*}
 T_1&:=\|\tau_\xi\chi_j^\varepsilon-\chi_j^\varepsilon\|_\infty\|\pi_{j}^\varepsilon\mathbb{A}(f)[ \overline{\omega}]-\mathbb{A}(f)[\pi_{j}^\varepsilon \overline{\omega}]\|_p\leq C\|\tau_\xi\chi_j^\varepsilon-\chi_j^\varepsilon\|_\infty\leq C|\xi|,\\[1ex]
T_2&:=\Big(\int_{\mathbb{R}}\Big|\int_\mathbb{R} \Big[\frac{ \tau_\xi f'(x) }{1+\big(\delta_{[x-\xi,y-\xi]} f/(y-\xi)\big)^2} \frac{  \delta_{[x-\xi,y-\xi]}  \pi_j^\varepsilon }{y-\xi}\\[1ex]
 &\hspace{2.25cm}-\frac{f'(x)}{1+\big(\delta_{[x,y]} f/y\big)^2} \frac{  \delta_{[x,y]}\pi_j^\varepsilon }{y}\Big]\overline{\omega}(x-y)\, dy\Big|^{p}\, dx\Big)^{1/p},\\[1ex]
 T_3&:=\Big(\int_{\mathbb{R}}\Big|\int_\mathbb{R} \Big[\frac{\delta_{[x-\xi,y-\xi]}  f}{1+\big(\delta_{[x-\xi,y-\xi]}  f/(y-\xi)\big)^2} \frac{  \delta_{[x-\xi,y-\xi]}  \pi_j^\varepsilon}{(y-\xi)^2}\\[1ex]
 &\hspace{2.25cm}-\frac{\delta_{[x,y]} f}{1+\big(\delta_{[x,y]} f/y\big)^2} \frac{  \delta_{[x,y]}\pi_j^\varepsilon }{y^2}\Big]\overline{\omega}(x-y)\, dy\Big|^{p}\, dx\Big)^{1/p}.
 \end{align*}
 Moreover, $T_2\leq  T_{2a}+T_{2b}, $ where, using H\"older's inequality, we have 
 \begin{align*}
T_{2a}&:=\Big(\int_{\mathbb{R}} |\tau_\xi f'(x)-f'(x)|^p\Big(\int_\mathbb{R} \Big| \frac{ \delta_{[x-\xi,y-\xi]}  \pi_j^\varepsilon}{y-\xi}\overline{\omega}(x-y)\Big|\, dy\Big)^{p}\, dx\Big)^{1/p}\leq C\|\tau_\xi f'-f'\|_p,
\end{align*}
uniformly for $|\xi|<1/2$, and
\begin{align*}
 T_{2b}&:=\Big(\int_{\mathbb{R}}|f'(x)|^p\Big|\int_\mathbb{R} \Big[\frac{ \delta_{[x-\xi,y-\xi]}\pi_j^\varepsilon/(y-\xi)}{1+\big(\delta_{[x-\xi,y-\xi]}f/(y-\xi)\big)^2}
 -\frac{\delta_{[x,y]}\pi_j^\varepsilon/y}{1+\big(\delta_{[x,y]} f/y\big)^2} \Big]\overline{\omega}(x-y)\, dy\Big|^{p} dx\Big)^{1/p}\!.
 \end{align*}
 Taking into account that for $|\xi|<1/2$ it holds that 
 \begin{align*}
  \Big|\frac{ \delta_{[x-\xi,y-\xi]}\pi_j^\varepsilon/(y-\xi)}{1+\big(\delta_{[x-\xi,y-\xi]}f/(y-\xi)\big)^2}
 -\frac{\delta_{[x,y]}\pi_j^\varepsilon/y}{1+\big(\delta_{[x,y]} f/y\big)^2} \Big| \leq C |\xi|^{s-1-1/p}\Big({\bf 1}_{[|y|\leq 1]}(y)+{\bf 1}_{[|y|> 1]}(y)\frac{1}{|y|}\Big),
 \end{align*}
 H\"older's inequality leads, for $|\xi|<1/2,$  to
 \[
T_{2b}\leq C|\xi|^{ s-1-1/p }.
 \]
 
 It remains to show that $T_3\to 0$ for $\xi\to0$.
Since
 \begin{align*}
 &\hspace{-0,5cm}\Big|\frac{\delta_{[x-\xi,y-\xi]}  f}{1+\big(\delta_{[x-\xi,y-\xi]}  f/(y-\xi)\big)^2} \frac{  \delta_{[x-\xi,y-\xi]}  \pi_j^\varepsilon}{(y-\xi)^2}
 -\frac{\delta_{[x,y]} f}{1+\big(\delta_{[x,y]} f/y\big)^2} \frac{  \delta_{[x,y]}\pi_j^\varepsilon }{y^2}\Big|\\[1ex]
 &\leq  C |\xi|^{s-1-1/p}\Big({\bf 1}_{[|y|\leq 1]}(y)+{\bf 1}_{[|y|> 1]}(y)\frac{1}{|y|^2}\Big)
 \end{align*}
 for all $|\xi|<1/2$, Minkowski's inequality   yields
 \begin{align*}
 T_{2}&\leq C|\xi|^{s-1-1/p}.
 \end{align*}
 Hence, (\ref{Prop2}) holds true for all  $-N+1\leq j\leq N$ and  the proof is complete.
\end{proof}

The next result is a key ingredient in the proof of Theorem~\ref{T:I1} and  Theorem~\ref{T:I2} below.

\begin{lemma}\label{L:comm'}
Let  $\varepsilon\in(0,1)$ and $\{\pi_j^\varepsilon\,:\, -N+1\leq j\leq N\} $ and $\{\chi_j^\varepsilon\,:\, -N+1\leq j\leq N\}$ be as described above.
Let further   $f\in W^s_p(\mathbb{R})$, $s\in(1+1/p,2)$, and pick $q\in[ p,\infty).$ 
Then, for each ${-N+1\leq j\leq N},$ it holds that
\begin{align*}
\|\chi_j^\varepsilon\mathbb{A}(f)\chi_j^\varepsilon\|_{\mathcal{L}(L_q(\mathbb{R}))}<1,
\end{align*} 
provided that $\varepsilon$ is   sufficiently small.
\end{lemma}
\begin{proof}
We first establish the claim for  $|j|\leq N-1$.
Using Minkowski's inequality and the embedding $W^s_p(\mathbb{R})\hookrightarrow {\rm C}^{s-1/p}(\mathbb{R}),$ it follows that
\begin{align*}
\|\chi_j^\varepsilon\mathbb{A}(f)[\chi_j^\varepsilon\overline{\omega}]\|_q
&=\Big(\int_\mathbb{R}\Big|\int_\mathbb{R} \chi_j^\varepsilon(x)\frac{yf'(x)-\delta_{[x,y]} f}{1+\big(\delta_{[x,y]} f/y\big)^2}\frac{(\chi_j^\varepsilon\overline{\omega})(x-y)}{y^2}\, dy\Big|^{q}\, dx\Big)^{1/q}\\[1ex]
&\leq [f']_{s-1-1/p}\int_\mathbb{R}|y|^{s-2-1/p}\Big(\int_{\supp \chi_j^\varepsilon}  |(\chi_j^\varepsilon\overline{\omega})(x-y)|^q \, dx\Big)^{1/q}\, dy\\[1ex]
&\leq [f']_{s-1-1/p}\|\overline{\omega}\|_q\int_{\supp \chi_j^\varepsilon-\supp \chi_j^\varepsilon}|y|^{s-2-1/p} \, dy\\[1ex]
&\leq C\|f\|_{W^s_p}\varepsilon^{s-1-1/p}\|\overline{\omega}\|_q.
\end{align*}
In the case $j=N$ we have   $\|\chi_j^\varepsilon\mathbb{A}(f)[\chi_j^\varepsilon\overline{\omega}]\|_q\leq T_1+T_2$, where 
\begin{align*}
T_1&:=\Big(\int_{\mathbb{R}}\Big|\int_{[|y|\leq1]} \chi_j^\varepsilon(x)\frac{yf'(x)-\delta_{[x,y]} f}{1+\big(\delta_{[x,y]} f/y\big)^2}\frac{(\chi_j^\varepsilon\overline{\omega})(x-y)}{y^2}\, dy\Big|^{q}\, dx\Big)^{1/q},\\[1ex]
T_2&:=\Big(\int_{\mathbb{R}}\Big|\int_{[|y|>1]} \chi_j^\varepsilon(x)\frac{yf'(x)-\delta_{[x,y]} f}{1+\big(\delta_{[x,y]} f/y\big)^2}
\frac{(\chi_j^\varepsilon\overline{\omega})(x-y)}{y^2}\, dy\Big|^{q}\, dx\Big)^{1/q}.
\end{align*}
Using the mean value theorem, we have
\[
|yf'(x)-\delta_{[x,y]} f|\leq 2 \|f'\|_{L_\infty([|x|>1/\varepsilon- 2])}^{1/2}[f']_{s-1-1/p}^{1/2}|y|^{s/2+1/2-1/2p},
\]
and herewith the term $T_1$ can be estimated as follows
\begin{align*}
T_1&\leq 2\|f'\|_{L_\infty([|x|>1/\varepsilon- 2])}^{1/2}[f']_{s-1-1/p}^{1/2}\Big(\int_\mathbb{R}\Big|\int_{[|y|\leq1]} \chi_j^\varepsilon(x) 
\frac{(\chi_j^\varepsilon\overline{\omega})(x-y)}{y^{3/2+1/2p-s/2}}\, dy\Big|^{q}\, dx\Big)^{1/q}\\[1ex]
&\leq 2\|f'\|_{L_\infty([|x|>1/\varepsilon-2])}^{1/2}[f']_{s-1-1/p}^{1/2}\int_{[|y|\leq1]}\frac{1}{y^{3/2+1/2p-s/2}}\Big(\int_{\mathbb{R}}| (\chi_j^\varepsilon\overline{\omega})(x-y)|^q \, dx\Big)^{1/q}\, dy\\[1ex]
&\leq C\|f'\|_{L_\infty([|x|>1/\varepsilon-2])}^{1/2}\|f\|_{W^s_p}^{1/2}\|\overline{\omega}\|_q.
\end{align*}
Furthermore, $T_2\leq  T_{2a}+T_{2b}+T_{2c} $ where 
\begin{align*}
  T_{2a}&:=\Big(\int_{\mathbb{R}} |(f'\chi_j^\varepsilon)(x)|^q\Big|\int_{[|y|>1]}  
\frac{(\chi_j^\varepsilon\overline{\omega})(x-y)}{y }\, dy\Big|^{q}\, dx\Big)^{1/q}\\[1ex]
&\,\leq C\|f'\|_{L_p([|x|>1/\varepsilon-1])}^{p/q}\|f'\|_{L_\infty([|x|>1/\varepsilon-1])}^{(q-p)/q}\|\overline{\omega}\|_q,
\end{align*}
and
\begin{align*}
T_{2b}&:= \Big(\int_{\mathbb{R}}\Big|\int_{[|y|>1]} (f\chi_j^\varepsilon)(x) 
\frac{(\chi_j^\varepsilon\overline{\omega})(x-y)}{y^2}\, dy\Big|^{q}\, dx\Big)^{1/q}\\[1ex]
&\,\leq \int_{[|y|>1]}\frac{1}{y^2}\Big(\int_{\mathbb{R}} |(f\chi_j^\varepsilon)(x) (\chi_j^\varepsilon\overline{\omega})(x-y)|^q \, dx\Big)^{1/q}\, dy\\[1ex]
&\, \leq C\|f\|_{L_\infty([|x|>1/\varepsilon-1])}\|\overline{\omega}\|_q,
\end{align*}
and finally
\begin{align*}
T_{2c}&:= \Big(\int_{\mathbb{R}}\Big|\int_{[|y|>1]} \chi_j^\varepsilon(x) 
\frac{(f\chi_j^\varepsilon\overline{\omega})(x-y)}{y^2}\, dy\Big|^{q}\, dx\Big)^{1/q}\\[1ex]
&\,\leq \int_{[|y|>1]}\frac{1}{y^2}\Big(\int_{\mathbb{R}} |  (f\chi_j^\varepsilon\overline{\omega})(x-y)|^q \, dx\Big)^{1/q}\, dy\\[1ex]
&\,\leq C\|f\|_{L_\infty([|x|>1/\varepsilon-1])}\|\overline{\omega}\|_q.
\end{align*}
Gathering these estimates and observing that $  f^{(k)}(x)\to0$ for $|x|\to\infty$ and  $k=0,\,1$, we conclude that the claim holds true.
\end{proof}

 We are now in a position to establish the aforementioned invertibility result in $\mathcal{L}(L_p(\mathbb{R})) $  for $p\in(1,2]$.

\begin{thm}\label{T:I1}  
Let $p\in(1,2]$ and $s\in(1+1/p,2)$. Given  $f\in W^s_p(\mathbb{R})$  and  $\lambda\in\mathbb{R}$ with $|\lambda|\geq1$ it holds  
\begin{align*} 
 \lambda-\mathbb{A}(f)\in{\rm Isom\,}(L_p(\mathbb{R})).
\end{align*}
\end{thm}
\begin{proof} The claim in the particular case $p=2$ has been established in \cite[Theorem 3.5]{MBV18}.
Let now $p\in(1,2)$, $f\in W^s_p(\mathbb{R}),$  and  $\lambda\in\mathbb{R}$ with $|\lambda|\geq 1 $  be given. \medskip

\noindent{\em Step 1.} We  first prove that  $\lambda-\mathbb{A}(f)$ is injective.
Let thus $\overline{\omega}\in L_p(\mathbb{R}) $ satisfy   $(\lambda-\mathbb{A}(f))[\overline{\omega}]=0$. 
Given $\varepsilon>0$, this equation is equivalent to the following system of equations 
\begin{equation}\label{REL:LER}
(\lambda-\chi_j^\varepsilon\mathbb{A}(f)\chi_j^\varepsilon)[\pi_j^\varepsilon \overline{\omega}]=K_j[\overline{\omega}]\qquad \mbox{for $-N+1\leq j\leq N,$}
\end{equation} 
where $K_j,$  $-N+1\leq j\leq N,$ are the operators introduced in Lemma \ref{L:comp}.
 Since $p\in(1,2)$, in view of  Lemma \ref{L:comm'} we may choose $\varepsilon>0$  such that 
 $$(\lambda-\chi_j^\varepsilon\mathbb{A}(f)\chi_j^\varepsilon)\in {\rm Isom}(L_2(\mathbb{R}))\cap {\rm Isom}(L_p(\mathbb{R}))\qquad\mbox{for all $-N+1\leq j\leq N$.}$$  
 As $\overline{\omega}\in L_p(\mathbb{R}),$ the right hand-side $K_j[\overline{\omega}]$ of (\ref{REL:LER}) belongs to $L_p(\mathbb{R})$.
 Moreover,  using once more the fact that $p\in(1,2)$ together with the $L_\infty$-bound 
 \begin{align*}
 |K_j[\overline{\omega}](x)|&\leq \int_\mathbb{R} \Big| \frac{yf'(x)-\delta_{[x,y]} f}{1+\big(\delta_{[x,y]} f/y\big)^2} \frac{  \delta_{[x,y]}\pi_j^\varepsilon }{y^2} \overline{\omega}(x-y)\Big|\, dy\\[1ex]
&\leq C\int_\mathbb{R} \Big| \Big({\bf 1}_{[|y|\leq 1]}(y)+{\bf 1}_{[|y|> 1]}(y)\frac{1}{|y|}\Big)\overline{\omega}(x-y)\Big|\, dy\\[1ex]
&\leq C\|\overline{\omega}\|_p,\quad\mbox{$x\in\mathbb{R},$}
 \end{align*}
 it follows that $K_j[\overline{\omega}]\in L_2(\mathbb{R})$ for all $-N+1\leq j\leq N$.
 Invoking  (\ref{REL:LER}), we then get
 \[
\pi_j^\varepsilon\overline{\omega}= (\lambda-\chi_j^\varepsilon\mathbb{A}(f)\chi_j^\varepsilon)^{-1}[K_j[\overline{\omega}]]\in L_p(\mathbb{R})\cap L_2(\mathbb{R}),\qquad \mbox{$-N+1\leq j\leq N,$}
 \]
 and therewith $\overline{\omega}\in L_2(\mathbb{R})$.
 Since $ \lambda- \mathbb{A}(f)\in{\rm Isom}(L_2(\mathbb{R}))$ for all $\lambda\in\mathbb{R}$ with $|\lambda|\geq 1$,   we conclude that the equation $(\lambda- \mathbb{A}(f))[\overline{\omega}]=0$ in $L_p(\mathbb{R})$ has only
 the trivial solution.\medskip

 \noindent{Step 2.} We now prove there exists $C>0$ with the property  
 \begin{align}\label{DE:EST12}
 \|(\lambda-\mathbb{A}(f))[\overline{\omega}]\|_p\geq C\|\overline{\omega}\|_p \qquad \mbox{for all $\overline{\omega}\in L_p(\mathbb{R})$ and $\lambda\in\mathbb{R}$ with $|\lambda|\geq1$}.
 \end{align}
Indeed, assuming   the claim is false, we may find a    sequence $(\overline{\omega}_n)_n\subset L_p(\mathbb{R})$ and a 
bounded sequence ${(\lambda_n)_n\subset\mathbb{R}}$ with the properties   $|\lambda_n|\geq1,$ $\|\overline{\omega}_n\|_p=1$ for all $n\in\mathbb{N}$, and such that
 ${(\lambda_n-\mathbb{A}(f))[\overline{\omega}_n]=:\varphi_n\to0}$ in $L_p(\mathbb{R})$.
 After possibly extracting a subsequence we may assume that $\lambda_n\to\lambda$ in $\mathbb{R}$ and $\overline{\omega}_n\rightharpoonup \overline{\omega}$ in $L_p(\mathbb{R}).$
 In virtue of (\ref{PRO1''}) it holds that 
 \begin{align*} 
\langle(\lambda_n-\mathbb{A}(f))[\overline{\omega}_n]|h\rangle_{L_2}=\langle\overline{\omega}_n|(\lambda_n-(\mathbb{A}(f))^*)[h]\rangle_{L_2}
\end{align*}
for all $n\in\mathbb{N}$ and $h\in L_{p'}(\mathbb{R})$.
Passing to the limit $n\to\infty$ in the previous equation it results that $\langle(\lambda-\mathbb{A}(f))[\overline{\omega}]|h\rangle_{L_2}=0$   for all
$h\in L_{p'}(\mathbb{R})$. 
Since $\lambda-\mathbb{A}(f)$ is injective, we get~$\overline{\omega}=0$. 

Let now $\varepsilon>0$ be chosen such that $(\lambda_n-\chi_j^\varepsilon\mathbb{A}(f)\chi_j^\varepsilon),\, (\lambda-\chi_j^\varepsilon\mathbb{A}(f)\chi_j^\varepsilon)\in  {\rm Isom}(L_p(\mathbb{R}))$  for all~$-N+1\leq j\leq N$ and all $n\in\mathbb{N}$.
Such an $\varepsilon$ exists in virtue of Lemma \ref{L:comm'} and of the fact that $|\lambda_n|\geq 1$ for all $n\in\mathbb{N}$.
Because of
 \[
 1=\|\overline{\omega}_n\|_p\leq \sum_{j=-N+1}^N\|\pi_j^\varepsilon\overline{\omega}_n\|_p,
 \]
there exists ${-N+1\leq j_*\leq N}$ and a subsequence of $(\overline{\omega}_n)_n$ (not relabeled)    such that 
 \begin{align}\label{contr}
 \|\pi_{j_*}^\varepsilon\overline{\omega}_n\|_p\geq (2N)^{-1} \qquad\mbox{for all $n\in\mathbb{N}$.}
 \end{align}
From $(\lambda_n-\mathbb{A}(f))[\overline{\omega}_n]=\varphi_n$ it then follows 
 \begin{equation}\label{nice}
 \begin{aligned}
 \pi_{j_*}^\varepsilon \overline{\omega}_n&=(\lambda_n-\chi_{j_*}^\varepsilon\mathbb{A}(f)\chi_{j_*}^\varepsilon)^{-1} [K_{j_*}[ \overline{\omega}_n]  ]+(\lambda_n-\chi_{j_*}^\varepsilon\mathbb{A}(f)\chi_{j_*}^\varepsilon)^{-1}[\pi_{j_*}^\varepsilon\varphi_n] 
\end{aligned}  
\end{equation} 
for all $n\in\mathbb{N}$. 
Recalling Lemma \ref{L:comp}, we obtain    $K_{j^*}[\overline{\omega}_n]\to0$ in $L_p(\mathbb{R})$.
Furthermore, taking into account that $\lambda_n-\chi_{j_*}^\varepsilon\mathbb{A}(f)\chi_{j_*}^\varepsilon \to \lambda-\chi_{j_*}^\varepsilon\mathbb{A}(f)\chi_{j_*}^\varepsilon$  in $\mathcal{L}(L_p(\mathbb{R}))$, we deduce from (\ref{nice}) that~$\pi_{j_*}^\varepsilon \overline{\omega}_n\to 0$ in $L_p(\mathbb{R})$, which contradicts (\ref{contr}).

 We have thus established the validity of   (\ref{DE:EST12}). 
 Since $\lambda-\mathbb{A}(f)\in{\rm Isom}(L_p(\mathbb{R}))$ for $|\lambda| $  sufficiently large,  the   method of continuity, cf. \cite[Proposition I.1.1.1]{Am95},  leads us to the conclusion that
 ${\lambda-\mathbb{A}(f)\in{\rm Isom}(L_p(\mathbb{R}))}$ for all $|\lambda|\geq 1.$
 This completes the proof. 
 \end{proof}

We now consider the case $p\in(2,\infty)$. 
From the proof of Theorem \ref{T:I1} we may infer that  if ${\lambda-\mathbb{A}(f)\in\mathcal{L}(L_p(\mathbb{R}))}$ is injective for all $\lambda\in\mathbb{R}$ with  $|\lambda|\geq1$, then $\lambda-\mathbb{A}(f)\in{\rm Isom}(L_p(\mathbb{R}))$ for all such $\lambda$.
The arguments used to establish the injectivity property of   $\lambda-\mathbb{A}(f)$ in the case $p\in(1,2]$ however do not work for $p\in(2,\infty)$ and therefore a new strategy is needed. 
 
\begin{thm}\label{T:I2}  
Let $p\in(2,\infty)$ and $s\in(1+1/p,2)$. 
Given  $f\in W^s_p(\mathbb{R})$  and $\lambda\in\mathbb{R}$ with  $|\lambda|\geq1$, it holds 
\begin{align*}
\lambda-\mathbb{A}(f)\in{\rm Isom}(L_p(\mathbb{R})).
\end{align*}
\end{thm}
\begin{proof}
To each  $\varepsilon\in(0,1)$ we associate a function $a_\varepsilon\in {\rm C}^\infty(\mathbb{R},[0,1]) $ with the   properties that 
${a_\varepsilon(x)=0}$ for $|x|<\varepsilon^{-1},$   $a_\varepsilon(x)=1$ for $|x|>\varepsilon^{-1}+1$, and $|a_\varepsilon'|\leq 2$. 
Recalling (\ref{FormulaA}), it is suitable to write
\begin{align*}
\mathbb{A}(f)=\mathbb{A}_{1,\varepsilon}+\mathbb{A}_{2,\varepsilon},
\end{align*}
where
\begin{align*}
&\pi\mathbb{A}_{1,\varepsilon}:=(a_\varepsilon f)'B_{0,1}(f)-B_{1,1}(f)[a_\varepsilon f,\cdot],\\[1ex]
&\pi\mathbb{A}_{2,\varepsilon}:=[(1-a_\varepsilon)f]'B_{0,1}(f)-B_{1,1}(f)[(1-a_\varepsilon) f,\cdot].
\end{align*}
According to Lemma \ref{L:MP1}~(i), it holds that $\mathbb{A}_{j,\varepsilon}\in\mathcal{L}(L_q(\mathbb{R}))$ for all $1<q<\infty$ and
\begin{align*}
\|\mathbb{A}_{1,\varepsilon}\|_{\mathcal{L}(L_q(\mathbb{R}))} \leq C \|(a_\varepsilon f)'\|_{\infty}\leq C  \|  f\|_{W^1_\infty([|x|>\varepsilon^{-1}])} \underset{\varepsilon\to0}\to0,
\end{align*}
since   $f^{(k)}\to0$ for $|x|\to\infty,\, k=0,\, 1$. 
Hence, if  $\varepsilon$ is sufficiently  small, then $$\lambda-\mathbb{A}_{1,\varepsilon}\in{\rm Isom}(L_p(\mathbb{R}))\cap {\rm Isom}(L_2(\mathbb{R}))$$ for all $|\lambda|\geq 1$.

Let now $\lambda\in\mathbb{R}$ with  $|\lambda|\geq1$ and $\overline{\omega}\in L_p(\mathbb{R})$ satisfy
$(\lambda-\mathbb{A}(f))[\overline{\omega}]=0,$
or equivalently
\begin{align*}
(\lambda-\mathbb{A}_{1,\varepsilon})[\overline{\omega}]=\mathbb{A}_{2,\varepsilon}[\overline{\omega}].
\end{align*}
Assuming that 
\begin{align}\label{assum}
\mathbb{A}_{2,\varepsilon}[\overline{\omega}]\in L_2(\mathbb{R})\cap L_p(\mathbb{R})\qquad\mbox{for all $\varepsilon$ that are sufficiently small},
\end{align}
the previous equality together with $\lambda-\mathbb{A}_{1,\varepsilon}\in{\rm Isom}(L_p(\mathbb{R}))\cap {\rm Isom}(L_2(\mathbb{R}))$  yields $w\in L_2(\mathbb{R})$. 
Recalling that $\lambda-\mathbb{A}(f) $ is an $L_2(\mathbb{R})$-isomorphism, we may then conclude that $\overline{\omega}=0$.

It thus remains to establish (\ref{assum}).
To this end we write
 \begin{align*}
 \mathbb{A}_{2,\varepsilon}[\overline{\omega}]= \mathbb{A}_{2,\varepsilon}[a_{\varepsilon^2}\overline{\omega}]+\mathbb{A}_{2,\varepsilon}[(1-a_{\varepsilon^2})\overline{\omega}].
 \end{align*}
In view of $p\in(2,\infty)$ it holds that $(1-a_{\varepsilon^2})\overline{\omega}\in L_2(\mathbb{R})\cap L_p(\mathbb{R})$ and therefore we obtain that~${\mathbb{A}_{2,\varepsilon}[(1-a_{\varepsilon^2})\overline{\omega}]\in L_2(\mathbb{R})\cap L_p(\mathbb{R})}$.
Letting $g_\varepsilon:=(1-a_\varepsilon)f$, it holds that
\begin{align*}
\|\mathbb{A}_{2,\varepsilon}[a_{\varepsilon^2}\overline{\omega}]\|_2&\leq\Big(\int_{\mathbb{R}}\Big|\int_\mathbb{R} \frac{yg_\varepsilon'(x)-\delta_{[x,y]} g_\varepsilon}{y^2+(\delta_{[x,y]} f)^2} (a_{\varepsilon^2}\overline{\omega})(x-y)\, dy\Big|^{2}\, dx\Big)^{1/2}\\[1ex]
&\leq \Big(\int_{\mathbb{R}}\Big|\int_{[|y|\leq1]} \frac{yg_\varepsilon'(x)-\delta_{[x,y]} g_\varepsilon}{y^2+(\delta_{[x,y]} f)^2} (a_{\varepsilon^2}\overline{\omega})(x-y)\, dy\Big|^{2}\, dx\Big)^{1/2}\\[1ex]
&\hspace{0.45cm}+\Big(\int_{\mathbb{R}}\Big|\int_{[|y|>1]} \frac{yg_\varepsilon'(x)-\delta_{[x,y]} g_\varepsilon}{y^2+(\delta_{[x,y]} f)^2} (a_{\varepsilon^2}\overline{\omega})(x-y)\, dy\Big|^{2}\, dx\Big)^{1/2}\\[1ex]
&=:T_1+T_2.
\end{align*}
If $\varepsilon$ is sufficiently small, then $T_1=0$. Indeed, let $x\in\mathbb{R}$.  Then   $|x|<\varepsilon^{-2}-1$ or $|x|>\varepsilon^{-1}+2.$
In the  case when $|x|<\varepsilon^{-2}-1,$ it follows that  $|x-y|<\varepsilon^{-2}$ and therewith $a_{\varepsilon^2}(x-y)=0.$
In the other case when $|x|>\varepsilon^{-1}+2 $  it holds $|x-y|>\varepsilon^{-1}+1$ and 
$g_\varepsilon'(x)=g_\varepsilon(x)=g_\varepsilon(x-y)=0.$
Concerning $T_2$, we first note that $g_\varepsilon a_{\varepsilon^2} =0$.
 Given  $|x|>\varepsilon^{-1}+1$, we get  $g_\varepsilon'(x)=g_\varepsilon(x)=0,$  and   H\"older's inequality leads to
\begin{align*}
T_2&=\Big(\int_{[|x|\leq \varepsilon^{-1}+1]}\Big|\int_{[|y|>1]} \frac{yg_\varepsilon'(x)-\delta_{[x,y]} g_\varepsilon}{y^2+(\delta_{[x,y]} f)^2} (a_{\varepsilon^2}\overline{\omega})(x-y)\, dy\Big|^{2}\, dx\Big)^{1/2}\\[1ex]
&\leq C \|g_\varepsilon'\|_\infty\Big(\int_{[|x|\leq \varepsilon^{-1}+1]} \Big(\int_{[|y|>1]} \frac{1}{|y|}  |a_{\varepsilon^2}\overline{\omega}|(x-y)\, dy\Big)^{2}\, dx\Big)^{1/2}\\[1ex]
&\leq C(\varepsilon)\|g_\varepsilon'\|_\infty\|\overline{\omega}\|_{L_p}.
\end{align*}
This proves (\ref{assum}) and  the injectivity of ${\lambda-\mathbb{A}(f)\in\mathcal{L}(L_p(\mathbb{R}))}$.
\end{proof}

Finally, we establish the invertibility of $\lambda-\mathbb{A}(f)$, $\lambda\in\mathbb{R}$ with  $|\lambda|\geq1$, in $\mathcal{L}(W^{s-1}_p(\mathbb{R}))$.

\begin{thm}\label{T:INV}       
Let $p\in(1,\infty),$ $s\in(1+1/p,2),$ and $f\in W^s_p(\mathbb{R})$. 
Given $\lambda\in\mathbb{R}$ with  $|\lambda|\geq1,$ it holds 
\begin{align*}
\lambda-\mathbb{A}(f)\in{\rm Isom}(W^{s-1}_p(\mathbb{R})).
\end{align*}
\end{thm}
\begin{proof} 
Given $\overline{\omega}\in W^{s-1}_p(\mathbb{R})$ and $\lambda\in\mathbb{R}$ with  $|\lambda|\geq1$, let  $\varphi:=(\lambda-\mathbb{A}(f))[\overline{\omega}]$. 
Lemma~\ref{L:REG} then yields $\varphi\in W^{s-1}_p(\mathbb{R})$. 
Recalling Theorem~\ref{T:I1} and Theorem~\ref{T:I2}, it holds that 
\begin{align*}
[\overline{\omega}]^p_{W^{s-1}_p}&=\int_{\mathbb{R}}\frac{\|\overline{\omega}-\tau_{\xi}\overline{\omega}\|_p^p}{|\xi|^{1+(s-1)p}}\, d{\xi}\leq C\int_{\mathbb{R}}\frac{\|(\lambda-\mathbb{A}(\tau_\xi f))[\overline{\omega}-\tau_{\xi} \overline{\omega} ]\|_p^p}{|\xi|^{1+(s-1)p}}\, d{\xi}\\[1ex]
&\leq 2^pC\int_{\mathbb{R}}\frac{\|\varphi-\tau_{\xi}\varphi\|_p^p}{|\xi|^{1+(s-1)p}}\, d{\xi}+2^pC\int_{\mathbb{R}}\frac{\|(\mathbb{A}(f) -\mathbb{A}(\tau_{\xi}f))[\overline{\omega}])\|_p^p}{|\xi|^{1+(s-1)p}}\, d{\xi}\\[1ex]
&\leq 2^pC[\varphi]^p_{W^{s-1}_p}+2^pC\int_{\mathbb{R}}\frac{\|(\mathbb{A}(f) -\mathbb{A}(\tau_{\xi}f))[\overline{\omega}])\|_p^p}{|\xi|^{1+(s-1)p}}\, d{\xi},
\end{align*}
with $C\geq\max_{[|\lambda|\geq 1]\cap\mathbb{R}}\|(\lambda-\mathbb{A}(f))^{-1}\|_{\mathcal{L}(L_p(\mathbb{R}))}^p$.
Recalling (\ref{FormulaA}), we further compute
\begin{align*}
\|\mathbb{A}(f) -\mathbb{A}(\tau_{\xi}f))[\overline{\omega}]\|_p&\leq \|f'B_{0,1}(f)[\overline{\omega}]-\tau_{\xi}f'B_{0,1}(\tau_{\xi}f)[\overline{\omega}]\|_p\\[1ex]
&\hspace{0,45cm}+\|B_{1,1}(f)[f,\overline{\omega}]-B_{1,1}(\tau_{\xi}f)[\tau_{\xi}f,\overline{\omega}]\|_p,
\end{align*}
 and the relation (\ref{spr3}), Lemma \ref{L:MP1}~(ii) (with $\tau=s/2-1/2+1/2p\in(1/p,1) $ and $r=s$), and~(\ref{REF1}) (with $r=s/2+1/2+1/2p\in(1+1/p,2))$ yield
 \begin{align*}
&\hspace{-0.5cm}\|f'B_{0,1}(f)[\overline{\omega}]-\tau_{\xi}f'B_{0,1}(\tau_{\xi}f)[\overline{\omega}]\|_p\\[1ex]
&\leq \|f'-\tau_{\xi}f'\|_p\|B_{0,1}(f)[\overline{\omega}]\|_\infty+\|f'\|_\infty\|B_{2,2}(f,\tau_{\xi}f)[f+\tau_{\xi}f, f-\tau_{\xi}f,\overline{\omega}]\|_p\\[1ex]
&\leq C\|f'-\tau_{\xi}f'\|_p\|\overline{\omega}\|_{W^{s/2-1/2+1/2p}_p}, 
\end{align*}
where $C=C(\|f\|_{W^s_p})$, and by similar arguments
 \begin{align*}
&\hspace{-0.5cm}\|B_{1,1}(f)[f,\overline{\omega}]-B_{1,1}(\tau_{\xi}f)[\tau_{\xi}f,\overline{\omega}]\|_p\\[1ex]
&\leq \|B_{1,1}(f)[f-\tau_{\xi}f,\overline{\omega}]\|_p+\|B_{3,2}(f,\tau_{\xi}f)[f+\tau_{\xi}f, f-\tau_{\xi}f,\tau_{\xi}f,\overline{\omega}]\|_p\\[1ex]
&\leq C\|f'-\tau_{\xi}f'\|_{p}\|\overline{\omega}\|_{W^{s/2-1/2+1/2p}_p}.
\end{align*}
The latter estimates together with Theorem~\ref{T:I1} and Theorem~\ref{T:I2} imply there exists a constant $C_0$ (which depends only on $f$) such that 
\begin{align}\label{L:ES}
\|\overline{\omega}\|_{W^{s-1}_p}\leq C_0(\|(\lambda-\mathbb{A}(f))[\overline{\omega}]\|_{W^{s-1}_p}+\|\overline{\omega}\|_{W^{s/2-1/2+1/2p}_p}),\qquad \overline{\omega}\in W^{s-1}_p(\mathbb{R}).
\end{align} 
Using the interpolation property 
 \begin{align}\label{IP}
W^{(1-\theta)s_1+\theta s_2}_p(\mathbb{R})=(W^{s_1}_p(\mathbb{R}), W^{s_2}_p(\mathbb{R}))_{\theta,p}, \quad 0\leq s_1<s_1<\infty, \, (1-\theta)s_1+\theta s_2\not\in\mathbb{N},
 \end{align}
where $(\cdot,\cdot)_{\theta,p}$, $\theta\in(0,1), $ denotes the real interpolation functor of exponent $\theta$ and parameter~$p$, in the particular case $s_1=0,$ $s_2=s-1$, and  $\theta:=(s-1)^{-1}(s/2-1/2+1/2p)$,
it follows from Young's inequality that
 \[\|\overline{\omega}\|_{W^{s/2-1/2+1/2p}_p}\leq \frac{1}{2C_0}\|\overline{\omega}\|_{W^{s-1}_p}+C\|\overline{\omega}\|_p,\qquad \overline{\omega}\in W^{s-1}_p(\mathbb{R}).\]
This property combined with (\ref{L:ES}),  Theorem~\ref{T:I1}, and Theorem~\ref{T:I2} yields 
\begin{align*}
\|\overline{\omega}\|_{W^{s-1}_p}\leq C \|(\lambda-\mathbb{A}(f))[\overline{\omega}]\|_{W^{s-1}_p}, \quad\mbox{$\overline{\omega}\in W^{s-1}_p(\mathbb{R})$ and $\lambda\in\mathbb{R}$ with  $|\lambda|\geq1.$}
\end{align*}
The method of continuity \cite[Proposition I.1.1.1]{Am95} leads  now to the desired conclusion.
 \end{proof}

\section{The  abstract evolution problem}\label{Sec:4}

In this section we first use the results of Section \ref{Sec:3} to  formulate (\ref{P'})   as an evolution problem in $W^{s-1}_p(\mathbb{R})$ with  $f$ as the only unknown (see (\ref{NNEP})).  
Subsequently, we show that the Rayleigh-Taylor condition identifies a domain of parabolicity for (\ref{NNEP}), cf. Theorem \ref{T:AG}. 

Observing  that the Atwood number $a_\mu$ satisfies $|a_\mu|<1$, Theorem~\ref{T:INV} ensures that, for each~$f\in W^s_p(\mathbb{R})$, the equation $(\ref{P'})_2$ has a unique solution
\begin{align}\label{SO}
\overline{\omega}(f):=-C_{\Theta}(1+a_\mu\mathbb{A}(f))^{-1}[f'].
\end{align}
Moreover, Lemma \ref{L:REG} yields
\begin{align}\label{RSO}
[f\mapsto\overline{\omega}(f)]\in {\rm C}^\omega(W^s_p(\mathbb{R}), W^{s-1}_p(\mathbb{R})).
\end{align} 
We can thus reformulate  the system (\ref{P'}) as the abstract  evolution problem
\begin{align}\label{NNEP}
\frac{d f}{dt}=\Phi(f),\,\, t\geq0,\quad f(0)=f_0,
\end{align}
where the (fully) nonlinear and nonlocal  operator $\Phi:W^s_p(\mathbb{R})\to W^{s-1}_p(\mathbb{R})$ is defined by
\begin{align*}
\Phi(f):=\mathbb{B}(f)[\overline{\omega}(f)].
\end{align*}
In virtue of  (\ref{PRO1b}) and (\ref{RSO})  it holds
\begin{equation}\label{REG}
\Phi\in {\rm C}^\omega(W^s_p(\mathbb{R}), W^{s-1}_p(\mathbb{R})).
\end{equation}
It is important to point out that the operator $\Phi$ is fully nonlinear as the definition of the  function $\overline{\omega}(f)=-C_{\Theta}(1+a_\mu\mathbb{A}(f))^{-1}[f']\in W^{s-1}_p(\mathbb{R})$ requires that $f'\in W^{s-1}_p(\mathbb{R}),$
 but also the ''nonlinear argument'' $f$ in $(1+a_\mu\mathbb{A}(f))^{-1}$is required to belong to $W^{s}_p(\mathbb{R})$. 
 This  differs of course if $a_\mu=0$ and in this setting $\Phi$ has (in a suitable setting)  a quasilinear structure, cf.~\cite{MBV19}. 
 
The Rayleigh-Taylor condition can be simply formulated in our notation as
\begin{align}\label{RT}
C_{\Theta}+a_\mu\Phi(f)>0,
\end{align}
cf., e.g., \cite{MBV18}.
Since $\Phi(f)\in W^{s-1}_p(\mathbb{R})$, this condition implies, under the assumption\footnote{The case when $\Theta=0$ is trivial as the function $\overline{\omega}(f)$ defined in (\ref{SO}) (and therewith also $\Phi(f)$)
 is identically zero. Hence, for $\Theta=0,$ the initial surface is transported vertically  with constant velocity $V$ (and the Rayleigh-Taylor condition needs not to be imposed) and the velocities of the fluids
  are zero. } $\Theta\neq0$, that $\Theta>0$.
Restricting to the setting when $\Theta>0$, it follows from (\ref{REG}) that the set
\begin{align}\label{cO}
\mathcal{O}:=\{f\in W^s_p(\mathbb{R}) \,:\, C_{\Theta}+a_\mu\Phi(f)>0\}
\end{align} 
is an open subset  of $W^s_p(\mathbb{R}).$
The analysis below is devoted to showing that the Fr\'echet derivative $\partial\Phi(f_0)$ of $\Phi$ at $f_0\in\mathcal{O}$ generates an analytic semigroup in $\mathcal{L}( W^{s-1}_p(\mathbb{R}))$,
which in the notation from \cite{Am95} writes $-\partial\Phi(f_0)\in\mathcal{H}(W^s_p(\mathbb{R}), W^{s-1}_p(\mathbb{R})). $
This property identifies (\ref{NNEP}) as a parabolic evolution equation in $\mathcal{O}$ and facilitates us the use of abstract parabolic theory from \cite{L95} when solving it.

\begin{thm}\label{T:AG}
Let  $f_0\in \mathcal{O}$.
It then holds
\begin{align*}
-\partial\Phi(f_0)\in\mathcal{H}(W^s_p(\mathbb{R}), W^{s-1}_p(\mathbb{R})).
\end{align*}
\end{thm}

The proof of Theorem \ref{T:AG} (which is postponed to the end of the section) requires some preparation.
In the following we set
\[
\overline{\omega}_0:=\overline{\omega}(f_0)\in W^{s-1}_p(\mathbb{R}).
\]
By the chain rule 
\begin{align*} 
\partial\Phi(f_0)[f]=\partial\mathbb{B}(f_0)[f][\overline{\omega}_0]+\mathbb{B}(f_0)[\partial\overline{\omega}(f_0)[f]],\qquad f\in W^{s}_p(\mathbb{R}),
\end{align*}
where 
\begin{equation*} 
\begin{aligned}
 \pi\partial\mathbb{B}(f_0)[f][\overline{\omega}_0]&=-2B_{2,2}(f_0,f_0)[f_0,f,\overline{\omega}_0]+f'B_{1,1}(f_0)[f_0,\overline{\omega}_0] \\[1ex]
 &\hspace{0.45cm}+f'_0B_{1,1}(f_0)[f,\overline{\omega}_0]-2f_0'B_{3,2}(f_0,f_0)[f_0,f_0,f,\overline{\omega}_0]. 
\end{aligned}
\end{equation*}
Furthermore, differentiation  of (\ref{SO}) with respect to $f$ at $f_0$ yields
\begin{align*}
 (1+a_\mu\mathbb{A}(f_0))[\partial\overline{\omega}(f_0)[f]]=-a_\mu\partial\mathbb{A}(f_0)[f][\overline{\omega}_0]-C_{\Theta} f', \quad f\in W^s_p(\mathbb{R}),
\end{align*}
where
\begin{equation}\label{FD3}
\begin{aligned}
  \pi\partial\mathbb{A}(f_0)[f][\overline{\omega}_0]&=f'B_{0,1}(f_0)[\overline{\omega}_0]-2f_0'B_{2,2}(f_0,f_0)[f_0,f,\overline{\omega}_0]\\[1ex]
  &\hspace{0.424cm}-B_{1,1}(f_0)[f,\overline{\omega}_0]+2B_{3,2}(f_0,f_0)[f_0,f_0,f,\overline{\omega}_0], \quad f\in W^s_p(\mathbb{R}).
\end{aligned}
\end{equation}
In the derivation of (\ref{FD3}) we have several times made use of the formula
\begin{align*}
\partial (B_{n,m}(f,\ldots,f)[f,\ldots,f,\cdot])\big|_{f_0}[f]&=nB_{n,m}(f_0,\ldots,f_0)[f_0,f_0,\ldots,f,\cdot]\\[1ex]
&\hspace{0,45cm}-2mB_{n+2,m+1}(f_0,\ldots,f_0)[f_0,f_0,\ldots,f,\cdot].
\end{align*}
In order to establish Theorem \ref{T:AG} we consider a continuous path in 
$\mathcal{L}(  W^s_p(\mathbb{R}), W^{s-1}_p(\mathbb{R}))$ which is related to  $\partial\Phi(f_0),$ that is we define $\Psi:[0,1]\to \mathcal{L}(  W^s_p(\mathbb{R}), W^{s-1}_p(\mathbb{R}))$ 
by setting
\begin{align*}
\Psi(\tau)[f]:=\tau\partial\mathbb{B}(f_0)[f][\overline{\omega}_0]+\mathbb{B}(\tau f_0)[w(\tau)[f]],
\end{align*}
where $w\in {\rm C}([0,1],\mathcal{L}(  W^s_p(\mathbb{R}), W^{s-1}_p(\mathbb{R})))$ is    defined as the solution to 
\begin{align}\label{WT}
(1+a_\mu\mathbb{A}(\tau f_0))[w(\tau)[f]]=-\tau a_\mu\partial\mathbb{A}(f_0)[f][\overline{\omega}_0]-C_{\Theta} f'-(1-\tau)a_\mu f'\Phi(f_0)
\end{align}
for $\tau\in[0,1]$ and $f\in W^s_p(\mathbb{R})$.

\begin{rem}\label{R:1}
\begin{itemize}
\item[(i)] If $\tau=1$, then $w(1)=\partial\overline{\omega}(f_0)$ and therewith ${\Psi(1)=\partial\Phi(f_0).}$\\[-2ex]
\item[(ii)] Letting $H$ denote the Hilbert transform,  it holds $\Psi(0)=H\circ w(0)$. 
Moreover, noticing that $\mathbb{A}(0)=0$,  it holds  
 \begin{align*}
w(0)=-[C_{\Theta} +a_\mu \Phi(f_0)]\frac{d}{dx}.
\end{align*}
It is important to point out that the  function  of the right-hand side of the latter relation is exactly the function in the Rayleigh-Taylor condition (\ref{RT}).
This is one of the reasons why we artificially  introduced the term  $(1-\tau)a_\mu f'\Phi(f_0)$ in the definition~(\ref{WT}).
This construction  is essential for our purpose because it provides on one hand some useful cancellations in the proof of Theorem \ref{T:AP} and on the other hand it facilitates us to 
establish the invertibility of $\lambda-\Psi(0)\in\mathcal{L}(  W^s_p(\mathbb{R}), W^{s-1}_p(\mathbb{R}))$ for sufficiently large and positive $\lambda$, cf. Proposition \ref{P:IP0}. 
The latter point is important when establishing the invertibility of $\lambda-\partial\Phi(f_0)$ for such $\lambda$.  
\item[(iii)]  $H$ is the Fourier multiplier with symbol  $[\xi\mapsto-i\,{\rm sign}(\xi)]$ and $${H\circ(d/dx)=(- d^2/dx^2)^{1/2}.}$$
\item[(iv)] Given $s'\in(1+1/p,s),$   Lemma \ref{L:MP1}~(i), Theorem \ref{T:I1}, and Theorem~\ref{T:I2} imply there exists a constant $C=C(f_0)$ such that 
\begin{align}\label{we1}
\|w(\tau)[f]\|_p\leq C\|f\|_{W^{s'}_p}, \quad f\in W^{s}_p(\mathbb{R}),\, \tau\in[0,1].
\end{align}
Besides, Lemma \ref{L:MP3} (with $r=s'$)  and  Theorem \ref{T:INV} (with $s=s'$) yield
\begin{align}\label{we2}
\|w(\tau)[f]\|_{W^{s'-1}_p}\leq C\|f\|_{W^{s'}_p}, \quad \tau\in[0,1],\, f\in W^{s}_p(\mathbb{R}).
\end{align}
\end{itemize}
\end{rem}

Theorem \ref{T:AP} is the main step in the proof of Theorem \ref{T:AG}. 
In Theorem \ref{T:AP} it is shown that the operator $ \Psi(\tau)$  can be locally approximated by  certain Fourier multipliers $\mathbb{A}_{j,\tau}$. 
Theorem \ref{T:AP} also reveals  the importance of the  Rayleigh-Taylor condition  which ensures in this context   the positivity of the 
coefficient $\alpha_\tau(x_j^\varepsilon)$ in the definition of  $\mathbb{A}_{j,\tau}$ below.

\begin{thm}\label{T:AP} 
Let   $\mu>0$ be given and fix $s'\in (1+1/p,s)$. 
Then, there exist $\varepsilon\in(0,1)$, an $\varepsilon$-locali\-za\-tion family  $\{\pi_j^\varepsilon\,:\, -N+1\leq j\leq N\} $,  a constant $K=K(\varepsilon,f_0)$, 
and   bounded operators 
$$
\mathbb{A}_{j,\tau}\in\mathcal{L}(W^s_p(\mathbb{R}), W^{s-1}_p(\mathbb{R})), \qquad\mbox{$j\in\{-N+1,\ldots,N\}$ and $\tau\in[0,1]$,} 
$$
 such that 
 \begin{equation}\label{D1}
  \|\pi_j^\varepsilon \Psi(\tau) [f]-\mathbb{A}_{j,\tau}[\pi^\varepsilon_j f]\|_{W^{s-1}_p}\leq \mu \|\pi_j^\varepsilon f\|_{W^s_p}+K\|  f\|_{W^{s'}_p}
 \end{equation}
 for all $ j\in\{-N+1,\ldots,N\}$, $\tau\in[0,1],$  and  $f\in W^s_p(\mathbb{R})$. 
 The operators $\mathbb{A}_{j,\tau}$ are defined  by 
  \begin{align*} 
 \mathbb{A}_{j,\tau }:=- \alpha_\tau(x_j^\varepsilon) \Big(-\frac{d^2}{dx^2}\Big)^{1/2}+\beta_\tau (x_j^\varepsilon)\frac{d}{dx}, \quad |j|\leq N-1, \qquad\mathbb{A}_{N,\tau }:= -   C_{\Theta} \Big(-\frac{d^2}{dx^2}\Big)^{1/2},
 \end{align*}
 where  $x_j^\varepsilon\in \supp  \pi_j^\varepsilon,$ $|j|\leq N-1,$ and  with functions $\alpha_\tau,\, \beta_\tau$ given by
 \begin{align*}
 \alpha_\tau:={\Big(1-\frac{\tau f_0'^2}{1+f_0'^2}\Big)} [C_{\Theta} +a_\mu \Phi(f_0)], \qquad  \beta_\tau:= \frac{\tau}{\pi}B_{1,1}(f_0)[f_0,\overline{\omega}_0]-\tau a_\mu\frac{\overline{\omega}_0}{1+f_0'^2}.   
 \end{align*}
\end{thm}

Before proving Theorem~\ref{T:AP} we first present some lemmas (which are proved in the Appendix~\ref{Sec:C}) which are used in an essential  way when establishing Theorem~\ref{T:AP}.\smallskip

The following commutator estimate is used several times in the paper.

\begin{lemma}\label{L:B2} 
Let $n,\, m \in \mathbb{N}$, $p\in(1,\infty),$  $s\in(1+1/p, 2)$, $f\in W^s_p(\mathbb{R})$, and  ${\varphi\in {\rm C}^1(\mathbb{R})}$ with uniformly continuous derivative $\varphi'$ be given. 
Then, there exist  a constant $K$ that depends only on $ n, $ $m, $ $\|\varphi'\|_\infty, $ and $\|f\|_{W^s_p}$  such that 
 \begin{equation}\label{LB2}
  \|\varphi B_{n,m}(f,\ldots,f)[f,\ldots,f, h]- B_{n,m}(f,\ldots,f)[f,\ldots,f, \varphi h]\|_{W^{1}_p}\leq K\| h\|_{p}
 \end{equation}
for all   $h\in L_p(\mathbb{R})$.
\end{lemma}

The next lemma is used in the proof of Theorem \ref{T:AP}.

\begin{lemma}\label{L:C1} 
Let $n,\, m \in \mathbb{N}$, $p\in(1,\infty),$ $1+1/p<s'<s<2$, and  $\nu\in(0,\infty)$ be given. 
Let further  $f\in W^s_p(\mathbb{R})$ and  $\overline{\omega}\in \{1\}\cup W^{s-1}_p(\mathbb{R})$.
For sufficiently small $\varepsilon\in(0,1)$   there exists a constant $K=K(\varepsilon, n, m, \|f\|_{W^s_p},\|\overline{\omega}\|_{W^{s-1}_p})$  such that
 \begin{equation}\label{LC1}
\begin{aligned}
  &\Big\|\pi_j^\varepsilon\overline{\omega} B_{n,m}(f,\ldots,f)[f,\ldots,f, h]-\frac{\overline{\omega}(x_j^\varepsilon)(f'(x_j^\varepsilon))^n}{[1+(f'(x_j^\varepsilon))^2]^m}B_{0,0}[\pi_j^\varepsilon h]\Big\|_{W^{s-1}_p}\\[1ex]
  &\hspace{2.5cm}\leq \nu \|\pi_j^\varepsilon h\|_{W^{s-1}_p}+K\| h\|_{W^{s'-1}_p}
\end{aligned}  
 \end{equation}
for all $|j|\leq N-1$ and  $h\in W^{s-1}_p(\mathbb{R})$  (with $x_j^\varepsilon\in \supp  \pi_j^\varepsilon$).
\end{lemma}

The next two lemmas are the analogues of Lemma \ref{L:C1}  that deal with the case when $j=N$.

\begin{lemma}\label{L:C1a} 
Let $n,\, m \in \mathbb{N}$, $p\in(1,\infty),$ $1+1/p<s'<s<2$, and  $\nu\in(0,\infty)$ be given. 
Let further  $f\in W^s_p(\mathbb{R})$ and  $\overline{\omega}\in  W^{s-1}_p(\mathbb{R})$.
For sufficiently small $\varepsilon\in(0,1)$   there exists a constant $K=K(\varepsilon, n, m, \|f\|_{W^s_p},\|\overline{\omega}\|_{W^{s-1}_p})$  such that 
  \begin{equation}\label{LC2}
  \|\pi_j^\varepsilon\overline{\omega} B_{n,m}(f,\ldots,f)[f,\ldots,f, h]\|_{W^{s-1}_p}\leq \nu \|\pi_j^\varepsilon h\|_{W^{s-1}_p}+K\| h\|_{W^{s'-1}_p}
 \end{equation} 
 for $j=N$  and $h\in W^{s-1}_p(\mathbb{R})$.
\end{lemma}

We now establish the counterpart of (\ref{LC2}) in the case when $\overline{\omega}=1$.

\begin{lemma}\label{L:C2} 
Let $n,\, m \in \mathbb{N}$, $p\in(1,\infty),$ $1+1/p<s'<s<2$, and  $\nu\in(0,\infty)$ be given. 
Let further  $f\in W^s_p(\mathbb{R})$.
For sufficiently small $\varepsilon\in(0,1)$   there exist a constant $K=K(\varepsilon, n, m, \|f\|_{W^s_p})$  such that 
 \begin{equation}\label{LC1Na}
  \|\pi_j^\varepsilon B_{0,m}(f,\ldots,f)[ h]-B_{0,0}[\pi_j^\varepsilon h]\|_{W^{s-1}_p}\leq \nu \|\pi_j^\varepsilon h\|_{W^{s-1}_p}+K\| h\|_{W^{s'-1}_p}
 \end{equation}
 and 
  \begin{equation}\label{LC1Nb}
  \|\pi_j^\varepsilon B_{n,m}(f,\ldots,f)[f,\ldots,f, h]\|_{W^{s-1}_p}\leq \nu \|\pi_j^\varepsilon h\|_{W^{s-1}_p}+K\| h\|_{W^{s'-1}_p},\qquad n\geq 1,
 \end{equation}
 for $j=N$ and all $h\in W^{s-1}_p(\mathbb{R})$.
\end{lemma}

\pagebreak

We are now in a position to prove Theorem \ref{T:AP}.

\begin{proof}[Proof of Theorem \ref{T:AP}]
Let $\{\pi_j^\varepsilon\,:\, -N+1\leq j\leq N\} $ be an  $\varepsilon$-localization family   with the associated family~${\{\chi_j^\varepsilon\,:\, -N+1\leq j\leq N\}}$, 
with $\varepsilon\in(0,1)$ to be fixed later on. 
In this proof we denote by $C$ constants that depend only on $f_0$. 
Constants  denoted by $K$ may depend only  on $\varepsilon$ and $f_0$.\medskip

\noindent{\em Step 1: The terms $\partial\mathbb{B}(f_0)[f][\overline{\omega}_0].$} 
In virtue of Lemma \ref{L:MP3b} (with $r=s$ and $r'=s'$) it holds 
\begin{equation}\label{nbn}
\begin{aligned}
&\hspace{-0.5cm}\big\|-2B_{2,2}(f_0,f_0)[f_0,f,\overline{\omega}_0]+f'_0B_{1,1}(f_0)[f,\overline{\omega}_0]-2f_0'B_{3,2}(f_0,f_0)[f,f_0,f_0,\overline{\omega}_0]\\[1ex]
&-\overline{\omega}_0\big(-2B_{1,2}(f_0,f_0)[f_0,f']+f'_0B_{0,1}(f_0)[f']-2f_0'B_{2,2}(f_0,f_0)[f_0,f_0,f']\big)\big\|_{W^{s-1}_p}\\[1ex]
&\hspace{4cm}\leq C\|f\|_{W^{s'}_p}.
\end{aligned}
\end{equation}
Moreover, invoking Lemma \ref{L:C1}, if $\varepsilon$ is sufficiently small, then
\begin{align*}
&\hspace{-0.5cm}\big\|\pi_j^\varepsilon\overline{\omega}_0\big(-2B_{1,2}(f_0,f_0)[f_0,f']+f'_0B_{0,1}(f_0)[f']-2f_0'B_{2,2}(f_0,f_0)[f_0,f_0,f']\big)\\[1ex]
&\hspace{2cm}+\frac{\overline{\omega}_0(x_j^\varepsilon) f'_0(x_j^\varepsilon)}{1+(f'_0(x_j^\varepsilon))^2}B_{0,0}[\pi_j^\varepsilon f']\big\|_{W^{s-1}_p}\leq \frac{\mu}{4}\|\pi_j^\varepsilon f\|_{W^{s }_p}+K\|  f\|_{W^{s'}_p} 
\end{align*}
for all $|j|\leq N-1$.  Besides, for  $j=N$,  Lemma \ref{L:C1a} yields 
\begin{align*}
&\hspace{-0.5cm}\big\|\pi_j^\varepsilon\overline{\omega}_0\big(-2B_{1,2}(f_0,f_0)[f_0,f']+f'_0B_{0,1}(f_0)[f']-2f_0'B_{2,2}(f_0,f_0)[f_0,f_0,f']\big)\big\|_{W^{s-1}_p}\\[1ex]
&\hspace{2cm}\leq \frac{\mu}{4}\|\pi_j^\varepsilon f\|_{W^{s }_p}+K\|  f\|_{W^{s'}_p}.
\end{align*}
Finally, in view of  (\ref{MES})
and of $B_{1,1}(f_0)[f_0,\overline{\omega}_0]\in{\rm C}^{s-1-1/p}(\mathbb{R})$, it holds 
\begin{equation}\label{DECI}
\begin{aligned}
&\hspace{-0.5cm}\|\pi_j^\varepsilon f'B_{1,1}(f_0)[f_0,\overline{\omega}_0]-(\pi_j^\varepsilon f)'B_{1,1}(f_0)[f_0,\overline{\omega}_0](x_j)\|_{W^{s-1}_p}\\[1ex]
&\leq  \|\chi_j^\varepsilon (B_{1,1}(f_0)[f_0,\overline{\omega}_0]-B_{1,1}(f_0)[f_0,\overline{\omega}_0](x_j))(\pi_j^\varepsilon f)'\|_{W^{s-1}_p}+K\|f\|_{W^{s-1}_p}\\[1ex]
&\leq \frac{\mu}{4}\|\pi_j^\varepsilon f\|_{W^{s}_p} +K\|f\|_{W^{s'}_p}
\end{aligned}
\end{equation}
 for $|j|\leq N-1$. 
 Moreover, since $B_{1,1}(f_0)[f_0,\overline{\omega}_0]\to0$ for $|x|\to\infty$, (\ref{MES})   yields
\begin{align*}
\|\pi_j^\varepsilon f'B_{1,1}(f_0)[f_0,\overline{\omega}_0]\|_{W^{s-1}_p}  \leq \frac{\mu}{4}\|\pi_j^\varepsilon f\|_{W^{s}_p} +K\|f\|_{W^{s'}_p}
\end{align*}
for $j=N$.
Hence, if $\varepsilon$ is sufficiently small, then
\begin{equation}\label{ST3a}
\begin{aligned}
  &\Big\|\pi_j^\varepsilon\tau\partial\mathbb{B}(f_0)[f][\overline{\omega}_0] + \frac{\tau\overline{\omega}_0(x_j^\varepsilon) f'_0(x_j^\varepsilon)}{\pi(1+(f'_0(x_j^\varepsilon))^2)}B_{0,0}[(\pi_j^\varepsilon f)']-\frac{\tau}{\pi}B_{1,1}(f_0)[f_0,\overline{\omega}_0](x_j^\varepsilon)(\pi_j^\varepsilon f)'\Big\|_{W^{s-1}_p} \\[1ex]
  &\hspace{2cm}\leq \frac{\mu}{2}\|\pi_j^\varepsilon f\|_{W^{s}_p} +K\|f\|_{W^{s'}_p}
\end{aligned}
\end{equation}
for all $ |j|\leq N-1$ and  $f\in W^{s }_p(\mathbb{R}) ,$ and 
\begin{align}\label{ST3b}
  \|\pi_j^\varepsilon\tau\partial\mathbb{B}(f_0)[f][\overline{\omega}_0] \|_{W^{s-1}_p} \leq \frac{\mu}{2}\|\pi_j^\varepsilon f\|_{W^{s}_p} +K\|f\|_{W^{s'}_p}
\end{align}
for $ j=N$ and all $f\in W^{s }_p(\mathbb{R})$.\medskip

\noindent{\em Step 2: The terms $\mathbb{B}(\tau f_0)[w(\tau)[f]].$} 
We first estimate   $\|\pi_j^\varepsilon w(\tau)[f]\|_{W^{s-1}_p}$, $-N+1\leq j\leq N$.
Recalling Lemma~\ref{L:B2} and (\ref{we1}), it holds 
 \[
 \|\pi_j^\varepsilon\mathbb{A}(\tau f_0)[w(\tau)[f]]- \mathbb{A}(\tau f_0)[\pi_j^\varepsilon(w(\tau)[f]]\|_{W^{s-1}_p}\leq   K\|f\|_{W^{s'}_p}.
 \]
Besides, multiplying   (\ref{WT}) by $\pi_j^\varepsilon$,  it follows from Lemma \ref{L:MP3}, Lemma \ref{L:MP3b}, and Lemma \ref{L:B2}  that
\begin{align*}
\|\pi_j^\varepsilon(1+a_\mu\mathbb{A}(\tau f_0))[w(\tau)[f]]\|_{W^{s-1}_p}&\leq C\|\pi_j^\varepsilon\partial\mathbb{A}(f_0)[f][\overline{\omega}_0]\|_{W^{s-1}_p}+C\|\pi_j^\varepsilon f'\|_{W^{s-1}_p}\\[1ex]
&\leq C\|\pi_j^\varepsilon f\|_{W^{s}_p}+ K\|f\|_{W^{s'}_p}.
\end{align*}
In order to estimate the last three terms of $\pi_j^\varepsilon\partial\mathbb{A}(f_0)[f][\overline{\omega}_0]$ we  used Lemma~\ref{L:MP3b} in a similar manner as in the derivation of (\ref{nbn}), and afterwards the commutator estimate in Lemma~\ref{L:B2} to write in the end $\pi_j^\varepsilon$ as a multiplying factor of $f'$.

Combining the last two estimates we arrive at
\begin{align*}
\|(1+a_\mu\mathbb{A}(\tau f_0))[\pi_j^\varepsilon w(\tau)[f]]\|_{W^{s-1}_p}&\leq C\|\pi_j^\varepsilon f\|_{W^{s}_p}+ K\|f\|_{W^{s'}_p}.
\end{align*}
Finally, Theorem \ref{T:INV} ensures  there exists a constant $C_0=C_0(f_0)>0$  with
\begin{align}\label{wtf4}
\|\pi_j^\varepsilon w(\tau)[f]\|_{W^{s-1}_p}\leq  C_0\|\pi_j^\varepsilon f\|_{W^{s}_p}+ K\|f\|_{W^{s'}_p}.
\end{align}

It virtue of Lemma~\ref{L:C1} (with $\nu=\mu/(8C_0)$, where $C_0$ is the constant in (\ref{wtf4})), (\ref{we2}), and~(\ref{wtf4}) for $\varepsilon$ sufficiently small and $|j|\leq N-1$ it holds  that 
\begin{equation}\label{ST2a}
\begin{aligned}
&\hspace{-0,5cm}\|\pi_j^\varepsilon \mathbb{B}(\tau f_0)[w(\tau)[f]]- \pi^{-1}B_{0,0}[\pi_j^\varepsilon w(\tau)[f]]\|_{W^{s-1}_p}\\[1ex] 
&\leq \frac{\mu}{4C_0}\|\pi_j^\varepsilon w(\tau)[f]\|_{W^{s-1}_p}+K\|w(\tau)[f]\|_{W^{s'-1}_p}\\[1ex]
&\leq \frac{\mu}{4}\|\pi_j^\varepsilon f \|_{W^s_p}+K\|f\|_{W^{s'}_p}. 
\end{aligned} 
\end{equation}
Lemma \ref{L:C1a}, Lemma \ref{L:C2}, (\ref{we2}), and~(\ref{wtf4}) show that  (\ref{ST2a}) stays true also for~$j=N$.

We now set
\[
\varphi_\tau:= C_{\Theta} +(1-\tau)a_\mu\Phi(f_0)   +\frac{\tau a_\mu}{\pi} B_{0,1}(f_0)[\overline{\omega}_0],\qquad\tau\in[0,1],
\]
and  prove that
\begin{equation} \label{ST2b}
\begin{aligned}
&\Big\|  B_{0,0}[\pi_j^\varepsilon w(\tau)[f]]+ \varphi_\tau(x_j^\varepsilon)B_{0,0}[(\pi_j^\varepsilon f)'] + \tau a_\mu \pi \frac{\overline{\omega}_0(x_j^\varepsilon)}{1+f_0'^2(x_j^\varepsilon)} (\pi_j^\varepsilon f)' \Big\|_{W^{s-1}_p} \\[1ex]
&\hspace{2cm}\leq \frac{\mu}{4}\|\pi_j^\varepsilon f \|_{W^s_p}+K\|f\|_{W^{s'}_p} 
\end{aligned}
\end{equation}
 for all   $|j|\leq N-1$, provided   $\varepsilon$ is small.
 Indeed, since $B_{0,0}^2=\pi^2 H^2=-\pi^2\id_{L_p(\mathbb{R})},$ it holds  
\begin{equation} \label{ST2c}
\begin{aligned}
&\hspace{-0,5cm}\Big\|  B_{0,0}[\pi_j^\varepsilon w(\tau)[f]]  +\varphi_\tau(x_j^\varepsilon)B_{0,0}[(\pi_j^\varepsilon f)'] + \tau a_\mu \pi \frac{\overline{\omega}_0(x_j^\varepsilon)}{1+f_0'^2(x_j^\varepsilon)} (\pi_j^\varepsilon f)'\Big\|_{W^{s-1}_p}\\[1ex]
& \leq C_1\Big\|\pi_j^\varepsilon w(\tau)[f]+\varphi_\tau(x_j^\varepsilon)\pi_j^\varepsilon f'-\frac{\tau a_\mu}{\pi}\frac{\overline{\omega}_0(x_j^\varepsilon)}{1+f_0'^2(x_j^\varepsilon)}B_{0,0}[\pi_j^\varepsilon f']\Big\|_{W^{s-1}_p}+K\|f\|_{W^{s-1}_p}. 
\end{aligned}
\end{equation}
Besides, multiplying (\ref{WT}) by $\pi_j^\varepsilon$ and using the definition of $\varphi_\tau$, we arrive at
\begin{equation}\label{FLFOR}
\pi_j^\varepsilon w(\tau)[f]+\varphi_\tau(x_j^\varepsilon)\pi_j^\varepsilon f'-\frac{\tau a_\mu}{\pi}\frac{\overline{\omega}_0(x_j^\varepsilon)}{1+f_0'^2(x_j^\varepsilon)}B_{0,0}[\pi_j^\varepsilon f']=T_1+T_{2}+T_{3},\\[1ex]
\end{equation}
where
\begin{align*}
T_1&:=(1-\tau)a_\mu(\Phi(f_0)(x_j^\varepsilon)-\Phi(f_0))\pi_j^\varepsilon f',\\[1ex]
T_2&:=-\frac{\tau a_\mu}{\pi}\Big(\pi\pi_j^\varepsilon\partial \mathbb{A}(f_0)[f][\overline{\omega}_0]-B_{0,1}(f_0)[\overline{\omega}_0](x_j^\varepsilon)\pi_j^\varepsilon f'+\frac{\overline{\omega}_0(x_j^\varepsilon)}{1+f_0'^2(x_j^\varepsilon)}B_{0,0}[\pi_j^\varepsilon f']\Big),\\[1ex]
T_3&:=-a_\mu\pi_j^\varepsilon \mathbb{A}(\tau f_0)[w(\tau)[f]].
\end{align*}
 The   term  $T_1$ can be estimated by using  the fact that ${\Phi(f_0)\in{\rm C}^{s-1-1/p}}(\mathbb{R})$, similarly  as in~(\ref{DECI}). 
Concerning $T_3$, we infer from (\ref{FormulaA}) that 
\begin{align*}
T_3&=-\frac{\tau a_\mu}{\pi}\Big[\Big(\pi_j^\varepsilon f_0'B_{0,1}(\tau f_0)[w(\tau)[f]]-\frac{f_0'(x_j^\varepsilon)}{1+\tau^2f_0'^2(x_j^\varepsilon)}B_{0,0}[\pi_j^\varepsilon w(\tau)[f]]\Big)\\[1ex]
&\hspace{1,75cm}-\Big(\pi_j^\varepsilon  B_{1,1}(\tau f_0)[f_0,w(\tau)[f]]-\frac{f_0'(x_j^\varepsilon)}{1+\tau^2f_0'^2(x_j^\varepsilon)}B_{0,0}[\pi_j^\varepsilon w(\tau)[f]]\Big)\Big]
\end{align*}
and both terms can be estimated by using (\ref{we2}), Lemma \ref{L:C1}, and (\ref{wtf4}).
Finally, using~(\ref{FD3}), it holds that 
\begin{equation*} 
\begin{aligned}
T_2&=-\frac{\tau a_\mu}{\pi}\Big[\pi_j^\varepsilon f'\big(B_{0,1}(f_0)[\overline{\omega}_0]-B_{0,1}(f_0)[\overline{\omega}_0](x_j^\varepsilon)\big)+T_{\rm LOT}[f]\\[1ex]
&\hspace{1,65cm}-2\Big(\pi_j^\varepsilon f_0'\overline{\omega}_0 B_{1,2}(f_0,f_0)[f_0,f']-\frac{f_0'^2(x_j^\varepsilon)\overline{\omega}_0(x_j^\varepsilon)}{[1+f_0'^2(x_j^\varepsilon)]^2}B_{0,0}[\pi_j^\varepsilon f']\Big)\\[1ex]
&\hspace{1,65cm}-\Big(\pi_j^\varepsilon\overline{\omega}_0B_{0,1}(f_0)[f']-\frac{\overline{\omega}_0(x_j^\varepsilon)}{1+f_0'^2(x_j^\varepsilon)}B_{0,0}[\pi_j^\varepsilon f']\Big)\\[1ex]
&\hspace{1,65cm}+2\Big(\pi_j^\varepsilon\overline{\omega}_0B_{2,2}(f_0,f_0)[f_0,f_0,f']-\frac{f_0'^2(x_j^\varepsilon)\overline{\omega}_0(x_j^\varepsilon)}{[1+f_0'^2(x_j^\varepsilon)]^2}B_{0,0}[\pi_j^\varepsilon f']\Big)\Big],
\end{aligned}
\end{equation*}
where
\begin{align*}
T_{\rm LOT}[f]&:=-2(f_0'B_{2,2}(f_0,f_0)[f_0,f,\overline{\omega}_0]-f_0'\overline{\omega}_0B_{1,2}(f_0,f_0)[f_0,f'])\\[1ex]
&\hspace{0,54cm}-(B_{1,1}(f_0)[f,\overline{\omega}_0]-\overline{\omega}_0B_{0,1}(f_0)[f'])\\[1ex]
&\hspace{0,54cm}+2(B_{3,2}(f_0,f_0)[f_0,f_0,f,\overline{\omega}_0]-\overline{\omega}_0B_{2,2}(f_0,f_0)[f_0,f_0,f']).
\end{align*}
Lemma \ref{L:MP3b} yields
\begin{align*}
\|\pi_j^\varepsilon T_{\rm LOT}[f]\|_{W^{s-1}_p}\leq K\|f\|_{W^{s'}_p}.
\end{align*}
 The first term in the decomposition of $T_2$ is estimated by using ${B_{0,1}(f_0)[\overline{\omega}_0]\in{\rm C}^{s-1-1/p}}(\mathbb{R})$, similarly  as in (\ref{DECI}).  
 For the last three   terms we rely on Lemma \ref{L:C1}, (\ref{we2}),    and (\ref{wtf4}).
 Altogether, we conclude that if $\varepsilon$ is sufficiently small, then
 \begin{align*}
 \Big\|\pi_j^\varepsilon w(\tau)[f]+\varphi_\tau(x_j^\varepsilon)\pi_j^\varepsilon f'-\frac{\tau a_\mu}{\pi}\frac{\overline{\omega}_0(x_j^\varepsilon)}{1+f_0'^2(x_j^\varepsilon)}
 B_{0,0}[\pi_j^\varepsilon f']\Big\|_{W^{s-1}_p}
  \leq  \frac{\mu}{4C_1}\|\pi_j^\varepsilon f \|_{W^s_p}+K\|f\|_{W^{s'}_p}, 
 \end{align*}
 and together with (\ref{ST2c}) we have proven (\ref{ST2b}).
From (\ref{ST2a}) and (\ref{ST2b}) we finally conclude that
\begin{equation}\label{ST3a'}
\begin{aligned}
&\Big\|\pi_j^\varepsilon \mathbb{B}(\tau f_0)[w(\tau)[f]]+\pi^{-1}\varphi_\tau(x_j^\varepsilon)B_{0,0}[(\pi_j^\varepsilon f)'] + \tau a_\mu  \frac{\overline{\omega}_0(x_j^\varepsilon)}{1+f_0'^2(x_j^\varepsilon)} (\pi_j^\varepsilon f)'\Big\|_{W^{s-1}_p} \\[1ex]
&\hspace{2cm}\leq \frac{\mu}{2}\|\pi_j^\varepsilon f \|_{W^s_p}+K\|f\|_{W^{s'}_p} 
\end{aligned}
\end{equation}
 for all   $|j|\leq N-1$, provided that $\varepsilon$ is small.
 Using also  Lemma \ref{L:C1a} and Lemma \ref{L:C2},   it is not difficult to infer from the latter relations that
 \begin{equation}\label{ST3a''}
\Big\|\pi_j^\varepsilon \mathbb{B}(\tau f_0)[w(\tau)[f]]+\pi^{-1}C_{\Theta} B_{0,0}[(\pi_j^\varepsilon f)'] \|_{W^{s-1}_p}\leq \frac{\mu}{2}\|\pi_j^\varepsilon f \|_{W^s_p}+K\|f\|_{W^{s'}_p} 
\end{equation}
 for    $j=N$, provided that $\varepsilon$ is small.
 
Combining the relation $\overline{\omega}_0=-C_{\Theta}f_0'-a_\mu\mathbb{A}(f_0)[\overline{\omega}_0]$  with the estimates(\ref{ST3a})  and (\ref{ST3a'})(and recalling also Remark~\ref{R:1}~(iii) and the identity $B_{0,0}=\pi H$), we conclude that  (\ref{D1}) holds true in the case when $|j|\leq N-1$.
For $j=N$ the desired claim (\ref{D1}) follows from~(\ref{ST3b})  and (\ref{ST3a''}).
 \end{proof}

We now consider the Fourier multipliers from Theorem  \ref{T:AP} more closely.

\begin{lemma}\label{L:GAP} 
 There exists a constant  $\eta=\eta(f_0)\in(0,1)$ with the property that
\begin{align}\label{BEta}
\eta\leq\alpha_\tau\leq \frac{1}{\eta}\qquad\mbox{and}\qquad \|\beta_\tau\|_\infty\leq \frac{1}{\eta}
\end{align}
for all $\tau\in[0,1]$.
Moreover, given  $\alpha\in[\eta,1/\eta]$ and $|\beta|\leq 1/\eta$, there exists a constant~${\kappa_0\geq 1}$   such that the Fourier multiplier 
\begin{align*} 
 \mathbb{A}_{\alpha,\beta}:=- \alpha\Big(-\frac{d^2}{dx^2}\Big)^{1/2}+\beta\frac{d}{dx},  
 \end{align*}
 satisfies
 \begin{align}
\bullet &\quad \lambda-\mathbb{A}_{\alpha,\beta}\in {\rm Isom}(W^s_p(\mathbb{R}),W^{s-1}_p(\mathbb{R})),\qquad  \forall\, \re\lambda\geq 1,\label{L:FM1}\\[1ex]
\bullet &\quad  \kappa_0\|(\lambda-\mathbb{A}_{\alpha,\beta})[f]\|_{W^{s-1}_p}\geq |\lambda|\cdot\|f\|_{W^{s-1}_p}+\|f\|_{W^s_p}, \qquad \forall\, f\in W^{s}_p(\mathbb{R}),\, \re\lambda\geq 1\label{L:FM2}.
\end{align}
\end{lemma}
\begin{proof}
The bounds  (\ref{BEta}) are a consequence of $f_0\in\mathcal{O}$.
Finally, in order to prove the properties (\ref{L:FM1})-(\ref{L:FM2}),  we first consider the realizations 
$$\mathbb{A}_{\alpha,\beta}\in \mathcal{L}(W^1_p(\mathbb{R}),L_p(\mathbb{R}))\quad\mbox{and}\quad\mathbb{A}_{\alpha,\beta}\in\mathcal{L}(W^2_p(\mathbb{R}),W^1_p(\mathbb{R}))$$ 
for which the properties (\ref{L:FM1})-(\ref{L:FM2}) (in the appropriate spaces) can be  established in view of the  identification $W^k_p(\mathbb{R})=H^k_p(\mathbb{R})$, $k\in\mathbb{N}$ (using Fourier analysis and, in particular, Mikhlin's multiplier theorem, cf. e.g. \cite[Theorem 4.23]{AB12}).
 Then,  using the interpolation property (\ref{IP}) we conclude that (\ref{L:FM1})-(\ref{L:FM2}) hold true.
\end{proof}
\pagebreak

We next exploit for a second time the Rayleigh-Taylor condition to show that $\lambda-\Psi(0)$  
is an isomorphism provided that $\lambda\in\mathbb{R}$ is sufficiently large.

\begin{prop}\label{P:IP0} Let $\delta>0$ and $\phi\in W^{s-1}_p(\mathbb{R})$ satisfy  $a:=\delta+\phi>0.$ 
Then, there exists a  constant $\omega_0>0$ such that
\begin{align}\label{eq:IP0}
\lambda+H\circ\Big(a\frac{d}{dx}\Big)\in{\rm Isom}(W^{s}_p(\mathbb{R}), W^{s-1}_p(\mathbb{R}))\qquad\mbox{for $\lambda\in[\omega_0,\infty)$.}
\end{align}
\end{prop}
\begin{proof}
Let $[\tau\mapsto B(\tau)]:[0,1]\to \mathcal{L}(W^s_p(\mathbb{R}), W^{s-1}_p(\mathbb{R}))$ be given by
\[
B(\tau):=  H\circ\Big(a_\tau\frac{d}{dx}\Big),
\]
where $a_{\tau}:=(1-\tau)\delta+\tau a=\delta+\tau\phi$, $\tau\in[0,1]$.
It then holds $B\in {\rm C}([0,1], \mathcal{L}(W^s_p(\mathbb{R}), W^{s-1}_p(\mathbb{R})))$. 
We prove below   there exist  constants $\omega_0>0$ and  $C>0$  such that
 \begin{align}\label{EDY}
  \|(\lambda+B(\tau))[f]\|_{W^{s-1}_p}\geq C\|f\|_{W^s_p},\qquad \mbox{$\forall$ $\tau\in[0,1]$, $\lambda\in[\omega_0,\infty),$ $f\in W^s_p(\mathbb{R})$.}
 \end{align}
Since $\lambda+B(0)$ is the Fourier multiplier with symbol $m_\lambda(\xi):= \lambda+\delta |\xi|$, $\xi \in\mathbb{R}$, it holds that  $\lambda+B(0)\in{\rm Isom}(W^{s}_p(\mathbb{R}), W^{s-1}_p(\mathbb{R}))$  for all $\lambda>0$. 
The method of continuity and (\ref{EDY})   imply then that (\ref{eq:IP0}) holds true.\medskip

\noindent{\em Step 1.}
Let  $s'\in (1+1/p,s)$. 
Given  $\mu>0$, we find below $\varepsilon\in(0,1)$, an $\varepsilon$-locali\-za\-tion family~$\{\pi_j^\varepsilon\,:\, -N+1\leq j\leq N\} $, 
a constant $K=K(\varepsilon)$ and bounded operators 
$$\mathbb{B}_{j,\tau}\in\mathcal{L}(W^s_p(\mathbb{R}), W^{s-1}_p(\mathbb{R})),  \qquad\mbox{$j\in\{-N+1,\ldots,N\}$ and $\tau\in[0,1]$},$$
 such that 
 \begin{equation}\label{D1'}
  \|\pi_j^\varepsilon B(\tau) [f]-\mathbb{B}_{j,\tau}[\pi^\varepsilon_j f]\|_{W^{s-1}_p}\leq \mu \|\pi_j^\varepsilon f\|_{W^s_p}+K\|  f\|_{W^{s'}_p}
 \end{equation}
 for all $ j\in\{-N+1,\ldots,N\}$, $\tau\in[0,1],$  and  $f\in W^s_p(\mathbb{R})$. 
The operators $\mathbb{B}_{j,\tau}$ are defined  by 
  \begin{align*} 
 \mathbb{B}_{N,\tau }:=   \delta \Big(-\frac{d^2}{dx^2}\Big)^{1/2}\qquad\mbox{and}\qquad\mathbb{B}_{j,\tau }:=  a_\tau(x_j^\varepsilon) \Big(-\frac{d^2}{dx^2}\Big)^{1/2}, \quad |j|\leq N-1, 
 \end{align*}
 where  $x_j^\varepsilon\in \supp  \pi_j^\varepsilon$.  
  
 For $-N+1\leq j\leq N$ it holds that
 \begin{align*}
 \|\pi_j^\varepsilon B(\tau) [f]-\mathbb{B}_{j,\tau}[\pi^\varepsilon_j f]\|_{W^{s-1}_p}&\leq \|\pi_j^\varepsilon B(\tau) [f]-B(\tau)[\pi^\varepsilon_j f]\|_{W^{s-1}_p}\\[1ex]
 &\hspace{0,45cm}+\|B(\tau)[\pi^\varepsilon_j f]-\mathbb{B}_{j,\tau}[\pi^\varepsilon_j f]\|_{W^{s-1}_p}.
 \end{align*}
Lemma \ref{L:B2} yields
 \[
 \|\pi_j^\varepsilon B(\tau) [f]-B(\tau)[\pi^\varepsilon_j f]\|_{W^{s-1}_p}\leq K\|f\|_{W^{1}_p},\qquad -N+1\leq j\leq N.
 \]
 Moreover, for $|j|\leq N-1$ we use the identity $\chi_j^\varepsilon\pi_j^\varepsilon=\pi_j^\varepsilon$  together with (\ref{MES}), Lemma \ref{L:MP1}~(ii) (with $r=s$ and  $\tau=s'-1$), and Lemma~\ref{L:MP3} to derive that
 \begin{align*}
\|B(\tau)[\pi^\varepsilon_j f]-\mathbb{B}_{j,\tau}[\pi^\varepsilon_j f]\|_{W^{s-1}_p}&=\|H[(a_\tau-a_\tau(x_j^\varepsilon))\chi_j^\varepsilon(\pi^\varepsilon_j f)']\|_{W^{s-1}_p}\\[1ex]
 &\leq C\|(\phi-\phi(x_j^\varepsilon))\chi_j^\varepsilon(\pi^\varepsilon_j f)']\|_{W^{s-1}_p}\\[1ex]
 &\leq  C\|(\phi-\phi(x_j^\varepsilon))\chi_j^\varepsilon\|_\infty\|(\pi^\varepsilon_j f)' \|_{W^{s-1}_p}+K\|f\|_{W^{s'}_p}\\[1ex]
 &\leq \mu \|\pi_j^\varepsilon f\|_{W^s_p}+K\|  f\|_{W^{s'}_p} ,
 \end{align*}
 provided that $\varepsilon$ is sufficiently small.
This proves (\ref{D1'}) for $|j|\leq N-1$.
Finally, for $j=N$ we have 
  \begin{align*}
\|B(\tau)[\pi^\varepsilon_j f]-\mathbb{B}_{j,\tau}[\pi^\varepsilon_j f]\|_{W^{s-1}_p}&=\|H[(a_\tau-\delta)\chi_j^\varepsilon(\pi^\varepsilon_j f)']\|_{W^{s-1}_p}\\[1ex]
 &\leq  C\|\phi\chi_j^\varepsilon\|_\infty\|(\pi^\varepsilon_j f)'\|_{W^{s-1}_p}+K\|f\|_{W^{s'}_p}\\[1ex]
 &\leq \mu \|\pi_j^\varepsilon f\|_{W^s_p}+K\|  f\|_{W^{s'}_p},
 \end{align*}
 provided that $\varepsilon$ is sufficiently small, and (\ref{D1'}) holds also for $j=N$.\medskip
 
 \noindent{\em Step 2.} 
 Let $\eta\in (0,1)$ be chosen such that the function $a_\tau$ satisfies
 \[
\eta\leq a_\tau\leq 1/\eta,\qquad \tau\in[0,1].
 \]
   Lemma \ref{L:GAP} implies there exists a constant $\kappa=\kappa(\eta)\geq1$ such that the Fourier multipliers
 \begin{align*} 
 \mathbb{B}_{\alpha}:= \alpha\Big(-\frac{d^2}{dx^2}\Big)^{1/2},  \qquad \alpha\in[\eta,\eta^{-1}],
 \end{align*}
 satisfy
 \begin{align}
  \kappa\|(\lambda+\mathbb{B}_{\alpha})[f]\|_{W^{s-1}_p}\geq |\lambda|\cdot\|f\|_{W^{s-1}_p}+\|f\|_{W^s_p}, \qquad f\in W^{s}_p(\mathbb{R}),\,  \lambda\geq 1\label{L:FM'}.
\end{align}
Let $\varepsilon>0$ be   determined in the previous step for $\mu:=(2\kappa)^{-1}$.
It then holds     
 \begin{align*}
  2\kappa \|\pi_j^\varepsilon (\lambda+B(\tau)) [ f]\|_{W^{s-1}_p}&\geq 2\kappa\|(\lambda+\mathbb{B}_{j,\tau})[\pi^\varepsilon_j f]\|_{W^{s-1}_p}-  2\kappa \|\pi_j^\varepsilon B(\tau) [ f]-\mathbb{B}_{j,\tau}[\pi^\varepsilon_j f]\|_{W^{s-1}_p}\\[1ex]
  &\geq   \|\pi_j^\varepsilon f\|_{W^s_p}+2 \lambda\|\pi_j^\varepsilon f\|_{W^{s-1}_p}-2\kappa K\|  f\|_{W^{s'}_p}
 \end{align*}
 for $-N+1\leq j\leq N$, $\tau\in[0,1]$, and $\lambda\geq1.$
 Summing up over $j$, we conclude together with Lemma \ref{L:EN}, (\ref{IP}), and Young's inequality   there are constants $\kappa_0\geq 1$   
 and $\omega_0>0$ such that 
  \begin{align*} 
  \kappa_0\| (\lambda+B(\tau)) [ f]\|_{W^{s-1}_p}&\geq   \| f\|_{W^s_p}+\lambda\| f\|_{W^{s-1}_p}
 \end{align*}
 for all   $\tau\in[0,1]$, $\lambda\geq \omega_0,$ and $f\in W^s_p(\mathbb{R})$. This proves
  (\ref{EDY}) and the proof is complete. 
\end{proof}

We are now in a position to prove Theorem \ref{T:AG}.

\begin{proof}[Proof of Theorem \ref{T:AG}] Let $s'\in(1+1/p,s).$
 Let  further $\kappa_0\geq1$ be the constant found in Lemma~\ref{L:GAP} and set $\mu:=1/2\kappa_0$.
 Theorem \ref{T:AP} implies there exist a constant  $\varepsilon\in(0,1) $,  an~$\varepsilon$-localization family $\{\pi_j^\varepsilon\,:\, -N+1\leq j\leq N\}$, a constant $K=K(\varepsilon,f_0)>0$
 and  bounded operators $\mathbb{A}_{j,\tau}\in\mathcal{L}(W^s_p(\mathbb{R}), W^{s-1}_p(\mathbb{R}))$, $ -N+1\leq j\leq N$ and $\tau\in[0,1],$ satisfying 
 \begin{equation*} 
  2\kappa_0\|\pi_j^\varepsilon\Psi(\tau )[f]-\mathbb{A}_{j,\tau}[\pi^\varepsilon_j f]\|_{W^{s-1}_p}\leq \|\pi_j^\varepsilon f\|_{W^{s}_p}+2\kappa_0 K\|  f\|_{W^{s'}_p},\qquad f\in W^s_p(\mathbb{R}).
 \end{equation*}
Furthermore, Lemma \ref{L:GAP}  yields
  \begin{equation*} 
    2\kappa_0\|(\lambda-\mathbb{A}_{j,\tau})[\pi^\varepsilon_jf]\|_{W^{s-1}_p}\geq 2|\lambda|\cdot\|\pi^\varepsilon_jf\|_{W^{s-1}_p}+ 2\|\pi^\varepsilon_j f\|_{W^s_p}
 \end{equation*}
 for all $-N+1\leq j\leq N$, $\tau\in[0,1],$  $\re \lambda\geq 1$, and  $f\in W^s_p(\mathbb{R})$.
 The latter inequalities lead to
 \begin{align*}
   2\kappa_0\|\pi_j^\varepsilon(\lambda-\Psi(\tau ))[f]\|_{W^{s-1}_p}\geq& 2\kappa_0\|(\lambda-\mathbb{A}_{j,\tau})[\pi^\varepsilon_jf]\|_{W^{s-1}_p}-2\kappa_0\|\pi_j^\varepsilon\Psi(\tau)[f]-\mathbb{A}_{j,\tau}[\pi^\varepsilon_j f]\|_{W^{s-1}_p}\\[1ex]
   \geq& 2|\lambda|\cdot\|\pi^\varepsilon_j f\|_{W^{s-1}_p}+ \|\pi^\varepsilon_j f\|_{W^s_p}-2\kappa_0K\|  f\|_{W^{s'}_p}
 \end{align*}
for all $-N+1\leq j\leq N$, $\tau\in[0,1],$   $\re \lambda\geq 1$, and  $f\in W^s_p(\mathbb{R})$.
 Summing  up over $j$, Lemma~\ref{L:EN}, relation (\ref{IP}), and Young's inequality  imply there exist constants  $\kappa=\kappa(f_0)\geq1$  and $\omega_1=\omega_1( f_0)>0 $ such that 
  \begin{align}\label{KDED}
   \kappa\|(\lambda-\Psi(\tau ))[f]\|_{W^{s-1}_p}\geq |\lambda|\cdot\|f\|_{W^{s-1}_p}+ \| f\|_{W^s_p}
 \end{align}
for all   $\tau\in[0,1],$   $\re \lambda\geq \omega_1$, and  $f\in W^s_p(\mathbb{R})$.

Let  $\omega_0>0$ denote the constant from Proposition  \ref{P:IP0} found for $\delta:=C_{\Theta}$ and ${\phi:=a_\mu\Phi(f_0)}$.
Setting  $\omega:=\max\{\omega_0,\omega_1\},$ it holds
  $\omega-\Psi(0) \in {\rm Isom}(W^s_p(\mathbb{R}), W^{s-1}_p(\mathbb{R})),$ cf. Lemma~\ref{L:GAP}.
 The method of continuity together with (\ref{KDED}) yields that 
\begin{align}\label{DEDK2}
   \omega-\Psi(1)\in {\rm Isom}(W^s_p(\mathbb{R}), W^{s-1}_p(\mathbb{R})).
 \end{align}
Gathering  (\ref{KDED}) (with $\tau=1$) and (\ref{DEDK2}) it follows  that $-\partial\Phi(f_0)\in\mathcal{H}(W^s_p(\mathbb{R}),W^{s-1}_p(\mathbb{R})),$ cf. \cite[Chapter I]{Am95}
 and the proof is complete.
\end{proof}

We conclude this section with the proof of our main result.
The well-posedness result follows by applying abstracts result for fully nonlinear parabolic problems from \cite{L95}.
It is important to point out that in fact we can establish the  uniqueness of solutions in the setting of strict solutions (as stated in Theorem \ref{MT}), which is an improvement compared to the theory in \cite{L95}.
This feature is essential when proving the claim  (ii) of Theorem~\ref{MT}, as it enables us to 
use a parameter trick which was successfully applied also to other problems, cf., e.g., \cite{An90, ES96, PSS15, MBV19}.

\begin{proof}[Proof of Theorem \ref{MT}] {\em  Well-posedness:}
In view of (\ref{REG}) and Theorem \ref{T:AG}, we find that the assumptions of \cite[Theorem 8.1.1]{L95} 
are satisfied in the context of the evolution problem~(\ref{NNEP}) when restricting $\Phi$ to the open set $\mathcal{O}$.
Hence, given $f_0\in \mathcal{O}$, there exists $T>0$ and a solution~$f(\cdot;f_0)$ to (\ref{NNEP}) that satisfies
\[ f\in {\rm C}([0,T],\mathcal{O})\cap {\rm C}^1([0,T], W^{s-1}_p(\mathbb{R}))\cap {\rm C}^{\alpha}_{\alpha}((0,T], W^s_p(\mathbb{R}))\] 
for some  $\alpha\in(0,1)$ (actually, since the problem is autonomous, for all $\alpha\in(0,1)$).
Moreover, the solution is unique within the class 
\[
  \bigcup_{\alpha\in(0,1)}{\rm C}^{\alpha}_{\alpha}((0,T], W^s_p(\mathbb{R})) \cap {\rm C}([0,T],\mathcal{O})\cap {\rm C}^1([0,T], W^{s-1}_p(\mathbb{R})).
 \]
In fact the solution is unique in ${\rm C}([0,T],\mathcal{O})\cap {\rm C}^1([0,T], W^{s-1}_p(\mathbb{R})).$  
Indeed, assuming there are two solutions $f,\,   \widetilde f:[0,T]\to\mathcal{O}$ corresponding to the same initial data $f_0\in\mathcal{O}$, since the problem (\ref{NNEP}) is autonomous,
 we can assume   $f(t)\neq \widetilde f(t)$ for $t\in(0,T]$. 
 Let $  s'\in (1+1/p,s)$ and set $\alpha:= s-s'\in(0,1)$. 
 In virtue of (\ref{IP}) there exists a constant $C>0$ such that
 \begin{equation}\label{AdReg}
\|f(t_1)-f(t_2)\|_{W^{s'}_p} +\|\widetilde f(t_1)-\widetilde f(t_2)\|_{W^{s'}_p} \leq C|t_1-t_2|^\alpha,\qquad t_1,\, t_2\in[0, T],
 \end{equation}
 and therefore $f,\, \widetilde f\in {\rm C}^{\alpha}([0,T], W^{s'}_p(\mathbb{R}))\hookrightarrow {\rm C}^{\alpha}_{\alpha}((0,T], W^{s'}_p(\mathbb{R}))$.
We may now apply the abstract result \cite[Theorem 8.1.1]{L95} in the context of  (\ref{NNEP}) with $\Phi\in {\rm C}^\omega(\widetilde{\mathcal{O}}, W^{s'-1}_p(\mathbb{R})),$
where
\begin{align*}
 \widetilde{\mathcal{O}}:= \{f\in W^{s'}_p(\mathbb{R}) \,:\, C_{\Theta}+a_\mu\Phi(f)>0\}.
\end{align*}
 Since $f_0\in \widetilde{\mathcal{O}}$, we get in virtue of (\ref{AdReg}), that $f=\widetilde f$ on $[0,T],$ hence our assumption was~false.

 Finally,  the unique solution can be extended up to a maximal existence time $T_+(f_0)$, see \cite[Section 8.2]{L95}.
  In virtue of \cite[Proposition 8.2.3]{L95}  the solution map also defines  a semiflow on $\mathcal{O}$, and  it remains to establish (ii). \medskip

\noindent{\em  Parabolic smoothing:} Given $\lambda:=(\lambda_1,\lambda_2)\in(0,\infty)\times\mathbb{R}$ and a maximal  solution $f=f(\cdot;f_0)$ with maximal existence time $T_+=T_+(f_0)$ to (\ref{NNEP}), let  
\[
f_{\lambda}(t,x):=f(\lambda_1 t,x+\lambda_2 t), \qquad   x\in\mathbb{R}, \, 0\leq t\leq T_{+,\lambda}:=T_+/\lambda_1.
\]
Straightforward calculations show that $f_{\lambda}\in {\rm C}([0,T_{+,\lambda}),\mathcal{O})\cap {\rm C}^1([0,T_{+,\lambda}), W^{s-1}_p(\mathbb{R})) $
is a solution to  the evolution problem  
\begin{align}\label{QC}
\frac{df}{dt}= \Psi(f,\lambda) ,\quad t\geq0,\qquad f(0)=f_0,
\end{align}
 where $ \Psi:\mathcal{O}\times (0,\infty)\times\mathbb{R}\subset W^s_p(\mathbb{R})\times \mathbb{R}^2\to  W^{s-1}_p(\mathbb{R})$ is defined by  
\begin{align*}
  \Psi(f,\lambda):=\lambda_1\Phi(f)+\lambda_2\frac{ df}{dx}.
\end{align*}
Using (\ref{REG}), we get $\Psi\in {\rm C}^\omega(\mathcal{O}\times (0,\infty)\times\mathbb{R}, W^{s-1}_p(\mathbb{R}))$.
Also, given $(f_0,\lambda)\in \mathcal{O}\times (0,\infty)\times\mathbb{R}$,
the partial derivative of $\Psi$ with respect to $f$ is 
\[
\partial_f\Psi(f_0,\lambda)=\lambda_1\partial\Phi(f_0)+\lambda_2\frac{ d}{dx}.
\]
Since $d/dx$ is a  Fourier multiplier   with   symbol $m(\xi)= i\xi$, $\xi\in\mathbb{R}$,
the results leading to Theorem~\ref{T:AG} can be easily adapted to obtain that    $-\partial_f\Psi(f_0,\lambda)$ belongs to $\mathcal{H}(W^s_p(\mathbb{R}), W^{s-1}_p(\mathbb{R}))$
for all $(f_0,\lambda)\in \mathcal{O}\times (0,\infty)\times\mathbb{R}.$
According to \cite[Theorem~8.1.1 and Theorem~8.3.9]{L95} and arguing as in the proof of (i),
 it follows that  (\ref{QC}) has for each $(f_0,\lambda)\in \mathcal{O}\times (0,\infty)\times\mathbb{R}$ 
 a unique strict solution 
 $$f=f(\,\cdot\,; f_0,\lambda)\in {\rm C}([0,\widetilde T_+),\mathcal{O})\cap {\rm C}^1([0,\widetilde T_+),W^{s-1}_p(\mathbb{R})),$$ 
 where $\widetilde T_+=T_+(f_0,\lambda)\in(0,\infty]$ is the maximal existence time.
 Moreover, the set
\[
\Omega:=\{(t,f_0,\lambda)\,:\, (f_0,\lambda)\in \mathcal{O}\times (0,\infty)\times\mathbb{R},\, 0<t< T_+(f_0,\lambda)\}
\]
is open and 
\[
[(t,f_0,\lambda)\mapsto f(t; f_0,\lambda)]\in {\rm C}^\omega(\Omega, \mathcal{O}).
\]

Hence, given  $f_0\in\mathcal{O}$, we may conclude that  
\begin{align*}
T_+(f_0,\lambda)=\frac{T_+(f_0)}{\lambda_1}\qquad\mbox{and}\qquad   f(t; f_0,\lambda)=f_{\lambda}(t),\quad 0\leq t<\frac{T_+(f_0)}{\lambda_1}.
\end{align*}
In particular, given $t_0<T_+(f_0)$, we may choose $\delta>0$ such that  $t_0<T_+(f_0,\lambda)$ for all $\lambda $ belonging to the disc $ D_\delta((1,0))$, and therewith
\begin{align}\label{RAM}
[\lambda\mapsto f_{\lambda}(t_0)]:D_\delta((1,0))\to W^s_p(\mathbb{R})
\end{align}
is  also a  real-analytic map. Repeated differentiation with respect to $\lambda_2$ immediately yields~(iib).
Let now $x_0\in\mathbb{R}$. Since $[h\mapsto h(x_0)]:W^s_p(\mathbb{R})\to\mathbb{R}$ is real-analytic, then so is
   \[
[\lambda\mapsto f(\lambda_1 t_0, x_0+\lambda_2t_0)]:D_\delta((1,0))\to  \mathbb{R}.
   \]
Besides, if $\varepsilon>0$  small,   the mapping $\varphi:D_\varepsilon((t_0,x_0))\to D_\delta ((1,0))$ with
  \begin{align*}
  \varphi(t,x):= \Big(\frac{t}{t_0},\frac{x-x_0}{t_0}\Big)
  \end{align*}
is well-defined and real-analytic, and composing it with the previous function shows that 
$$ [(t,x)\mapsto  f(t,x)]:D_\varepsilon ((t_0,x_0))\to\mathbb{R},$$
is also real-analytic. This proves (iia).  
\end{proof}

\appendix
 \section{Preparatory results used in Section \ref{Sec:4}}\label{Sec:C}
In this section we present the proofs of Lemmas \ref{L:B2}-\ref{L:C2}.

\begin{proof}[Proof of Lemma \ref{L:B2}]
We may assume that $h\in C^\infty_0(\mathbb{R})$. 
Setting
\[
T:=\varphi B_{n,m}(f,\ldots,f)[f,\ldots,f, h]- B_{n,m}(f,\ldots,f)[f,\ldots,f, \varphi h],
\]
it follows from Lemma \ref{L:MP1}~(i) that
\begin{align}\label{EST:P0}
\|T\|_p\leq K\|h\|_p.
\end{align}
Moreover, given $0\neq\xi\in\mathbb{R}$, it holds that
\begin{align*}
\frac{\tau_\xi T-T}{\xi} = T_1+T_2+T_3+T_4+T_5,
\end{align*}
where, using (\ref{spr2}) and (\ref{spr3}), it holds
\begin{align*}
T_1 &:=  \frac{\tau_\xi\varphi-\varphi}{\xi}B_{n,m}(\tau_\xi f,\ldots,\tau_\xi f)[\tau_\xi f,\ldots,\tau_\xi f, \tau_\xi h],\\[1ex]
T_2&:=\varphi B_{n,m}(\tau_\xi f,\ldots,\tau_\xi f)\Big[\tau_\xi f,\ldots,\tau_\xi f, \frac{\tau_\xi h-h}{\xi}\Big],\\[1ex]
T_3&:=-B_{n,m}(\tau_\xi f,\ldots,\tau_\xi f)\Big[\tau_\xi f,\ldots,\tau_\xi f, \frac{\tau_\xi (\varphi h)-\varphi h}{\xi}\Big],\\[1ex]
T_4&:= \frac{\tau_\xi f-f}{\xi}\sum_{i=1}^n B_{n,m}(\tau_\xi f,\ldots,\tau_\xi f)[ \underset{\mbox{$i$-times}}{\underbrace{f,\ldots,f}},\varphi,\tau_\xi f,\ldots \tau_\xi f, h]\\[1ex]
&\hspace{0,5cm}-\sum_{i=1}^n B_{n,m}(\tau_\xi f,\ldots,\tau_\xi f)\Big[ \underset{\mbox{$i$-times}}{\underbrace{f,\ldots,f}},\varphi,\tau_\xi f,\ldots \tau_\xi f,\frac{\tau_\xi f-f}{\xi} h\Big],
\end{align*}
\begin{align*}
T_5&:=-   \frac{\tau_\xi f-f}{\xi}\sum_{i=1}^m B_{n+2,m+1}( \underset{\mbox{$i$-times}}{\underbrace{\tau_\xi f,\ldots,\tau_\xi f}},f,\ldots,f)\Big[\varphi, \tau_\xi f+f,f,\ldots,f, h\Big]\\[1ex]
&\hspace{0,5cm}+ \sum_{i=1}^m B_{n+2,m+1}( \underset{\mbox{$i$-times}}{\underbrace{\tau_\xi f,\ldots,\tau_\xi f}},f,\ldots,f)\Big[ \varphi, \tau_\xi f+f,f,\ldots,f,   \frac{\tau_\xi f-f}{\xi} h\Big].
\end{align*}
Lemma \ref{L:MP1}~(i) implies the limit  $\lim_{\xi\to0 }(\tau_\xi T-T)/\xi$
exists in $L_p(\mathbb{R})$. Hence $T\in W^1_p(\mathbb{R})$ and
\begin{align*}
T'&=\varphi' B_{n,m}(f,\ldots, f)[ f,\ldots,f,  h]+ \varphi B_{n,m}( f,\ldots, f)[f,\ldots, f, h']\\[1ex]
&\hspace{0,45cm}-B_{n,m}( f,\ldots, f)[f,\ldots,f, \varphi h']-B_{n,m}( f,\ldots, f)[f,\ldots,f, \varphi' h]\\[1ex]
&\hspace{0,45cm}+nf'B_{n,m}(  f,\ldots,  f)[ \varphi, f,\ldots,f , h]-n B_{n,m}(  f,\ldots,  f)[ \varphi, f,\ldots,f , f'h]\Big]\\[1ex]
&\hspace{0,45cm}-   2mf' B_{n+2,m+1}(  f,\ldots,  f)[ \varphi, f,\ldots,f , h]+2m B_{n+2,m+1}(  f,\ldots,  f)[ \varphi, f,\ldots,f , f'h].
\end{align*}
Using again Lemma \ref{L:MP1}~(i), we get
\begin{align}\label{EST:P1}
\|T'-\varphi B_{n,m}( f,\ldots, f)[f,\ldots, f, h']+B_{n,m}( f,\ldots, f)[f,\ldots,f, \varphi h']\|_p\leq K\|h\|_p.
\end{align}
It remains to estimate the term
\[
T_6:=\varphi B_{n,m}( f,\ldots, f)[f,\ldots, f, h']-B_{n,m}( f,\ldots, f)[f,\ldots,f, \varphi h'].
\]
Since 
\begin{align*}
T_6(x)&=\int_\mathbb{R}\frac{(\delta_{[x,y]}f/y)^n}{(1+(\delta_{[x,y]}f/y)^2)^m}\frac{\delta_{[x,y]}\varphi}{y}\frac{d}{dy}(-h(x-y))\, dy,\qquad x\in\mathbb{R},
\end{align*}
integration by parts leads to the following representation
\begin{align*}
T_{6} &= B_{n,m} (f,\ldots,f)[f,\ldots,f,  \varphi' h]-B_{n+1,m} (f,\ldots,f)[f,\ldots,f,\varphi,   h]\\[1ex]
 &\hspace{0,45cm}+   nB_{n,m} (f,\ldots,f)[f,\ldots,f,  \varphi,f'h]-nB_{n+1,m} (f,\ldots,f)[f,\ldots,f,\varphi,   h]\\[1ex]
  &\hspace{0,45cm}-2m  B_{n+2,m+1} (f,\ldots,f)[f,\ldots,f,  \varphi,f'h]+2mB_{n+3,m+1} (f,\ldots,f)[f,\ldots,f,\varphi,h]
\end{align*}
and Lemma \ref{L:MP1}~(i)  yields
\begin{align}\label{EST:P2}
\|T_6\|_p\leq K\|h\|_p.
\end{align}
Gathering (\ref{EST:P0})-(\ref{EST:P2}), we arrive at (\ref{LB2}) and the proof is complete.
\end{proof}

We now establish  Lemma \ref{L:C1}.
 
\begin{proof}[Proof of Lemma \ref{L:C1}]
We  first deal with the case $|j|\leq N-1$ and write 
\begin{align*}
\pi_j^\varepsilon\overline{\omega} B_{n,m}(f,\ldots,f)[f,\ldots,f, h]-\frac{\overline{\omega}(x_j^\varepsilon)(f'(x_j^\varepsilon))^n}{[1+(f'(x_j^\varepsilon))^2]^m}B_{0,0}[\pi_j^\varepsilon h]=T_1+\overline{\omega}(x_j^\varepsilon)T_2 ,
\end{align*}
with
\begin{align*}
T_1&:=\pi_j^\varepsilon\overline{\omega} B_{n,m}(f,\ldots,f)[f,\ldots,f, h]-\overline{\omega}(x_j^\varepsilon) B_{n,m}(f,\ldots,f)[f,\ldots,f, \pi_j^\varepsilon h],\\[1ex]
T_2&:=  B_{n,m}(f,\ldots,f)[f,\ldots,f, \pi_j^\varepsilon h]-\frac{(f'(x_j^\varepsilon))^n}{[1+(f'(x_j^\varepsilon))^2]^m}B_{0,0}[\pi_j^\varepsilon h].
\end{align*}

\noindent{\it The term $T_1$.} 
In view of $\chi_j^\varepsilon\pi_j^\varepsilon=\pi_j^\varepsilon$ we decompose $T_1=T_{1a}+\overline{\omega}(x_j^\varepsilon)T_{1b}$, where
\begin{align*}
T_{1a}&:=\chi_j^\varepsilon(\overline{\omega}-\overline{\omega}(x_j^\varepsilon)) \pi_j^\varepsilon B_{n,m}(f,\ldots,f)[f,\ldots,f, h],\\[1ex]
T_{1b}&:= \pi_j^\varepsilon  B_{n,m}(f,\ldots,f)[f,\ldots,f, h]- B_{n,m}(f,\ldots,f)[f,\ldots,f, \pi_j^\varepsilon h].
\end{align*}
Applying Lemma \ref{L:B2}, we get 
\begin{align}\label{T1b}
\|T_{1b}\|_{W^{s-1}_p}\leq K\|h\|_{p},\qquad -N+1\leq j\leq N.
\end{align}
Moreover, recalling (\ref{MES}), 
it follows from   (\ref{T1b}), Lemma \ref{L:MP3} (with $r=s$), and Lemma \ref{L:MP1}~(ii) (with $r=s$ and $\tau=s'-1$)    that
\begin{align*}
\|T_{1a}\|_{W^{s-1}_p}&\leq 2\|\chi_j^\varepsilon(\overline{\omega}-\overline{\omega}(x_j^\varepsilon))\|_\infty\| \pi_j^\varepsilon B_{n,m}(f,\ldots,f)[f,\ldots,f, h]\|_{W^{s-1}_p}\\[1ex]
&\hspace{0,45cm}+K\|B_{n,m}(f,\ldots,f)[f,\ldots,f, h]\|_{\infty}\\[1ex]
&\leq 2\|\chi_j^\varepsilon(\overline{\omega}-\overline{\omega}(x_j^\varepsilon))\|_\infty\| B_{n,m}(f,\ldots,f)[f,\ldots,f, \pi_j^\varepsilon h]\|_{W^{s-1}_p}+K\|h\|_{W^{s'-1}_p}\\[1ex]
&\leq\frac{\nu}{2}\|\pi_j^\varepsilon h\|_{W^{s-1}_p}+K\|h\|_{W^{s'-1}_p},
\end{align*}
provided   $\varepsilon$ is sufficiently small, and therewith
\begin{align}\label{C1a}
\|T_{1}\|_{W^{s-1}_p}\leq\frac{\nu}{2}\|\pi_j^\varepsilon h\|_{W^{s-1}_p}+K\|h\|_{W^{s'-1}_p}.
\end{align}

\noindent{\it The term $T_2$.} We use again the identity $\chi_j^\varepsilon\pi_j^\varepsilon=\pi_j^\varepsilon$ and write $T_2=T_{2a}+T_{2b}$, where
\begin{align*}
T_{2a}&:=\frac{(f'(x_j^\varepsilon))^n}{[1+(f'(x_j^\varepsilon))^2]^m}(\chi_j^\varepsilon B_{0,0}[\pi_j^\varepsilon h]- B_{0,0}[\chi_j^\varepsilon(\pi_j^\varepsilon h)])\\[1ex]
&\hspace{0,45cm}-( \chi_j^\varepsilon B_{n,m} (f,\ldots,f)[f,\ldots,f,  \pi_j^\varepsilon h]-  B_{n,m} (f,\ldots,f)[f,\ldots,f,  \chi_j^\varepsilon(\pi_j^\varepsilon h)]),\\[1ex]
T_{2b}&:=  \chi_j^\varepsilon B_{n,m} (f,\ldots,f)[f,\ldots,f,  \pi_j^\varepsilon h]-\frac{(f'(x_j^\varepsilon))^n}{[1+(f'(x_j^\varepsilon))^2]^m}\chi_j^\varepsilon B_{0,0}[\pi_j^\varepsilon h].
\end{align*}
  Lemma \ref{L:B2} yields
\begin{equation}\label{C2a}
\|T_{2a}\|_{W^{s-1}_p}\leq K\|h\|_{p}.
\end{equation}

It remains to estimate $T_{2b}.$ 
We first use Lemma \ref{L:MP1}~(i) to deduce that 
\begin{align}\label{C2a'}
\|T_{2b}\|_p\leq K\|h\|_{p}.
\end{align}
Moreover, noticing that $f'(x_j^\varepsilon)=\delta_{[x,y]} (f'(x_j^\varepsilon)\id_{\mathbb{R}})/y$ and recalling (\ref{spr3}), we write 
\begin{align*}
T_{2b}&=\sum_{k=0}^{n-1}(f'(x_j^\varepsilon))^{n-k-1}\chi_j^\varepsilon B_{k+1,m}(f,\ldots,f)[f,\ldots,f, f-f'(x_j^\varepsilon)\id_{\mathbb{R}},\pi_j^\varepsilon h]\\[1ex]
&\hspace{0,45cm}- \sum_{k=0}^{m-1}\frac{(f'(x_j^\varepsilon))^{n}}{[1+ ( f'(x_j^\varepsilon))^2]^{m-k}}\chi_j^\varepsilon B_{2,k+1}(f,\ldots,f)[f-f'(x_j^\varepsilon)\id_{\mathbb{R}}, f+f'(x_j^\varepsilon)\id_{\mathbb{R}},\pi_j^\varepsilon h].
\end{align*} 
Let $T_k:=\chi_j^\varepsilon B_{k+1,m}(f,\ldots,f)[f,\ldots,f, f-f'(x_j^\varepsilon)\id_{\mathbb{R}},\pi_j^\varepsilon h]$ for $0\leq k \leq n-1$. 
In order to estimate the ${W^{s-1}_p}$-seminorm of $T_k$ we write for $\xi\in\mathbb{R}$ 
\begin{align*}
T_k-\tau_\xi T_k=T_{kA}+T_{kB}+\chi_j^\varepsilon T_{kC},
\end{align*}
where, appealing again to   (\ref{spr3}), it holds
\begin{align*}
T_{kA}&:=(\chi_j^\varepsilon-\tau_\xi \chi_j^\varepsilon)\tau_\xi B_{k+1,m}(f,\ldots,f)[f,\ldots,f, f-f'(x_j^\varepsilon)\id_{\mathbb{R}},\pi_j^\varepsilon h],\\[1ex]
T_{kB}&:=  \chi_j^\varepsilon B_{k+1,m}(f,\ldots,f)[f,\ldots,f, f-f'(x_j^\varepsilon)\id_{\mathbb{R}},\pi_j^\varepsilon h -\tau_\xi (\pi_j^\varepsilon h)],\\[1ex]
T_{kC}&:= \sum_{j=1}^{k}B_{k+1,m}(f,\ldots,f)[\underset{(j-1)-{\rm times}}{\underbrace{\tau_\xi f,\ldots,\tau_\xi f}},f-\tau_\xi f,f,\ldots, f, f-f'(x_j^\varepsilon)\id_{\mathbb{R}},\tau_\xi (\pi_j^\varepsilon h)]\\[1ex]
&\hspace{0,45cm}+B_{k+1,m}(f,\ldots,f)[\tau_\xi f,\ldots,\tau_\xi f, f-\tau_\xi f,\tau_\xi (\pi_j^\varepsilon h)]\\[1ex]
&\hspace{0,45cm}+\sum_{j=1}^mB_{k+3,m+1}^j [\tau_\xi f,\ldots,\tau_\xi f, \tau_\xi f-f'(x_j^\varepsilon)\id_{\mathbb{R}},\tau_\xi f+f,\tau_\xi f-f,\tau_\xi (\pi_j^\varepsilon h)],
\end{align*}
with
$$B_{k+3,m+1}^j:=B_{k+3,m+1}(\underset{j-{\rm times}}{\underbrace{f,\ldots,f}},\tau_\xi f,\ldots,\tau_\xi f).$$
Observing that
\[
T_{kA}=(\chi_j^\varepsilon-\tau_\xi \chi_j^\varepsilon)\tau_\xi\big( B_{k+1,m}(f,\ldots,f)[f,\ldots,f,\pi_j^\varepsilon h]-f'(x_j^\varepsilon)  B_{k,m}(f,\ldots,f)[f,\ldots,f,\pi_j^\varepsilon h]\big),
\]
Lemma \ref{L:MP1}~(ii) (with $r=s$ and $\tau=s'-1$) yields
\begin{equation}\label{2E1}
\|T_{kA}\|_p\leq K\|\chi_j^\varepsilon-\tau_\xi\chi_j^\varepsilon\|_p\|\pi_j^\varepsilon h\|_{W^{s'-1}_p}\leq K \|\chi_j^\varepsilon-\tau_\xi\chi_j^\varepsilon\|_p\| h\|_{W^{s'-1}_p}.
\end{equation}
In order to estimate  $T_{kB},$ let $F$ denote the Lipschitz continuous function  defined by $F=f$ on $\supp\chi_j^\varepsilon$ and $F'=f'(x_j^\varepsilon)$ on $\mathbb{R}\setminus\supp\chi_j^\varepsilon$.
If $|\xi|\geq\varepsilon,$ we infer from Lemma \ref{L:MP1}~(i)  that 
\begin{align}\label{2E21}
\|T_{kB}\|_p\leq   K\|h \|_{p}.
\end{align}
If $|\xi|<\varepsilon$, then $\xi+ \supp\pi_j^\varepsilon\subset \supp\chi_j^\varepsilon$, and  Lemma \ref{L:MP1}~(i)  and the properties defining~$F$ lead to
\begin{equation}\label{2E22}
\begin{aligned}
\|T_{kB}\|_p&=  \| \chi_j^\varepsilon B_{k+1,m}(f,\ldots,f)[f,\ldots,f, F-f'(x_j^\varepsilon)\id_{\mathbb{R}},\pi_j^\varepsilon h -\tau_\xi (\pi_j^\varepsilon h)]\|_p\\[1ex]
&\leq C\|  f'-f'(x_j^\varepsilon)\|_{L_\infty(\supp\chi_j^\varepsilon )}\|\pi_j^\varepsilon h-\tau_\xi (\pi_j^\varepsilon h)\|_p\\[1ex]
&\leq \frac{\nu}{12(n+1)C_0^{n}}\|\pi_j^\varepsilon h-\tau_\xi (\pi_j^\varepsilon h)\|_p,
\end{aligned}
\end{equation}
provided that $\varepsilon$ is sufficiently small, where $C_0:=1+ \|\overline{\omega}\|_\infty+\|f'\|_\infty$. 

Finally, Lemma \ref{L:MP2} (with $r=s'$) yields
\begin{equation}\label{2E2}
\|\chi_j^\varepsilon T_{kC}\|_p\leq K\|f'-\tau_\xi f'\|_p\| \pi_j^\varepsilon h \|_{W^{s'-1}_p}\leq K\|f'-\tau_\xi f'\|_p\|h\|_{W^{s'-1}_p}.
\end{equation}

The estimates (\ref{2E1})-(\ref{2E22}) combined imply that
\begin{align*}
 [T_k ]_{W^{s-1}_p} \leq \frac{\nu}{4(n+1)C_0^{n}}\|\pi_j^\varepsilon h \|_{W^{s-1}_p}+\|h\|_{W^{s'-1}_p}.
\end{align*}
The arguments used to estimate $T_k$ show also that 
\begin{align*}
 &\big[\chi_j^\varepsilon B_{2,k+1}(f,\ldots,f)[f-f'(x_j^\varepsilon)\id_{\mathbb{R}}, f+f'(x_j^\varepsilon)\id_{\mathbb{R}},\pi_j^\varepsilon h]\big]_{W^{s-1}_p} \\[1ex]
 &\hspace{2cm}\leq \frac{\nu}{4(m+1)C_0^{n+1}}\|\pi_j^\varepsilon h \|_{W^{s-1}_p}+\|h\|_{W^{s'-1}_p},
\end{align*}
provided that $\varepsilon$ is chosen sufficiently small. 
Recalling also (\ref{C2a'}), we obtain for such $\varepsilon$ that
\begin{align*}
\|T_{2b}\|_{W^{s-1}_p}\leq \frac{\nu}{2C_0}\|\pi_j^\varepsilon h \|_{W^{s-1}_p}+\|h\|_{W^{s'-1}_p},
\end{align*} 
and together with (\ref{C1a}) and (\ref{C2a}) we have established (\ref{LC1}).\medskip
\end{proof}

We continue with the proof of Lemma \ref{L:C1a}.

\begin{proof}[Proof of Lemma~\ref{L:C1a}]
 Let now $j=N$. Since $\chi_j^\varepsilon\pi_j^\varepsilon=\pi_j^\varepsilon$, it holds that 
\begin{align*}
\pi_j^\varepsilon\overline{\omega} B_{n,m}(f,\ldots,f)[f,\ldots,f, h]=T_1+T_2 ,
\end{align*}
where
\begin{align*}
T_1&:=\chi_j^\varepsilon\overline{\omega} \big(\pi_j^\varepsilon B_{n,m}(f,\ldots,f)[f,\ldots,f, h]- B_{n,m}(f,\ldots,f)[f,\ldots,f, \pi_j^\varepsilon h]\big),\\[1ex]
T_2&:= \chi_j^\varepsilon\overline{\omega}  B_{n,m}(f,\ldots,f)[f,\ldots,f, \pi_j^\varepsilon h].
\end{align*}
Lemma \ref{L:MP3} (with $r=s$) together with Lemma \ref{L:MP1}~(i) (with $r=s$ and $\tau=s'-1$) and (\ref{MES}) yields 
\begin{align*}
\|T_2\|_{W^{s-1}_p}&\leq 2\|\chi_j^\varepsilon\overline{\omega} \|_\infty\|B_{n,m}(f,\ldots,f)[f,\ldots,f, \pi_j^\varepsilon h]\|_{W^{s-1}_p}\\[1ex]
&\hspace{0,45cm}+K\|B_{n,m}(f,\ldots,f)[f,\ldots,f, \pi_j^\varepsilon h]\|_\infty\\[1ex]
&\leq C\|\chi_j^\varepsilon\overline{\omega} \|_\infty\|\pi_j^\varepsilon h\|_{W^{s-1}_p}+K\|\pi_j^\varepsilon h\|_{W^{s'-1}_p}\\[1ex]
&\leq \nu[\pi_j^\varepsilon h ]_{W^{s-1}_p}+K\|  h\|_{W^{s'-1}_p},
\end{align*}
provided that $\varepsilon$ is sufficiently small. 
We have made use here of the fact that if $\varepsilon$ is sufficiently small, then $\|\overline{\omega}\|_{L_\infty(\supp\chi_j^\varepsilon)}<\nu$.
Furthermore,  Lemma \ref{L:B2} shows that 
\begin{align*}
\|T_1\|_{W^{s-1}_p}\leq K\| h\|_{p} 
\end{align*}
and (\ref{LC2}) follows.
\end{proof}

We conclude this appendix with the proof of Lemma \ref{L:C2}.

\begin{proof}[Proof of Lemma~\ref{L:C2}]
We first address the case $n=0$.
Then $$\pi_j^\varepsilon B_{0,m}(f,\ldots,f)[ h]-B_{0,0}[\pi_j^\varepsilon h]=T_a+T_b+T_c,$$ where
\begin{align*}
T_{a}&:=\pi_j^\varepsilon B_{0,m}(f,\ldots,f)[ h]- B_{0,m}(f,\ldots,f)[\pi_j^\varepsilon h],\\[1ex]
T_{b}&:=  \chi_j^\varepsilon B_{0,0}[\pi_j^\varepsilon h]- B_{0,0}[\chi_j^\varepsilon(\pi_j^\varepsilon h)]-( \chi_j^\varepsilon B_{0,m}(f,\ldots,f)[\pi_j^\varepsilon h]-B_{0,m}(f,\ldots,f)[\chi_j^\varepsilon(\pi_j^\varepsilon h)]),\\[1ex]
T_{c}&:=  \chi_j^\varepsilon (B_{0,m}(f,\ldots,f)[\pi_j^\varepsilon h]-  B_{0,0}[\pi_j^\varepsilon h]).
\end{align*}
Lemma \ref{L:B2} yields
\begin{equation}\label{C2a''}
\|T_a\|_{W^{s-1}_p}+\|T_b\|_{W^{s-1}_p}\leq K\|h\|_{p}.
\end{equation}
It remains to estimate   
\begin{align*}
T_{c}&= - \sum_{k=0}^{m-1} \chi_j^\varepsilon B_{2,m-k}(f,\ldots,f)[f, f,\pi_j^\varepsilon h].
\end{align*} 
Using Lemma \ref{L:MP1}~(i), we get
\begin{align}\label{2EEE}
\|T_c\|_p\leq K\|h\|_{p}.
\end{align}
Let $T_k:=\chi_j^\varepsilon B_{2,m-k}(f,\ldots,f)[f, f,\pi_j^\varepsilon h]$, $0\leq k \leq m-1$. 
To estimate the ${W^{s-1}_p}$-seminorm of~$T_k$ we write for $\xi\in\mathbb{R}$ 
\begin{align*}
T_k-\tau_\xi T_k=T_{kA}+T_{kB}+\chi_j^\varepsilon T_{kC},
\end{align*}
where,  using (\ref{spr3}), we get
\begin{align*}
T_{kA}&:=(\chi_j^\varepsilon-\tau_\xi \chi_j^\varepsilon)\tau_\xi B_{2,m-k}(f,\ldots,f)[f, f,\pi_j^\varepsilon h],\\[1ex]
T_{kB}&:=  \chi_j^\varepsilon B_{2,m-k}(f,\ldots,f)[f, f, \pi_j^\varepsilon h -\tau_\xi (\pi_j^\varepsilon h)],\\[1ex]
T_{kC}&:=  B_{2,m-k}(f,\ldots,f) [ f-\tau_\xi f,f ,\tau_\xi (\pi_j^\varepsilon h')]+B_{2,m-k}(f,\ldots,f)[  \tau_\xi f,f-\tau_\xi f,\tau_\xi (\pi_j^\varepsilon h)] \\[1ex]
&\hspace{0,45cm}+\sum_{\ell=1}^{m-k}B_{4,m-k+1}(\underset{\ell-{\rm times}}{\underbrace{f,\ldots,f}},\tau_\xi f,\ldots,\tau_\xi f)
[\tau_\xi f, \tau_\xi f,  \tau_\xi f+f,\tau_\xi f-f,\tau_\xi (\pi_j^\varepsilon h)].
\end{align*}
Lemma \ref{L:MP1}~(ii) (with $r=s$ and $\tau=s'-1$) yields
\begin{equation}\label{2E1'}
\|T_{kA}\|_p\leq K \|\chi_j^\varepsilon-\tau_\xi\chi_j^\varepsilon\|_p\|\pi_j^\varepsilon h\|_{W^{s'-1}_p}\leq K\|(1-\chi_j^\varepsilon)-\tau_\xi(1-\chi_j^\varepsilon)\|_p \| h\|_{W^{s'-1}_p}.
\end{equation}
Let $F$ denote the Lipschitz continuous function  defined by $F=f$ on $[|x|\geq 1/\varepsilon-\varepsilon]$ and which is linear in $[|x|\le 1/\varepsilon-\varepsilon]$.
If $|\xi|\geq\varepsilon,$ we infer from Lemma \ref{L:MP1}~(i)  that 
\begin{align}\label{2E21'}
\|T_{kB}\|_p\leq   K\|h \|_{p}.
\end{align}
If $|\xi|<\varepsilon$, then $\xi+ \supp\pi_j^\varepsilon\subset \supp\chi_j^\varepsilon$, and using Lemma \ref{L:MP1}~(i)   we get
\begin{equation}\label{2E22'}
\begin{aligned}
\|T_{kB}\|_p&=  \|\chi_j^\varepsilon B_{2,m-k}(f,\ldots,f)[F, F, \pi_j^\varepsilon h -\tau_\xi (\pi_j^\varepsilon h)]\|_p\\[1ex]
&\leq C\| F'\|_{\infty}^2\|\pi_j^\varepsilon h-\tau_\xi (\pi_j^\varepsilon h)\|_p\\[1ex]
&\leq \frac{\nu}{3( m+1)  }\|\pi_j^\varepsilon h-\tau_\xi (\pi_j^\varepsilon h)\|_p,
\end{aligned}
\end{equation}
provided that $\varepsilon$ is sufficiently small. 
The arguments in (\ref{2E22'}) rely on the fact that  ${ \|F'\|_{\infty}  \to 0}$ for $\varepsilon\to0$.
Finally, Lemma \ref{L:MP2} (with $r=s'$) yields
\begin{equation}\label{2E2'}
\|\chi_j^\varepsilon T_{kC}\|_p\leq K\|f'-\tau_\xi f'\|_p\| \pi_j^\varepsilon h \|_{W^{s'-1}_p}\leq K\|f'-\tau_\xi f'\|_p\|h\|_{W^{s'-1}_p}.
\end{equation}
The estimates (\ref{C2a''})-(\ref{2E2'})  lead us to (\ref{LC1Na}).

The estimate (\ref{LC1Nb}) can be derived by using the same arguments as above. 
\end{proof}\bigskip

\section*{Acknowledgement}
We are grateful to the anonymous referees for the helpful comments on an earlier version of this contribution.
Helmut Abels and Bogdan-Vasile Matioc were partially supported by the RTG 2339 ''Interfaces, Complex Structures, and Singular Limits''
of the German Science Foundation (DFG). The support is gratefully acknowledged.

\bibliographystyle{siam}
\bibliography{AbMa}

\begin{thebibliography}{10}

\bibitem{AB12}
{\sc H.~Abels}, {\em {Pseudodifferential and Singular Integral Operators}}, De
  Gruyter Graduate Lectures, De Gruyter, Berlin, 2012.
\newblock An introduction with applications.

\bibitem{AL20}
{\sc T.~Alazard and O.~Lazar}, {\em Paralinearization of the {M}uskat equation
  and application to the {C}auchy problem}, Arch. Ration. Mech. Anal., 237
  (2020), pp.~545--583.

\bibitem{Am95}
{\sc H.~Amann}, {\em {Linear and Quasilinear Parabolic Problems. {V}ol. {I}}},
  vol.~89 of {Monographs in Mathematics}, Birkh\"auser Boston, Inc., Boston,
  MA, 1995.
\newblock Abstract linear theory.

\bibitem{A04}
{\sc D.~M. Ambrose}, {\em {Well-posedness of two-phase {H}ele-{S}haw flow
  without surface tension}}, European J. Appl. Math., 15 (2004), pp.~597--607.

\bibitem{A14}
\leavevmode\vrule height 2pt depth -1.6pt width 23pt, {\em {The zero surface
  tension limit of two-dimensional interfacial {D}arcy flow}}, J. Math. Fluid
  Mech., 16 (2014), pp.~105--143.

\bibitem{An90}
{\sc S.~B. Angenent}, {\em {Nonlinear analytic semiflows}}, Proc. Roy. Soc.
  Edinburgh Sect. A, 115 (1990), pp.~91--107.

\bibitem{BV14}
{\sc B.~V. Bazaliy and N.~Vasylyeva}, {\em {The two-phase {H}ele-{S}haw problem
  with a nonregular initial interface and without surface tension}}, Zh. Mat.
  Fiz. Anal. Geom., 10 (2014), pp.~3--43, 152, 155.

\bibitem{BCG14}
{\sc L.~C. Berselli, D.~C\'ordoba, and R.~Granero-Belinch\'on}, {\em {Local
  solvability and turning for the inhomogeneous {M}uskat problem}}, Interfaces
  Free Bound., 16 (2014), pp.~175--213.

\bibitem{Cam19}
{\sc S.~Cameron}, {\em Global well-posedness for the two-dimensional {M}uskat
  problem with slope less than 1}, Anal. PDE, 12 (2019), pp.~997--1022.

\bibitem{CCFG13}
{\sc A.~Castro, D.~C\'ordoba, C.~Fefferman, and F.~Gancedo}, {\em Breakdown of
  smoothness for the {M}uskat problem}, Arch. Ration. Mech. Anal., 208 (2013),
  pp.~805--909.

\bibitem{CGFL11}
{\sc A.~Castro, D.~C\'ordoba, C.~L. Fefferman, F.~Gancedo, and
  M.~L\'opez-Fern\'andez}, {\em {Turning waves and breakdown for incompressible
  flows}}, Proc. Natl. Acad. Sci. USA, 108 (2011), pp.~4754--4759.

\bibitem{CCFGL12}
\leavevmode\vrule height 2pt depth -1.6pt width 23pt, {\em Rayleigh-{T}aylor
  breakdown for the {M}uskat problem with applications to water waves}, Ann. of
  Math. (2), 175 (2012), pp.~909--948.

\bibitem{BCS16}
{\sc C.~H.~A. Cheng, R.~Granero-Belinch\'on, and S.~Shkoller}, {\em
  {Well-posedness of the {M}uskat problem with {$H^2$} initial data}}, Adv.
  Math., 286 (2016), pp.~32--104.

\bibitem{CCGS16}
{\sc P.~Constantin, D.~C\'ordoba, F.~Gancedo, L.~Rodr\'guez-Piazza, and R.~M.
  Strain}, {\em {On the Muskat problem: Global in time results in 2D and 3D}},
  Amer. J. Math., 138 (2016), pp.~1455--1494.

\bibitem{CCGS13}
{\sc P.~Constantin, D.~C\'ordoba, F.~Gancedo, and R.~M. Strain}, {\em {On the
  global existence for the {M}uskat problem}}, J. Eur. Math. Soc. (JEMS), 15
  (2013), pp.~201--227.

\bibitem{CGSV17}
{\sc P.~Constantin, F.~Gancedo, R.~Shvydkoy, and V.~Vicol}, {\em Global
  regularity for 2{D} {M}uskat equations with finite slope}, Ann. Inst. H.
  Poincar\'e Anal. Non Lin\'eaire, 34 (2017), pp.~1041--1074.

\bibitem{CCG11}
{\sc A.~C\'ordoba, D.~C\'ordoba, and F.~Gancedo}, {\em {Interface evolution:
  the {H}ele-{S}haw and {M}uskat problems}}, Ann. of Math. (2), 173 (2011),
  pp.~477--542.

\bibitem{CG07}
{\sc D.~C\'ordoba and F.~Gancedo}, {\em {Contour dynamics of incompressible
  3-{D} fluids in a porous medium with different densities}}, Comm. Math.
  Phys., 273 (2007), pp.~445--471.

\bibitem{CG10}
\leavevmode\vrule height 2pt depth -1.6pt width 23pt, {\em {Absence of squirt
  singularities for the multi-phase {M}uskat problem}}, Comm. Math. Phys., 299
  (2010), pp.~561--575.

\bibitem{CGO14}
{\sc D.~{C\'ordoba Gazolaz}, R.~Granero-Belinch\'on, and R.~Orive-Illera}, {\em
  {The confined {M}uskat problem: differences with the deep water regime}},
  Commun. Math. Sci., 12 (2014), pp.~423--455.

\bibitem{DLL17}
{\sc F.~Deng, Z.~Lei, and F.~Lin}, {\em {On the two-dimensional Muskat problem
  with monotone large initial data}}, Comm. Pure Appl. Math., LXX (2017),
  pp.~1115--1145.

\bibitem{EMM12a}
{\sc J.~Escher, A.-V. Matioc, and B.-V. Matioc}, {\em {A generalized
  {R}ayleigh-{T}aylor condition for the {M}uskat problem}}, Nonlinearity, 25
  (2012), pp.~73--92.

\bibitem{EM11a}
{\sc J.~Escher and B.-V. Matioc}, {\em {On the parabolicity of the {M}uskat
  problem: well-posedness, fingering, and stability results}}, Z. Anal.
  Anwend., 30 (2011), pp.~193--218.

\bibitem{EMW18}
{\sc J.~Escher, B.-V. Matioc, and C.~Walker}, {\em {The domain of parabolicity
  for the Muskat problem}}, Indiana Univ. Math. J., 67 (2018), pp.~679--737.

\bibitem{ES96}
{\sc J.~Escher and G.~Simonett}, {\em {Analyticity of the interface in a free
  boundary problem}}, Math. Ann., 305 (1996), pp.~439--459.

\bibitem{NF20x}
{\sc P.~T. Flynn and H.~Q. Nguyen}, {\em {The vanishing surface tension limit
  of the Muskat problem}},  (2020).
\newblock arXiv:2001.10473.

\bibitem{G17}
{\sc F.~Gancedo}, {\em A survey for the {M}uskat problem and a new estimate},
  SeMA J., 74 (2017), pp.~21--35.

\bibitem{GGPS19}
{\sc F.~Gancedo, E.~Garc\'ia-Ju\'arez, N.~Patel, and R.~M. Strain}, {\em On the
  {M}uskat problem with viscosity jump: global in time results}, Adv. Math.,
  345 (2019), pp.~552--597.

\bibitem{GGS20}
{\sc F.~Gancedo, R.~Granero-Belinchón, and S.~Scrobogna}, {\em Surface tension
  stabilization of the rayleigh-taylor instability for a fluid layer in a
  porous medium}, Annales de l'Institut Henri Poincaré C, Analyse non
  linéaire, 37 (2020), pp.~1299 -- 1343.

\bibitem{GG14}
{\sc J.~G\'omez-Serrano and R.~Granero-Belinch\'on}, {\em {On turning waves for
  the inhomogeneous {M}uskat problem: a computer-assisted proof}},
  Nonlinearity, 27 (2014), pp.~1471--1498.

\bibitem{GB14}
{\sc R.~Granero-Belinch\'on}, {\em {Global existence for the confined {M}uskat
  problem}}, SIAM J. Math. Anal., 46 (2014), pp.~1651--1680.

\bibitem{GL20}
{\sc R.~Granero-Belinch\'{o}n and O.~Lazar}, {\em Growth in the {M}uskat
  problem}, Math. Model. Nat. Phenom., 15 (2020), pp.~Paper No. 7, 23.

\bibitem{GS19}
{\sc R.~Granero-Belinch\'on and S.~Shkoller}, {\em Well-posedness and decay to
  equilibrium for the {M}uskat problem with discontinuous permeability}, Trans.
  Amer. Math. Soc., 372 (2019), pp.~2255--2286.

\bibitem{L95}
{\sc A.~Lunardi}, {\em {Analytic Semigroups and Optimal Regularity in Parabolic
  Problems}}, {Progress in Nonlinear Differential Equations and their
  Applications, 16}, Birkh\"auser Verlag, Basel, 1995.

\bibitem{Matioc2019}
{\sc A.-V. Matioc and B.-V. Matioc}, {\em Well-posedness and stability results
  for a quasilinear periodic {Muskat} problem}, J. Differential Equations, 266
  (2019), pp.~5500-- 5531.

\bibitem{MBV18}
{\sc B.-V. Matioc}, {\em {Viscous displacement in porous media: the Muskat
  problem in 2D}}, Trans. Amer. Math. Soc., 370 (2018), pp.~7511--7556.

\bibitem{MBV19}
\leavevmode\vrule height 2pt depth -1.6pt width 23pt, {\em {The Muskat problem
  in two dimensions: equivalence of formulations, well-posedness, and
  regularity results}}, Anal. PDE, 12 (2019), pp.~281--332.

\bibitem{MBV20}
\leavevmode\vrule height 2pt depth -1.6pt width 23pt, {\em Well-posedness and
  stability results for some periodic {M}uskat problems}, J. Math. Fluid Mech.,
  22 (2020), pp.~Art. 31, 45.

\bibitem{MW20}
{\sc B.-V. Matioc and C.~Walker}, {\em On the principle of linearized stability
  in interpolation spaces for quasilinear evolution equations}, Monatshefte
  f\"ur Mathematik, 191 (2020), pp.~615--634.

\bibitem{CM97}
{\sc Y.~Meyer and R.~Coifman}, {\em {Wavelets}}, vol.~48 of {Cambridge Studies
  in Advanced Mathematics}, Cambridge University Press, Cambridge, 1997.
\newblock Calder\'on-Zygmund and multilinear operators, Translated from the
  1990 and 1991 French originals by David Salinger.

\bibitem{TM86b}
{\sc T.~Murai}, {\em Boundedness of singular integral operators of
  {C}alder\'{o}n type. {V}}, Adv. in Math., 59 (1986), pp.~71--81.

\bibitem{TM86}
\leavevmode\vrule height 2pt depth -1.6pt width 23pt, {\em {Boundedness of
  singular integral operators of {C}alder\'on type. {VI}}}, Nagoya Math. J.,
  102 (1986), pp.~127--133.

\bibitem{Mu34}
{\sc M.~Muskat}, {\em {Two fluid systems in porous media. The encroachment of
  water into an oil sand}}, Physics, 5 (1934), pp.~250--264.

\bibitem{Ngu20}
{\sc H.~Q. Nguyen}, {\em On well-posedness of the {M}uskat problem with surface
  tension}, Adv. Math., 374 (2020), p.~107344.

\bibitem{NP20}
{\sc H.~Q. Nguyen and B.~Pausader}, {\em A paradifferential approach for
  well-posedness of the {M}uskat problem}, Arch. Ration. Mech. Anal., 237
  (2020), pp.~35--100.

\bibitem{PS17}
{\sc N.~Patel and R.~M. Strain}, {\em Large time decay estimates for the
  {M}uskat equation}, Comm. Partial Differential Equations, 42 (2017),
  pp.~977--999.

\bibitem{PSS15}
{\sc J.~Pr\"uss, Y.~Shao, and G.~Simonett}, {\em {On the regularity of the
  interface of a thermodynamically consistent two-phase {S}tefan problem with
  surface tension}}, Interfaces Free Bound., 17 (2015), pp.~555--600.

\bibitem{PS16}
{\sc J.~Pr\"uss and G.~Simonett}, {\em Moving Interfaces and Quasilinear
  Parabolic Evolution Equations}, vol.~105 of Monographs in Mathematics,
  Birkh\"auser/Springer, [Cham], 2016.

\bibitem{PS16x}
\leavevmode\vrule height 2pt depth -1.6pt width 23pt, {\em {On the Muskat
  flow}}, Evol. Equ. Control Theory, 5 (2016), pp.~631--645.

\bibitem{SCH04}
{\sc M.~Siegel, R.~E. Caflisch, and S.~Howison}, {\em {Global existence,
  singular solutions, and ill-posedness for the {M}uskat problem}}, Comm. Pure
  Appl. Math., 57 (2004), pp.~1374--1411.

\bibitem{To17}
{\sc S.~Tofts}, {\em On the existence of solutions to the {M}uskat problem with
  surface tension}, J. Math. Fluid Mech., 19 (2017), pp.~581--611.

\bibitem{Y96}
{\sc F.~Yi}, {\em {Local classical solution of {M}uskat free boundary
  problem}}, J. Partial Differential Equations, 9 (1996), pp.~84--96.

\end{thebibliography}
\end{document}